 \pdfoutput=1
\documentclass[titlepage,a4paper,12pt]{amsart} 
\usepackage[foot]{amsaddr}
\usepackage{amssymb}
\usepackage{t1enc}
\usepackage[latin1]{inputenc}
\usepackage[english]{babel}

\usepackage[dvipdfmx]{graphicx}
\usepackage{bmpsize,caption,wrapfig}

\usepackage{mathrsfs,xcolor} 

\usepackage{amsmath,pslatex}

\usepackage{hyperref}
\usepackage[left=3cm,top=3cm,right=3cm,bottom=3cm,dvips]{geometry}
\numberwithin{equation}{section}

\newtheorem{thm}{Theorem}[section]
\newtheorem{prop}[thm]{Proposition}
\newtheorem{lem}[thm]{Lemma}
\newtheorem{cor}[thm]{Corollary}
\theoremstyle{definition}
\newtheorem{defn}[thm]{Definition}
\newtheorem{rem}[thm]{Remark}
\newtheorem{ex}[thm]{Example}

\renewcommand{\tilde}{\widetilde}

\newcommand{\bone}{\mathbf 1}

\DeclareMathOperator\ext{ext}
\DeclareMathOperator\cent{cent}
\DeclareMathOperator\length{length}
\DeclareMathOperator\dist{dist}
\DeclareMathOperator\conv{conv}

\DeclareMathOperator\E{E}
\DeclareMathOperator\argmin{argmin}
\DeclareMathOperator\argmax{argmax}

\DeclareMathOperator\Ad{Ad}
\DeclareMathOperator\ad{ad}
\DeclareMathOperator\id{id}

\DeclareMathOperator\diag{diag}

\DeclareMathOperator\tr{tr}

\DeclareMathOperator\Stab{Stab}
\DeclareMathOperator\Symp{Symp}
\DeclareMathOperator\Ham{Ham}
\DeclareMathOperator\PU{PU}

\DeclareMathOperator\SU{SU}
\DeclareMathOperator\SO{SO}

\DeclareMathOperator\bd{bd}

\def\X{\mathcal X}
\def\O{\mathcal O}

\def\W{\mathcal W}

\def\su{{\mathfrak su}}

\def\k{{\mathfrak k}}

\def\t{{\mathfrak h}}

\def\z{{\mathfrak z}}

\def\g{{\mathfrak g}}
\def\p{{\mathfrak p}}

\def\C{\mathbb{C}}
\def\R{\mathbb{R}}

\def\Z{\mathbb{Z}}

\begin{document}

\title{Hofer's metric in compact Lie groups}
\author{Gabriel Larotonda} 
\email{glaroton@dm.uba.ar}
\address[G. Larotonda]{Departamento de Matem\'atica, FCEyN-UBA, and Instituto Argentino de Matem\'atica, CONICET. Buenos Aires, Argentina.}
\author{Mart\'in Miglioli}
\email{martin.miglioli@gmail.com}
\address[M. Miglioli]{Instituto Argentino de Matem\'atica, CONICET. Buenos Aires, Argentina.}
\thanks{The authors were supported by grant PICT-2019-040602010 (ANPCyT)}


\begin{abstract}
In this article we study the Hofer geometry of a compact Lie group $K$ which acts by Hamiltonian diffeomorphisms on a symplectic manifold $M$. Generalized Hofer norms on the Lie algebra of $K$ are introduced and analyzed with tools from group invariant convex geometry, functional and matrix analysis. Several global results on the existence of geodesics and their characterization in finite dimensional Lie groups $K$ endowed with bi-invariant  Finsler metrics are proved. We relate the conditions for being a geodesic in the group $K$ and in the group of Hamiltonian diffeomorphisms. These results are applied to obtain necessary and sufficient conditions on the moment polytope of the momentum map, for the commutativity of the Hamiltonians of geodesics. Particular cases are studied, where a generalized non-crossing of eigenvalues property of the Hamiltonians hold. 
\end{abstract}

\maketitle

\tableofcontents

\section{Introduction}

The subject of this paper is the metric geometry of compact Lie groups $K$: we are interested in the geometry of such a group when it is provided with a Finsler bi-invariant metric, not necessarily smooth neither strictly convex. The rectifiable distance in $K$ is defined as the infimum of the lengths of paths joining given endpoints, and a \textit{geodesic} in $K$ is a distance minimising path (we also use \textit{short}). Since the unit sphere can have faces and corners, there is a zoo of short paths for the distance  in $K$ besides the one-parameter groups. We want to give a full characterization of these geodesics and relate it with other geometrical invariants of the group and its Lie algebra.  Let us recall here that a Finsler length structure on the group $\Ham(M,\omega)$ of Hamiltonian diffeomorphisms of a symplectic manifold $M$ with symplectic form $\omega$, was introduced by Hofer in the paper \cite{hofer}. The Lie algebra of this group is the set of Hamiltonian vector fields, which can be identified (by means of the symplectic gradient) with the set of Hamiltonian functions in $M$ (modulo constant functions). The norm of a vector field is then the quotient  $L^\infty(M)$-norm of the generating Hamiltonian function (modulo constant functions). A natural problem of current interest in the literature is the study of geodesics in this Finsler manifold. There has been a significant amount of progress, and fairly deep work on the properties of this metric, mostly from the point of view of symplectic topology, see \cite{bp,lm95,lm,pol01,ps} and also the textbook \cite{mds} and the references therein. The results of this paper are for finite dimensional groups; there is however an interesting relation among them and the Finsler structure of $(M,\omega)$: we will consider (almost) effective Hamiltonian actions of compact semi-simple Lie groups $K$ on a symplectic manifold $(M,\omega)$. This action defines an inclusion (modulo the discrete kernel of the action) $K\hookrightarrow \Ham(M,\omega)$, which allows us to introduce a pull-back Hofer metric on $K$ by means of
$$
 \|x\|_{\mu(M)}=\max_{y\in\mu(M)}\langle y,x\rangle-\min_{y\in\mu(M)}\langle y,x\rangle.
$$
Here $x\in \k=Lie(K)=T_1K$ and $\mu$ is the momentum map of the Hamiltonian action. For this norm, the intersection of its unit ball with a maximal abelian subalgebra $\t\subseteq \mathfrak k$ is the polar dual of a certain polytope $P$, which is derived from the moment polytope of the action $\mu$.  A main example of this situation is when $M=\conv(\O_w)$, the convex closure of the adjoint orbit of $w$ in $\k$, with trivial momentum map and the adjoint action of $K$. 

\smallskip

The results on geodesics in this article are stated for any Finsler length structure given by (left or right) translation of an $\Ad$-invariant Finsler norm in $\k$, therefore we include non-symmetric distances in our discussion; what follows in one of the main results of this paper:

{\textbf{Theorem A}: }\textit{Let $K$ be a compact  semi-simple Lie group with a bi-invariant Finsler metric, let $\gamma:[a,b]\to K$ be a piecewise $C^1$ path.  If $\gamma$ is short for the bi-invariant metric, then for (almost) all $t$ we have $\varphi(\gamma_t^{-1}\dot{\gamma}_t)=\|\gamma_t^{-1}\dot{\gamma}_t\|$ for some unit norm functional $\varphi\in \k^*$. Reciprocally, if the equality holds for some $\varphi$ and (almost) all $t\in [t_0,t_1]$, and   $L(\gamma)_{t_0}^{t_1}\le R$, then $\gamma$ is short in $[t_0,t_1]\subseteq [a,b]$}.

\smallskip

Here $R$ is the injectivity radius of the norm. In particular one-parameter groups in $K$ are always geodesics for these distances, provided their speed is in the domain of injectivity of the exponential map of $K$. It is worthwhile mentioning here that the notion of majorization of real vectors $\overrightarrow{v}\prec \overrightarrow{w}$ (identified with the eigenvalues of the operators $\ad v, \ad w$, where $v,w\in \k$) plays a significant role in the proofs concerning minimality of geodesics, and it is related to the condition $v\in \conv(\O_w)$, where the later set is the convex closure of the coadjoint orbit of $w$ in $\k$. If we specialize the previous result for a Hofer norm derived from a Hamiltonian almost effective action $K\curvearrowright (M,\omega)$, one-parameter groups in $K$ are in correspondence to paths in $\Ham(M,\omega)$ with autonomous Hamiltonian, and we obtain the following:

\smallskip

\textbf{Theorem B:} \textit{ Let $\gamma:[a,b]\to K$ be piecewise $C^1$, and denote its right logarithmic derivative by $x_t=\dot{\gamma}_t\gamma^{-1}_t$. If $\gamma$ is a short path in $K$, then $(\mu^{x_t})_{t\in [0,1]}$ is a quasi-autonomous Hamiltonian path, and if $\mu$ is quasi-autonomous, $\gamma$ is locally short (in each interval of length $\le R$)}.

\smallskip

Then, in the final part of the paper, we move on to a finer characterization of geodesics for some special norms, and we obtain several sharper results.  Relevant geometrical properties of the geodesics can be expressed in terms of the extreme points of Hofer's polytope $P$. Of particular relevance are the polytopes with only regular extreme points in $\k$, which are fully characterized both in terms of the Lie algebra (by polar duality) and in terms of the geometry of geodesics in $K$:

\smallskip

{\textbf{Theorem C}: }\textit{Let $B$ be an $\Ad$-invariant convex body in $\k$ containing $0$,  such that $P=(B\cap\t)^\circ$ is a polytope (here $\t$ is any Cartan subalgebra). Endow $K$ with the Finsler length structure corresponding to the Minkowski norm of $B$. Then, all the extreme points of $P$ are regular if and only if all short curves $\gamma$ in $K$ have commuting logarithmic derivatives}.

We obtain similar results for the pull-back of the one-sided Hofer norm, which is usually only positively homogeneous.

\smallskip

There are several relevant applications related to this setting of actions of compact Lie groups: for instance, as shown in \cite{entov}, the geometry of the canonical Hamiltonian action $\SU(n)\to \Ham(Gr_{r,n},\omega)$ can be used as a tool to study the eigenvalue inequalities in the quantum version of Horn's problem  \cite{belkale}.  Here $Gr_{r,n}$ is the Grassmannian of $r$-dimensional planes in $\C^n$, and $\omega$ is the canonical Kirilov-Kostant-Souriau symplectic form. We plan to extend some of the results in this article to the case of infinite dimensional groups using ideas connected to the results in \cite{blz} and \cite{lar19}. Some of the techniques developed in this article might also be relevant to the study of Finsler length structures derived from mechanics with non-smooth energy.

\smallskip

The article is organized as follows: in Section \ref{sectionnorms} we define generalized Hofer norms. We study the faces and norming functionals of the unit balls of these norms, based on two different theories. We first analyse the structure of these balls with functional analytic techniques via an embedding in certain function spaces, an approach that will be useful in the study of the stability under geodesy at the end of Section \ref{sectionquasiaut}. We also study these norms using results from convex geometry.

In Section \ref{hofergroup} we recall basic results on Hamiltonian actions and Hofer's metric on groups of Hamiltonian diffeomorphisms, and we pull-back these metrics to compact groups using the homomorphism $K\to \Ham(M,\omega)$. These are the motivations and main examples for the norms and length structures on groups studied in this article. Nevertheless, this part of symplectic geometry is not necessary for the understanding of several results in the article which are solely based on convexity and Finsler length structures. 

In Section \ref{convexgeocart} a characterization of the intersection of the unit balls of the Hofer norms with maximal abelian algebras is given based on group invariant convex analysis and symplectic convexity theorems.

In Section \ref{l} we first recall several results obtained in \cite{lar19} for groups endowed with Finsler length structures obtained from $\Ad$-invariant norms which are valid for the groups studied in this article. Then we prove several global results on geodesics in the case of finite dimensional groups endowed with continuous Finsler metrics. We show that geodesics in groups $K$ with Hofer's metric are quasi-autonomous, which provides a link between the conditions for length minimization in $K$ and the corresponding conditions in $\Ham(M,\omega)$. 

Finally, in Section \ref{geodesicommute} we study actions of groups with commuting Hamiltonians. We start with the important special case of actions on regular coadjoint orbits and related groups: the Hamiltonians of length minimizing curves have the interesting feature of "non-crossing of eigenvalues". We characterize the $\Ad$-invariant norms such that in groups with Finsler structures defined from these norms, all geodesics have commuting speeds. Based on this result we characterize the compact groups of Hamiltonian diffeomorphisms such that length minimizing curves have commuting Hamiltonians. We then show how conditions on Kirwan's polytope can characterize this property. The paper ends with a study of how these properties behave when we consider the direct product of Hamiltonian actions. It is proved then that it suffices to have geodesics with commuting Hamiltonians for one of the actions, to obtain the same property for the geometry induced in $K$ by the direct product of actions.

\section{The generalized Hofer norm and its convex geometry}\label{sectionnorms}

In this section we define the generalized Hofer norms and we study them with two approaches. In the first we embed the normed space in a quotient of a space of continuous functions and use functional analytic techniques. In the second we use the polar duality from convex geometry. 

\begin{defn}\label{convexb}
A subset $E\subseteq V$ of a vector space $V$ is called \textit{full} if it affinely generates the space. A set $B\subseteq V$ is a \textit{convex body} if $B$ is a compact convex  set with non empty interior. If additionally, $B$ is centrally symmetric ($v\in B\Rightarrow -v\in B$), then it is called a \textit{symmetric convex body}. Equivalently, a symmetric convex body is a convex balanced absorbing set in $V$.
\end{defn}

\begin{defn}\label{iota}
Let $(V,\langle \cdot,\cdot\rangle)$ be a finite dimensional inner product space and let $E\subseteq V$ be a compact full subset. Consider the norm $\|\cdot\|_E$ on $V$ given by the embedding
$$
\iota:V\hookrightarrow C(E)/\R\bone, \quad x\mapsto [\varphi_x]:=\varphi_x+\R\bone,$$
where $\varphi_x(y)=\langle x,y\rangle$ for $y\in E$. Then
\begin{equation}\label{generhofernorm}
\|x\|_E=\max_{y\in E}\varphi_x(y)-\min_{y\in E}\varphi_x(y)=2\|[\varphi_x]\|_{\infty},
\end{equation}
where $\|[\varphi_x]\|_{\infty}=\inf\{\|\varphi_x-\lambda \bone \|_{\infty} :\lambda\in\mathbb R\}$ is the quotient norm. We call $\|\cdot\|_E$ a \textit{generalized Hofer norm}.
\end{defn}

This is a norm since $E$ is full. If there is a group acting isometrically on $V$ and leaving the set $E$ invariant then the action is also isometric for the norm $\|\cdot\|_E$.  

\begin{figure}[h]
\def\svgwidth{10cm}
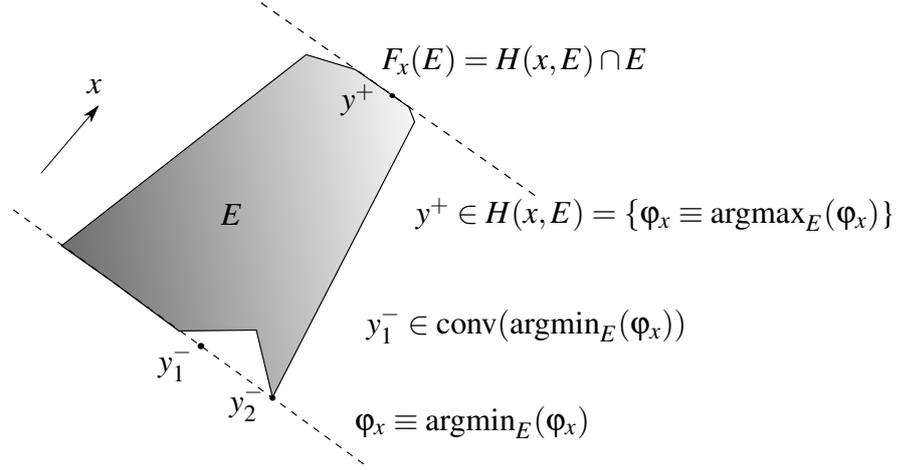
\caption{\small{Norming functionals, faces and supporting hyperplanes}}
\label{fig: norming}
\end{figure}

\begin{rem}\label{otrasnorm}
Another (semi)-norm which is invariant for isometric actions on $V$ that we will consider is the \textit{generalized second Hofer norm} given by the supremum norm
\begin{equation}\label{generhofernorm2}
\|x\|'_E= \max\{\max_{y\in E}\varphi_x(y),-\min_{y\in E}\varphi_x(y)\}=\|\varphi_x\|_{\infty}.
\end{equation}
It is also of interest to consider a third invariant norm which is only a Finsler norm (i.e. only positively homogeneous), given by
$$
\|x\|^+_E=\max_{y\in E}\varphi_x(y).
$$
We will call this the \textit{one-sided Hofer norm}. This defines a Finsler norm for a bounded set $E\subseteq V$ such that its convex hull contains $0$ in its interior, or equivalently, such that the cone generated by $E$ is all of $V$. It is also possible to consider
$$
\|x\|^-_E=-\min_{y\in E}\varphi_x(y).
$$
but note that this one can be obtained from the previous by replacing $E$ with $-E$. Clearly $\|x\|'_E=\max \{\|x\|_E^+,\|x\|_E^-\}$ is the second Hofer norm and $\|x\|_E=\|x\|_E^++\|x\|_E^-$ is the first Hofer norm.
\end{rem}


\subsection{Maximal faces and norming functionals}

We begin studying the faces of the sphere and the norming functionals of Hofer's norm by relating this norm on the space $V$ to the norm on the much larger space $C(E)/\R\bone$ (where $E$ is a compact full set) by means of the embedding $\iota:V\hookrightarrow C(E)/\R\bone$ given by $x\mapsto \varphi_x+\R\bone$, where $\varphi_x=\langle x,\cdot\rangle$.

\begin{defn}
Let $V$ be a normed space and denote $\|\varphi\|=\sup\{\varphi(v):\|v\|=1\}$ for $\varphi\in V^*$. The dual space $V^*$ with this norm is a Banach space. We say that $\varphi\in V^*$ is a \textit{norming functional} of $v\in V$ if $\varphi(v)=\|v\|$ and $\|\varphi\|=1$. A functional $\varphi$ is \textit{extremal} if $\varphi$ is an extreme point of the unit ball $B_{V^*}$. If $\|\cdot\|$ is only positively homogeneous (to remark it we say that it is a \textit{Finsler norm}) the same definitions apply, and the norm given to $V^*$ is only positively homogeneous. In any case we refer to it as the \textit{dual norm}.
\end{defn}

\begin{rem}
Note that the difference between the ball of a norm and that of a Finsler norm is that the last one might not be balanced (i.e. symmetric). In both cases, it is an absorbing, open convex set containing $0\in V$.
\end{rem}

\begin{defn}
A \textit{face} $F$ of the unit ball $B_V$ of a normed space $V$ is the intersection of the unit ball $B_V$ with the hyperplane determined by a unit norm functional $\varphi\in V^*, \|\varphi\|=1$, i.e.
$$
F_{\varphi}=B_V\cap\{v\in V:\varphi(v)=1\}.
$$
We say that the face is \textit{maximal} if $\varphi$ is extremal. Every face is contained in a maximal face: if $\varphi$ is a unit norm functional then $\|\varphi\|=1$ and since $B_{V^*}$ is compact and convex there exists by the Krein-Milman theorem extremal functionals $\{\varphi_i\}_{i=1,\dots,n}\subseteq B_{V^*}$ such that $\varphi$ is a convex combination of the $\varphi_i$:
$$
\varphi=\sum_i \lambda_i \varphi_i,  \quad \lambda_i\ge 0,\quad \sum_i \lambda_i=1.
$$
It is then easy to check that if $\varphi(v)=\|v\|$ then $\varphi_i(v)=\|v\|$ for all $i$. Therefore if $v\in F_{\varphi}$, $v\in F_{\varphi_i}$ for all $i$ and in fact $F_{\varphi}$ is the intersection of all the maximal faces that contain it.

\medskip

The \textit{cone generated} by a face $F_{\varphi}$ is $\R_+ F_{\varphi}$. Note that this cone consists exactly of those $v\in V$ such that $\varphi(v)=\|v\|$.
\end{defn}

The following elementary characterization will be useful:
\begin{lem}\label{mismacara}
In a vector space $V$,  $\|v_1+\dots +v_n\|=\|v_1\|+\dots+\|v_n\|$ holds if and only if $v_1,\dots,v_n$ belong to the cone generated by a face.
\end{lem}
\begin{proof}
If the $v_i$ are in the cone of $\varphi$, then 
$$
\|v_1\|+\dots+\|v_n\|\ge \|v_1+\dots+v_n\|\ge \varphi(v_1+\dots +v_n)=\|v_1\|+\dots + \|v_n\|.
$$
On the other hand, if $\|v_1+\dots +v_n\|=\|v_1\|+\dots+\|v_n\|$ holds, by means of Hahn-Banach's theorem pick a unit norm functional that norms the sum of the $v_i$, that is $\varphi(\sum_i v_i)=\|\sum_i v_i\|$. Then
$$
\|\sum_i v_i\|=\varphi(\sum_i v_i)=\sum_i \varphi(v_i)\le \sum_i \|v_i\|=\|\sum_i v_i\|
$$
since $\varphi(v_i)\le \|v_i\|$ for all $i$; this is only possible if equality holds for each $i$. Therefore $\varphi$ is a norming functional for all the $v_i$.
\end{proof}

\subsubsection{Norming functionals as Borel measures}

Let $X$ be a compact Hausdorff  topological space and let $C(X)$ be the continuous real valued functions on $X$. Endow $C(X)/\R\bone$ with (twice) the quotient $L^\infty$ norm (the factor $2$ is there to be consistent with formula (\ref{generhofernorm}) of the definition of generalized Hofer norm). By Riesz-Markov's theorem, its dual space can be identified with the regular finite Borel signed measures in $X$ such that $\mu(X)=0$ (that is because the identification $\mu\mapsto \varphi_\mu$ is given by integration $\varphi_\mu(f)=\int_X fd\mu$, and we require that $\varphi_\mu(\bone)=0$). The norm of $\varphi_\mu$ is given by the total variation of $\mu$, therefore unit norm functionals are characterized by having total variation equal to \textit{two}. The following characterization of norming functionals follows:

\begin{rem}\label{normingcont}
For $f\in C(X)/\R\bone$ different from zero, its norming functionals are given by $\varphi=\mu^+-\mu^-$, where $\mu^+$ and $\mu^-$ are probability measures in $X$,  supported in $\argmax(f)$ and $\argmin(f)$ respectively. Since the extreme points of the probability measures are the Dirac measures, the maximal faces are given by norming functionals $\varphi=\delta^+-\delta^-$, with delta measures $\delta^+,\delta^-$ supported in $x^+,x^-\in X$ respectively.
\end{rem}

\begin{prop}\label{argcont}
A set $\{f_i\}_{i\in I}\subseteq C(X)/\R\bone$ is a subset of a cone generated by a face if and only if $\cap_{i\in I}\argmin(f_i)\neq\emptyset$ and $\cap_{i\in I}\argmax(f_i)\neq\emptyset$. For $f,g\in C(X)/\R\bone$ we have $\|f+g\|=\|f\|+\|g\|$ if and only if $\argmin(f)\cap\argmin(g)\neq\emptyset$ and $\argmax(f)\cap\argmax(g)\neq\emptyset$.

\end{prop}
\begin{proof}
If both intersections are non empty, pick $x^-,x^+$ respectively in each of them and consider $\varphi$ in the dual given by $\varphi(f)=f(x^+)-f(x^-)$. Then $\varphi$ is a unit norm functional and 
$$
\sum_i \|f_i\|\ge \|\sum_i f_i\|\ge \varphi(\sum_i f_i)=\sum_i f_i(x^+)-f_i(x^-)=\sum_i\|f_i\|,
$$
therefore by Lemma \ref{mismacara} the $f_i$ are in the cone generated by the face given by $\varphi$. Reciprocally, if there are, say, $f=f_k$ and $g=f_l$ such that the maximal argument of $f$ does not intersect the maximal argument of $g$, then $\max (f+g)<\max f +\max g$, therefore
$$
\|f+g\|=\max(f+g)-\min(f+g)<\max f+ \max g-\min f - \min g=\|f\|+\|g\|
$$
and the conclusions follows by Lemma \ref{mismacara}.
\end{proof}

Putting together the previous characterizations (and recalling that the norm we are considering is twice the quotient norm), it is clear that

\begin{cor}
The maximal faces of the ball of $C(X)/\R\bone$ are given by the sets 
$$
F_{x^-,x^+}=\{[f]:\|[f]\|_{\infty}=2,x^-\in\argmin(f),x^+\in\argmax(f)\}
$$
for a choice of points $x^-,x^+\in X$. 
\end{cor}

\medskip

Since a face of the ball of $V$ is contained in a face of $C(E)/\R\bone$ by means of the map $\iota$ of Definition \ref{iota}, we have the following result (see Figure \ref{fig: norming2}):

\begin{cor}\label{arghofer}
Let $V$ be a vector space with the norm defined by the map $\iota$ and the compact full set $E\subseteq V$. A set $S\subseteq V$ is a subset of a cone generated by a face if and only if 
$$
\bigcap_{x\in S}\argmin_E(\varphi_{x})\neq\emptyset \quad \textrm{  and }\quad \bigcap_{x\in S}\argmax_E(\varphi_{x})\neq\emptyset.
$$
\end{cor}

\begin{defn}\label{conosmaxmin}
Given a compact $E\subseteq V$ and $x^-,x^+\in E$ we define the cone
$$C_{x^-,x^+}(E):=\{x\in V:x^-\in\argmin_E(\varphi_{x})\mbox{  and  }x^+\in\argmax_E(\varphi_{x})\}.$$
\end{defn}

\begin{prop}\label{maxfacemaxmin}
Each cone generated by a maximal face is equal to $C_{x^-,x^+}(E)$ for some $x^-,x^+\in E$.
\end{prop}

\begin{proof}
Let $\R_+F_{max}$ be the cone generated by a maximal face. By Corollary \ref{arghofer} it is contained in $C_{x^-,x^+}(E)$ for some $x^-,x^+\in E$. Since $C_{x^-,x^+}(E)$ satisfies the condition of Corollary \ref{arghofer} it is contained in the cone $\R_+F$ generated by a face $F$. By maximality of $F_{max}$ we get $F_{max}=F$ and the conclusion follows. 
\end{proof}

The cones of Definition \ref{conosmaxmin} have good properties with respect to sum of sets.

\begin{prop}\label{interconosmaxmin}
Let $E_1,\dots,E_n$ be compact sets in $V$ and let $x^-_i,x^+_i\in E_i$ for $i=1,\dots,n$. If we define $E=E_1+\dots + E_n$, $x^-=x^-_1+\dots +x^-_n$ and $x^+=x^+_1+\dots +x^+_n$ then
$$C_{x^-,x^+}(E)=\bigcap_{i=1,\dots,n}C_{x^-_i,x^+_i}(E_i).$$
\end{prop}

\begin{proof}
It is easy to verify that for $x\in V$ the functional $\varphi_x$ has a maximum at  $x^-_i,x^+_i\in E_i$ in $E_i$ for $i=1,\dots,n$ if and only if it has a maximum at $x^-_1+\dots +x^-_n$ in $E_1+\dots + E_n$. The same holds for the minimizers and the proof follows.
\end{proof}

\begin{rem}\label{hofer3argmax}
Similar results can be obtained for the one-sided Hofer norm $\|\cdot\|_E^+$ (Remark \ref{otrasnorm}), for a compact set $E$ such that its convex hull contains $0$ in its interior. A set $S\subseteq V$ is a subset of a cone generated by a face if and only if 
$$
\bigcap_{x\in S}\argmax_E(\varphi_{x})\neq\emptyset.
$$
Given a compact $E\subseteq V$ and $x^+\in E$ we define the cone
$$C_{x^+}(E):=\{x\in V:x^+\in\argmax_E(\varphi_{x})\}.$$
Each cone generated by a maximal face is equal to $C_{x^+}(E)$ for some $x^+\in E$. These cones have also good properties with respect to sum of sets. Let $E_1,\dots,E_n$ be compact sets in $V$ and let $x^+_i\in E_i$ for $i=1,\dots,n$. If we define $E=E_1+\dots + E_n$ and $x^+=x^+_1+\dots +x^+_n$ then
$$C_{x^+}(E)=\bigcap_{i=1,\dots,n}C_{x^+_i}(E_i).$$
\end{rem}

We now characterize norming functionals for the Hofer norm, see Figure \ref{fig: norming} and Figure \ref{fig: norming2}. The convex hull of a set $X$ is denoted by $\conv(X)$.

\begin{thm}\label{normingfunc}
The norming functionals of $x \in (V,\|\cdot\|_E)$ are $\varphi_{y^+ -y^-}$, with 
$$
y^+ \in \conv(\argmax_E  \varphi_x)\quad\textrm{ and }\quad y^-\in\conv(\argmin_E \varphi_x).
$$ 
\end{thm}
\begin{proof}
Suppose first that $y^+ \in \conv(\argmax_E  \varphi_x)$ and $y^-\in\conv(\argmin_E \varphi_x)$, then we have $\varphi_x(y^+)=\max_E(\varphi_x)$ and likewise with $y^-$. It is immediate from the definitions that $\|\varphi_{y^+-y^-}\|\le 1$, and on the other hand 
$$
\varphi_{y^+-y^-}(x)=\langle y^+-y^-,x\rangle=\max_E(\varphi_x)-\min_E(\varphi_x)=\|x\|_E,
$$
thus $\varphi_{y^+-y^-}$ has unit norm and it is norming for $x$. Suppose now $\varphi$ is a norming functional of $x$ in $V\simeq \iota(V)\subseteq C(E)/\R\bone$. We can extend it by the Hahn-Banach theorem to all $C(E)/\R\bone$, so that it is given by integration with $\mu^+-\mu^-$ for probability measures $\mu^+,\mu^-$ supported in $\argmax(\varphi_x)$ and $\argmin(\varphi_x)$ respectively. Hence 
\begin{align*}
\varphi(z)&=\int_E\langle w,z\rangle\, d(\mu_1-\mu_2)(w)\\
&=\langle \int_E w\,d\mu_1(w) -\int_Ew\,d\mu_2(w),z\rangle\\
&=\langle \cent(\mu_1) -\cent(\mu_2),z\rangle,
\end{align*}
where $\cent(\mu)$ denotes the center of mass of the probability measure $\mu$. The result follows if we take $y^+=\cent(\mu_1)$ and $y^-=\cent(\mu_2)$.
\end{proof}

\begin{figure}[h]
\def\svgwidth{12cm}
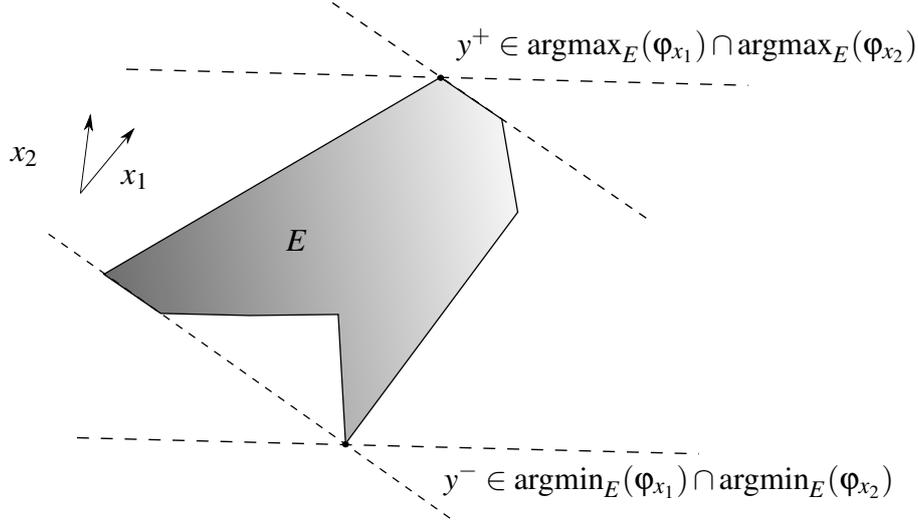
\caption{\small{$\varphi=\varphi_{y^+-y^-}$ is a norming functional of $x_1$ and $x_2$}}
\label{fig: norming2}
\end{figure}

In Theorem \ref{normingpolar} we will give another proof of the previous result, based on polar duality.

\subsection{Convex sets, polar duality and Minkowski norms}

We present some basic results on convex geometry and polar duality which will be used in this section to characterize the Hofer norms. We refer to Chapter I of \cite{bron} for basic results on convex sets and Chapter II of the same book for basic results on convex polytopes. Another reference is \cite{bar}. Let $V$ be a finite dimensional inner product space. For a non-zero vector $x\in V$ and a scalar $a\in\R$ we define the hyperplane
$$H_{x,a}=\{y\in H:\langle x,y\rangle=a\}.$$
For $a=1$ we set $H_x:=H_{x,1}$. 
We define the negative halfspace as
$$H_{x,a}^-=\{y\in H:\langle x,y\rangle\leq a\}.$$
The \textit{polar duality} is given by the following bijection between non-zero points in $V$ and hyperplanes in $V$ not containing zero: $x\mapsto H_x$. 

\begin{defn}[Supporting hyperplanes]\label{defih}
The \textit{support function} of a bounded subset $E\subseteq V$ is the function
$$
h_E:V\to\R,\quad h_E(u)=\sup_{x\in E}\langle x,u\rangle.
$$
Note that $h_E=h_{\conv(E)}$, and that if $\conv(E)$ contains $0$ in its interior, then $h_E$ is a Finsler norm, our one-sided Hofer norm (Remark \ref{otrasnorm}). 

If $0\neq x\in V$, the  hyperplane given by
$$
H(E,x):=\{v\in V:\langle v,x\rangle=h_E(x)\}
$$ 
is  the \textit{supporting hyperplane} of $E$ for $x$ (See Figure \ref{fig: norming}). For $x\in E$, the  set
$$
F_x(E):=E\cap H(E,x)=\argmax_E(\varphi_x)
$$ 
is called the \textit{face} of $E$ defined  by $x$, or also the \textit{support set} of $E$ for $x$. For $v\in F_x(E)$ we say that $H(E,x)$ supports $E$ at $v$.
\end{defn}

In the literature the faces defined above are usually called exposed faces.

\begin{defn}[Minkowski gauge]\label{minkowskigauge}
The \textit{Minkowski gauge} or \textit{gauge} of a bounded convex set $B\subseteq V$ which contains the origin in its interior is the function 
$$
g_B:V\to\R,\quad g_B(x)=\inf\{t>0:x\in tB\}.
$$ 
\end{defn}

The set $B$ is a symmetric convex body (Definition \ref{convexb}) if and only if the gauge function $g_B$ is a norm on $V$ whose unit ball is $B$. Otherwise it is a Finsler norm, i.e. only positively homogeneous.

\begin{rem}\label{normingsupport} 
If $x\in V$ is such that $g_B(x)=1$ then $H_y=\varphi_y^{-1}(1)=\{z\in V:\langle z,y\rangle=1\}$ is a supporting hyperplane of $B$ at $x$ if and only if $\varphi_y$ is a norming functional of $x$.
\end{rem}

\begin{defn}
The \textit{polar} of a nonempty bounded set $E\subseteq V$ is 
$$E^\circ =\{x\in V:\langle x,y\rangle\leq 1 \mbox{ for all }y\in E\}.$$
\end{defn}

Note that if $E$ is invariant by an isometric action so is its polar $E^\circ$. A \textit {polytope} is the convex hull of a finite set of points.

\bigskip

 These are the result of applying the polar operation to some standard sets
\begin{itemize}
\item[-] If $B_p$ is the unit ball of the $\ell^p$ space for $1\leq p\leq \infty$ then $B_p^\circ=B_q$ where $q$ is conjugate to $p$.
\item[-] The polar of an ellipsoid $E\subseteq \mathbb R^n$ with axes of length $a_1,\dots,a_n$ is the ellipsoid with axes of length $1/a_1,\dots, 1/a_n$.
\item[-] The polar of a polytope $P\subseteq V$ containing $0$ in its interior and which has $n$ faces and $k$ vertices is a polytope with $k$ faces and $n$ vertices.
\end{itemize}
\begin{rem}\label{polar}
These are some standard properties that will be used later; let $E,F\subseteq V$ be compact  convex sets containing $0\in V$ in the interior, then
\begin{enumerate}
\item $E^{\circ\circ}=E$, this is the bipolar property.
\item $(\lambda E)^\circ=\lambda^{-1}E^\circ$ for $\lambda >0$. 
\item If $E\subseteq F$, then $F^\circ\subseteq E^{\circ}$.
\item $(E\cup F)^\circ=E^\circ\cap F^\circ$.
\item $(E\cap F)^\circ=\conv(E^\circ\cup F^\circ)$.
\item For a polytope $E=\conv\{x_1,\dots,x_n\}$ we have 
$$
E^\circ=\{y\in V:\langle x_i,y\rangle\leq 1\mbox{  for  }i=1,\dots,n\}.
$$
\item For $E=\{y\in V:\langle x_i,y\rangle\leq 1\mbox{  for  }i=1,\dots,n\}$ its polar is the polytope  
$$
E^\circ=\conv\{0,x_1,\dots,x_n\}.
$$
\end{enumerate}
\end{rem}

The next result can be found in Theorem 6.4 and Corollary 6.5 of \cite{bron}, and Theorem 14.5 of \cite{rock}.

\begin{thm}\label{supportdual}
Let $E\subseteq V$ be a compact convex set containing $0$ in its interior. Then 
\begin{enumerate}
\item $h_E=g_{E^\circ}$ and $g_E=h_{E^\circ}$.
\item The supporting hyperplanes of $E$ are the hyperplanes determined by the points of the boundary $\bd E^\circ$ of $E^\circ$, i.e. $H_x$ with $x\in\bd E^\circ$.
\end{enumerate}
\item The following are equivalent, see Figure \ref{fig: duality}
\begin{itemize}
\item[-] the hyperplane $H_y$ supports $E$ at $x\in E$.
\item[-] the hyperplane $H_x$ supports $E^\circ$ at $y\in E^\circ$.
\item[-] $x\in\bd E$, $y\in \bd E^\circ$ and $\langle x,y\rangle=1$.
\end{itemize} 
\end{thm}

\begin{figure}[h]
\def\svgwidth{12cm}
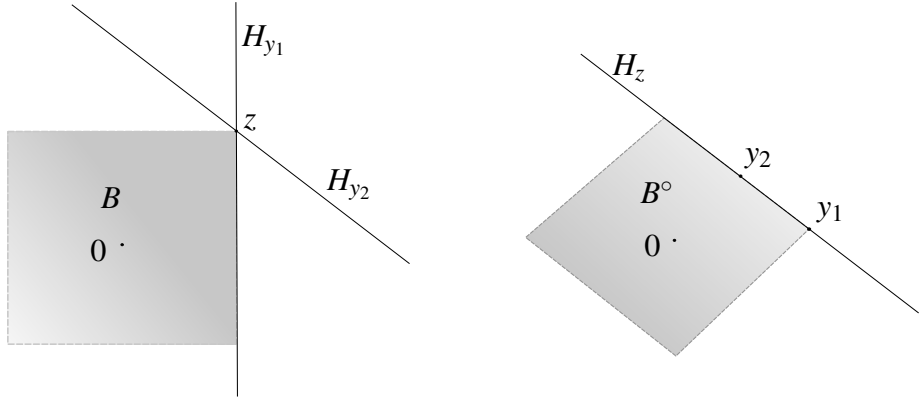
\caption{\small{Polar duality and supporting hyperplanes.}}
\label{fig: duality}
\end{figure}

The following, which relates the polar operation to orthogonal projections and sections with subspaces will also be used. Its proof is elementary therefore omitted. 

\begin{rem}\label{projpolar}
If $B$ is a convex body in the inner product space $V$, $W$ is a subspace of $V$,  and $p_W$ is the orthogonal projection onto $W$, then the following holds:
\begin{itemize}
\item[i)] $(B\cap W)^\circ=p_W(B^\circ).$
\item[ii)] $B^\circ\cap W=p_W(B)^\circ.$
\end{itemize}  
Item ii) actually holds for any subset $B$ of $V$ and item i) follows from polar duality.
\end{rem}

\subsection{Convex structure of the unit ball of Hofer norms} 

In this section, we characterize the generalized Hofer norms in terms of the supporting and gauge functions $h_E,g_E$.

\begin{prop}\label{hofernorms}
Let $E\subseteq V$ be a compact full set. Then we have 
\begin{align*}
\|y\|_E & =h_{E-E}(y)=g_{(E-E)^\circ}(y).\\
\|y\|_E'&=h_{E\cup-E}(y)=g_{(E\cup-E)^\circ}(y).
\end{align*}
If in addition $0$ is in the interior of $\conv(E)$ 
\begin{align*}
\|y\|_E^+ & =h_{E}(y)=g_{E^\circ}(y).
\end{align*}
\end{prop}
\begin{proof}
For the first Hofer norm observe that  
\begin{align*}
\|y\|_E&=\max_{x\in E}\langle x,y\rangle-\min_{x\in E}\langle x,y\rangle= \max_{x\in E}\langle x,y\rangle-(-\max_{x\in E}\langle x,-y\rangle)\\
&=h_E(y)-(-h_E(-y))=h_E(y)+h_E(-y)&\mbox{  by Definition \ref{defih}}\\
&=h_E(y)+h_{-E}(y)=h_{E-E}(y)&\\
&=g_{(E-E)^\circ}(y)&\mbox{  by Theorem  \ref{supportdual}.}
\end{align*}

For the second Hofer norm
\begin{align*}
\|y\|'_E&=\max\{\max_{x\in E}\langle x,y\rangle,-\min_{x\in E}\langle x,y\rangle\} \\
&=\max\{\max_{x\in E}\langle x,y\rangle,\max_{x\in E}-\langle x,y\rangle\}\\
&=\max\{\max_{x\in E}\langle x,y\rangle,\max_{x\in -E}\langle x,y\rangle\}=\max_{x\in E\cup-E}\langle x,y\rangle=h_{E\cup-E}(y)\\
&=g_{(E\cup-E)^\circ}(y)&\mbox{  by Theorem  \ref{supportdual}}.
\end{align*}
For the one-sided Hofer norm
\begin{align*}
\|y\|_E^+&=\max_{x\in E}\langle x,y\rangle=h_{E}(y)=g_{E^\circ}(y).
\end{align*}
\end{proof}

\begin{rem}\label{ballh}
Note that the unit ball of Hofer's norm is exactly
$$
(E-E)^\circ=(\conv(E-E))^\circ=(\conv(E)-\conv(E))^\circ.
$$
Therefore all norms are generalized Hofer norms if we take $E=\frac{1}{2}B^\circ$ where $B$ is the unit ball of the norm.
\end{rem}

With these tools, we give a second proof of the characterization of the norming functionals of vectors in the unit sphere (Theorem \ref{normingfunc}), or equivalently, the supporting hyperplanes of the unit ball for points in its sphere.

\begin{thm}\label{normingpolar}
The norming functionals of $x \in V$ are given by $\varphi_{y^+ -y^-}$, with 
$$
y^+ \in \conv(\argmax_E  \varphi_x)\quad \textrm{ and }\quad y^-\in\conv(\argmin_E \varphi_x).
$$ 
\end{thm}
\begin{proof} Rescaling, we can assume that $1=\|x\|_E=g_{\conv(E-E)^\circ}(x)$, so that  $x\in\bd (E-E)^\circ$. By Remark \ref{normingsupport} a functional $\varphi_y$ is a norming functional of $x$ if and only if the hyperplane $H_y$ supports $(E-E)^\circ$ at $x$. Theorem \ref{supportdual} implies that $H_y$ supports $(E-E)^\circ$ at $x$ if and only if $H_x$ supports $(E-E)^{\circ\circ}=\conv(E-E)=\conv(E)-\conv(E)$ at $y\in \conv(E-E)$. Moreover, $H_x$ supports $\conv(E-E)$ at $y\in \conv(E-E)$ if and only if $y\in\argmax_{\conv(E)-\conv(E)}(\varphi_x)$, and this happens if and only if 
$$
y\in\argmax_{\conv(E)}(\varphi_x)-\argmin_{\conv(E)}(\varphi_x).
$$
That is, $y=y^+-y^-$ with 
$$
y^+\in \argmax_{\conv(E)}(\varphi_x)=\conv(\argmax_E  \varphi_x)
$$
and likewise with $y^-$. This finishes the proof.
\end{proof}

\begin{rem}
For the one-sided Hofer norm the norming functionals of $x \in V$ are given by $\varphi_{y^+}$, with 
$y^+ \in \conv(\argmax_E  \varphi_x)$. 
\end{rem}

\section{Hofer's metric on compact Lie groups}\label{hofergroup}

In this section, we present actions of compact semi-simple Lie groups $K$ on compact connected manifolds $M$ with symplectic form $\omega$, as a nice setting for the convex geometry that was discussed in the previous sections. It should serve as motivation and also as a source of examples. We refer to \cite{ps} for general background on the geometry of Hamiltonian actions.

\subsection{Hamiltonian diffeomorphisms}

Let $(M,\omega)$ be a connected closed symplectic manifold and let $H:[0,1]\times M \to \R$ be a smooth function. We denote $H_t(m):=H(t,m)$. This function $H$ induces a time dependent Hamiltonian vector field $X_{H_t}$ by Hamilton's equations 
\begin{align}
dH_t=\omega(\cdot,X_{H_t})=-\iota_{X_{H_t}}\omega,\label{hamiltoneq}
\end{align}
and hence an isotopy $\phi^H_t:M\to M$, $t\in [0,1]$ by the prescription that
$$
\phi^H_0=\phi\mbox{  and  } \frac{d}{dt}\phi^H_t(m)=X_{H_t}(\phi^H_t(m))
$$
for a symplectic map $\phi$. The Hamiltonian diffeomorphism group $\Ham(M,\omega)$ is by definition the set of diffeomorphisms $\phi:M\to M$ which can be written as $\phi=\phi^H_1$ for some $H$ and $\phi^H_0=\id$ as above. 

The set $\Ham(M,\omega)$ is an infinite dimensional group under composition, all elements of which are symplectomorphisms of $(M,\omega)$. Its Lie algebra are the Hamiltonian vector fields in $\X(M)$, which can be identified with the smooth functions in $M$ modulo constant functions, via (\ref{hamiltoneq}), hence  
$$
T_{\id}\Ham(M,\omega)\simeq \X_{\Ham}(M)\simeq C^{\infty}(M)/\R\bone.
$$
Since $\psi^*X_H=X_{H\circ \psi}$ the adjoint action in this group is given by $\Ad_\psi [H]=[H\circ \psi]$, where $[H]$ will denote the class of $H$ modulo constant functions. If the Hamiltonian is time independent, i.e. $H_t=H$ for $t\in[0,1]$, it is called \textit{autonomous}. Note that the flow of such $H$ is ruled by the equation
$$
\frac{d}{dt}\phi_t(m)=X_H(\phi_t(m))=D(R_{\phi_t})_1(X_H)
$$
when we interpret the differential of the right translation $R_g$ in the group of diffeomorphisms, as composition from the right. Therefore $\phi_t$ is the flow of the right invariant field $X_g=D(R_g)_1(X_H)$ and with the initial condition $\phi_0=\id=1_{\Ham(M,\omega)}$ it is clear that $\phi_1=\exp(X_H)$, where $\exp$ is the exponential map of the group of Hamiltonian diffeomorphism. However, this exponential map is not well-suited as a chart for the group, since it is not a local diffeomorphism in any reasonable neighbourhood of the $0$ vector field, see \cite{ps}.

\subsubsection{Hofer's norm}

The $\Ad$-invariant $L^\infty$ norm 
\begin{equation}\label{normahofer}
\|[H]\|= \max_MH -\min_M H
\end{equation}
on the Lie algebra $T_{\id}\Ham(M,\omega)\simeq C^{\infty}(M)/\R\bone$ of the group $\Ham(M,\omega)$ is \textit{Hofer's norm}. It induces a Finsler length structure on curves $(\phi_t)_{t\in [0,1]}$ in $\Ham(M,\omega)$ by means of 
\begin{align*} 
\length(\phi_t^H)& =\int_0^1\|\frac{d}{dt}\phi_t^H\|dt =\int_0^1\|H_t\circ\phi^H_t\|dt\\
& =\int_0^1 \left(\max_M H_t-\min_M H_t \right)  dt,
\end{align*}
and hence a bi-invariant distance
$$\dist(\phi_0,\phi_1) =\inf\left\{\length(\phi^t_H) :\phi^0_H=\phi_0,\quad\phi^1_H=\phi_1\right\}.$$
As was shown for $\mathbb R^{2n}$ in \cite{hofer} and for general symplectic manifolds in \cite{lm}, $\dist$ is a non degenerate, bi-invariant metric on $\Ham(M,\omega)$. 

\bigskip

Hofer proved that the path of any autonomous Hamiltonian on $\C^n$ is length minimizing (among homotopic paths with fixed endpoints) as long as the corresponding Hamilton's equation has no non-constant time-one periodic orbit. This result was generalized in \cite{lm} to general symplectic manifolds.

\begin{defn}\label{quasiauto}
A path $(\phi_t)_{t\in [0,1]}\subseteq \Ham(M,\omega)$ is a geodesic of the Hofer metric if each $t\in [0,1]$ has a neighbourhood $I$ such that $\phi|_I$ is minimal, i.e. no longer than any other path joining its endpoints. A Hamiltonian $H_t$ is called \textit{quasi-autonomous} if there exists two points $x^-,x^+ \in M$ such that 
$$H_t(x^-) = \min_M H_t,\quad H_t(x^+) = \max_M H_t$$
for all $t\in [0,1]$. 
\end{defn}

In Section \ref{geodebi} below we will give a characterization of all short paths for a Finsler metric in a compact Lie group $K$  (Theorem \ref{cuasi}). When the metric is the pull-back metric obtained by the action of $K$ on a symplectic manifold $(M,\omega)$ (see Section \ref{subsectionhoferham} below), we will be able to show that the autonomous Hamiltonians are in correspondence with one-parameter groups in $K$, and all other minimizing paths in $K$ are in correspondence with the quasi-autonomous Hamiltonians (Theorem \ref{teoautono}). These results should be compared to the following, obtained by Hofer and others with an entirely different approach (see \cite[Section 12.3]{mds} and the references therein for proofs):
\begin{thm}\label{geodquasia}
Let $(\phi_t)_{t\in [0,1]}\subseteq \Ham(M,\omega)$ be a regular path ($C^1$ and with non-vanishing derivative). If $\phi$ is short for the Hofer metric, then the corresponding Hamiltonian $H_t$ is quasi-autonomous. If the Hamiltonian is quasi-autonomous, then $\phi$ is locally short (locally here refers to the time-variable). 
\end{thm}

\smallskip

A very similar norm used in the literature is defined without taking the quotient of the Hamiltonian functions by the constant functions. It is defined via the normalization $\int_MH\omega^n=0$, where $\omega^n$ is the Liouville measure on $M$. The norm is the $\Ad$-invariant $L^\infty$ norm 
$$
\|H\|'= \max\{\max H,-\min H\}= \|H\|_{\infty}
$$
which we are going to relate to the second Hofer norm in Remark \ref{hofernormalizada}.

\subsection{Hamiltonian actions}\label{hamactions}

Let $K$ be compact semi-simple Lie group with Lie algebra $\k=T_1K$ and dual space $\k^*$, which we identify with $\k$ through the duality pairing given by the (opposite of) the $\Ad$-invariant Killing form $\langle \cdot,\cdot\rangle$. This is an inner product in $\k$ because $\k$ is compact; since it is also $\Ad$-invariant, we can identify the coadjoint action with the adjoint action. The Lie algebra $\k$ acts by skew-symmetric transformations, that is $\ad x$ is skew-adjoint for this inner product, for any $x\in\k$.

\begin{defn}
Assume that the action of $K$ on $M$ is a \textit{symplectic action}, that is, there is a smooth map $\Phi:K\times M\to M$ that we denote $\Phi(g,m)=g\cdot m$ such that for each fixed $g\in K$, the automorphism $\Phi_g=\Phi(g,\cdot)$ is a symplectomorphism of $M$. We also assume that the action is \textit{almost effective} ($g\mapsto \Phi_g$ has discrete kernel). For fixed $m\in M$ we denote by $\pi_m:K\to M$ the map $\pi_m(g)=\Phi(g,m)$.

For $x\in \k$, let $x_M(m)=\frac{d}{dt}\mid_{t=0}\exp(tx)\cdot m\in T_mM$ denote the infinitesimal action on $M$. The assumption that the action is almost effective implies that $\k \ni x\mapsto x_M\in \mathfrak X(M)$ is injective; the assumption that the action is symplectic implies that the field $x_M$ is symplectic, i.e.  $\mathcal L_{x_M}\omega=0$.

A \textit{moment map} for a symplectic $K$ action on $(M,\omega)$ is a map $\mu:M\to\k^*$ defined by
$$
\mu^x(m)=\langle\mu(m),x\rangle, \quad \mu^x:M\to \mathbb R
$$
such that $\mu$ intertwines the $K$-action on $M$ and the coadjoint action on $\k^*$, i.e. $\mu(g\cdot m)= \Ad_g\mu(m)$ for all $g\in K,m\in M$, and such that $\mu$ satisfies Hamilton's equation
\begin{align}
d\mu^x=-\iota(x_M)\omega=-\omega (x_M,\cdot)\mbox{  for all  }x\in\k.\label{hamiltoneq2}
\end{align}
A symplectic $K$ action is called \textit{Hamiltonian} if it admits a moment map.
\end{defn}

\begin{rem}
A symplectic $K$ action on a symplectic manifold $(M,\omega)$ is the same thing as a homomorphism $K\to \Symp(M,\omega)$ to the group of symplectomorphisms of $(M,\omega)$ such that the map $K\times M\to M$, $(g,m)\to g\cdot m$ is smooth. The action has a moment map if and only if the image of the identity component of $K$ is contained in $\Ham(M,\omega)$, i.e. there is a map $\mu:M\to\k^*$ which satisfies Hamilton's equations (\ref{hamiltoneq2}). Averaging with the Haar probability measure $dk$ on $K$ yields an equivariant moment map 
$$
\tilde{\mu}(x)=\int_{k\in K}\Ad_k\mu(k^{-1}\cdot x)\, dk.
$$
\end{rem}

Note that if $x\in\k$ then $x_M$ is a Hamiltonian vector field and $H=\mu^x$ is a Hamiltonian; from (\ref{hamiltoneq2}) it is apparent that $X_{\mu^x}=x_M$. Also note that the Poisson bracket of the Hamiltonian functions $\mu^x$ and $\mu^y$ is given by $\{\mu^x,\mu^y\}=\mu^{[x,y]}$ for $x,y\in\k$.

The action defines a homomorphism with discrete kernel
$$K\to \Ham(M,\omega).$$
At the level of Lie algebras this morphism is injective
\begin{align}
T_1(K)=\k&\hookrightarrow C^{\infty}(M)/\R\bone\simeq T_{\id} \Ham(M,\omega)\label{liealgaction}\\
x&\mapsto [\mu^x]=\mu^x +\R\bone\nonumber\mapsto x_M.
\end{align}

Given $u\in K$ and a path $x:[0,1]\to\k$, $t\mapsto x_t$, we calculate the isotopy $\phi_t\in \Ham(M,\omega)$ given by the initial condition $\phi_0=\Phi_u$ and the time dependent Hamiltonian $H_t=\mu^{x_t}$ (as noted, the Hamiltonian vector fields are $(x_t)_M$). 

\begin{prop}\label{actiondiffeo}
If $\gamma:[0,1]\to K$ is the smooth solution to $\dot{\gamma}_t\gamma_t^{-1}=x_t$ and $\gamma_0=u$, then the isotopy is given by $\phi_t=\Phi_{\gamma_t}$. 
\end{prop}

\begin{proof}
We have to check that $\dot{\phi}_t\phi^{-1}_t=(x_t)_M$: for each $m\in M$,
\begin{align*}
\dot{\phi}_t(m) & = (\gamma_t\cdot m)^{\,\cdot}=D(\pi_m)_{\gamma_t}(\dot{\gamma}_t)=D(\pi_m)_{\gamma_t}(x_t \gamma_t)= D(\pi_m)_{\gamma_t}(\frac{d}{dr}\Bigr|_{r=0}\exp(rx_t)\gamma_t)\\
&=\frac{d}{dr}\Bigr|_{r=0} \pi_m(e^{rx_t}\gamma_t)=\frac{d}{dr}\Bigr|_{r=0} (e^{rx_t}\gamma_t)\cdot m=\frac{d}{dr}\Bigr|_{r=0} e^{rx_t}\cdot(\gamma_t\cdot m)\\
& =(x_t)_M(\gamma_t\cdot m)= (x_t)_M(\phi_t(m))=((x_t)_M\circ \phi_t)(m),
\end{align*}
and therefore $\dot{\phi}_t=(x_t)_M\circ\phi_t$ as claimed.
\end{proof}

\begin{rem}\label{chamber}
The image of the moment map $\mu(M)$ is a union of (co)adjoint orbits in $\k$ which we denote by $\O_{\lambda}=\Ad_K(\lambda)$ for $\lambda\in\k$. Let $T$ be a maximal torus of $K$ and $\t$ its Lie algebra, a Cartan subalgebra. The image of the moment map can therefore be parametrized by a subset $A=\mu(M)\cap\t_+$ of a closed positive Weyl chamber $\t_+$, corresponding to a choice positive simple roots. Hence we have
$$
\mu(M)=\bigsqcup_{\lambda\in A}\Ad_K(\lambda)=\bigsqcup_{\lambda\in A}\O_\lambda.
$$
\end{rem}

\subsubsection{Hofer's metric given by Hamitonian actions}\label{subsectionhoferham}

We use the inclusion of Lie algebras (\ref{liealgaction}) to pull-back the Finsler length structure given by  Hofer's norm. The $L^\infty$ norm on $C^{\infty}(M)/\R\bone$ restricted to the image $T_1(K)=\k$ is therefore by definition of the Hofer norm (\ref{normahofer}) given by
\begin{align}
\|\mu^x\| &=\max_{m\in M}\mu^x(m)-\min_{m\in M}\mu^x(m)=\max_{m\in M}\langle\mu(m),x\rangle-\min_{m\in M}\langle\mu(m),x\rangle\label{hofer1}\\
&=\max_{y\in\mu(M)}\langle y,x\rangle-\min_{y\in\mu(M)}\langle y,x\rangle= \|x\|_{\mu(M)},\nonumber
\end{align}
where the last equality is by definition of the generalized Hofer norm (Definition \ref{generhofernorm}) for the $\Ad$-invariant set $E=\mu(M)\subseteq \k$.

\begin{rem}[$\|\cdot\|_{\mu(M)}$ is a norm]\label{fullsiiae} Since $M$ is compact the image of the moment map is bounded. Also, because the image of the moment map is full in $\k$ if and only if the action is almost effective: The image of the moment map is not full if and only if its image is contained in a hyperplane, which in turn happens if and only if there exists $x\in\k$ such that $\varphi_x$ is constant on $\mu(M)$. This is equivalent to the map $\mu^x:M\to\R$ being a constant Hamiltonian which generates the zero vector field. Finally, this is also equivalent to the one-parameter group $\exp(tx)$ acting trivially, i.e. the action not being almost effective. 
\end{rem}

Therefore, the Hofer norm just described is the norm $\|\cdot\|_{\mu(M)}$ for a moment map $\mu:M\to\k^*\simeq\k$, given by equation (\ref{hofer1}) above. For a general $\Ad$-invariant set $E$ in $\k$ we obtain an $\Ad$-invariant norm $\|\cdot\|_{E}$ in $\k$. This norm is $\Ad$-invariant and induces a bi-invariant Finsler metric on $K$ by left (or right) translation given by
\begin{align*}
\length(\gamma)&=\int_0^1\|\dot{\gamma}_t\|_{\gamma_t}dt=\int_0^1\|\dot{\gamma}_t\gamma_t^{-1}\|_{\mu(M)} dt=\int_0^1\|\gamma_t^{-1}\dot{\gamma}_t\|_{\mu(M)}dt,
\end{align*}
\begin{align}\label{longitud}
\dist(u,v) &:=\inf\left\{\length(\gamma_t) :\gamma_0=u,\quad\gamma_1=v\right\}.
\end{align}

\begin{rem}[Barycenter]\label{hofernormalizada}
In the case of actions by semi-simple Lie groups, all Hamiltonians are normalized in the following sense:
$$\int_M \mu^x \omega^n=0.
$$
This is because
$$
\int_M \langle\mu,x\rangle \omega^n=\int_{\mu(M)}\langle y,x\rangle d\nu(y),
$$
where $\nu=\mu_*(\omega^n)$ is the pushforward of the normalized Liouville measure $\omega^n$ by the moment map. The measure $\nu$ is $\Ad$ invariant hence its center of mass is also invariant. The center of mass is zero since the group is semi-simple, therefore $\int_{\mu(M)}\langle y,x\rangle d\nu(y)=0$. The resulting $L^\infty$ norm is
\begin{equation}\label{hofer2}
\|\mu^x\|_{\infty}=\max\{\sup_{m\in M}\mu^x(m),-\inf_{m\in M}\mu^x(m)\}=\|x\|_{\mu(M)},
\end{equation}
by the definition of the second generalized Hofer norm in \ref{generhofernorm2} for the $\Ad$-invariant set $E=\mu(M)\subseteq\k$.
Therefore, this second Hofer norm just described is the norm defined in equation (\ref{generhofernorm2}) as $\|\cdot\|'_{\mu(M)}$ for a momentum map $\mu:M\to\k^*\simeq\k$.
\end{rem}

\begin{rem}\label{hoferlp}
We can define also an $L^p$-norm for $1\le p<\infty$ via the embedding
$$\iota:V\hookrightarrow L^p(E,\nu)/\R\bone, \quad x\mapsto [\varphi_x]=\varphi_x+\R\bone,$$
where $\nu$ is an $\Ad$-invariant measure. Consider the case of an $L^p$-norm defined by the inclusion
$$\iota:\k\hookrightarrow L^p(M,\nu)/\R\bone.$$
Since $L^p(M,\nu)$ is uniformly convex when $1<p<\infty$, the quotient $L^p(M,\nu)/\R\bone$ is uniformly convex. Therefore the  $L^p$-norm is uniformly convex and this implies by \cite[Theorem 4.15]{lar19} that $K$ with the induced metric is uniquely geodesic.  We won't be pursuing the geometry of these norms in this paper.
\end{rem}

\begin{rem}\label{regular}
Recall that for a compact semisimple group $K$, each element $\lambda\in K$ is semi-simple. A \textit{regular} element $\lambda$ is such that its commutant $\t={\mathfrak z}(\lambda)$ is a maximal abelian subalgebra; equivalently in this setting, $\t$ is a Cartan subalgebra of $\k$. 
\end{rem}

 In the case of $K=\SU(n)$, an element of $\k=\su(n)$ is regular if and only if all its eigenvalues are distinct.

\begin{ex}[Coadjoint orbits]\label{coadj}
An important special case is that of a compact connected semi-simple Lie group $K$ acting via $\Ad$ in a Hamiltonian way on a (co)adjoint orbit $\O_\lambda$ containing $\lambda\in\k^*\simeq \k$.

For the action to be almost effective, it is necessary and sufficient that $\lambda$ is non-zero in each simple summand of $\k$ (because in that case each simple factor of $K$ acts nontrivially on the orbit, and this is equivalent to the fact that the (co)adjoint orbit is full in $\k$  (see \cite[Lemma 17]{bgh} and Remark \ref{fullsiiae}). In particular this occurs if $\lambda$ is regular.

\smallskip

The adjoint orbit is endowed with the Kostant-Kirillov-Souriau symplectic form $\omega$. The kernel of this homomorphism is the center $Z(K)$ of $K$ which is a discrete subgroup. We have therefore an inclusion 
$$
K/Z(K)\hookrightarrow \Ham(\O_\lambda,\omega).
$$ 
At the level of Lie algebras this inclusion is $x\mapsto [\mu^x]$ where $\mu^x:\O_\lambda\to\R$ is the Hamiltonian map given by the $\mu^x(y)=\langle x, y\rangle$. The Hamiltonian is the component of the moment map $\iota:\O_\lambda\hookrightarrow\k^*\simeq\k$ along $x$, that is $\mu^x=\varphi_x\circ\iota$. Moreover, since the action of $K$ is given by $g\mapsto g xg^{-1}$, it is apparent that if $x\in \k$ then the induced Hamiltonian vector field is given by
$$
x_{\O_{\lambda}}(m)=\frac{d}{dt}\Bigr|_{t=0}\exp(tx)\cdot m = \frac{d}{dt}\Bigr|_{t=0}\Ad_{e^{tx}}(m)=(\ad x)(m)=[x,m].
$$
The symplectic form is 
$$
\omega_x( y_{\O_{\lambda}}(x),z_{\O_{\lambda}}(x) )=\langle x,[y,z]\rangle
$$
where the pairing is given by the opposite Killing form of $\k$ as before. The induced Hofer norm on $T_1(K)=\k$ is therefore
\begin{align}\label{normcadjoint}
\|x\|_{\O_\lambda}&=\max_{y\in\O_\lambda}\mu^x(y)-\min_{y\in\O_\lambda}\mu^x(y)=\max_{k\in K}\mu^x(\Ad_k\lambda)-\min_{k\in K}\mu^x(\Ad_k\lambda)\\
&=\max_{k\in K}\langle x,\Ad_k\lambda\rangle-\min_{k\in K}\langle x,\Ad_k\lambda\rangle.\nonumber
\end{align} 
Since the action is almost effective $\|\cdot\|_{\O_\lambda}$  is a norm. Since it is also $\Ad$-invariant it induces a bi-invariant Finsler metric on $K$.
\end{ex}

\begin{ex}\label{sun}
In the case of the special unitary group $\SU(n)$ we have the inclusion
$$\SU(n)/\Z_n\simeq\PU_n\hookrightarrow \Ham(\O_\lambda,\omega)$$
where $\O_\lambda$ is a (co)adjoint orbit of $\SU(n)$ passing through the given diagonal matrix $\lambda = i\diag(\lambda_1,\ldots,\lambda_n)\in\su(n)$ with $\lambda_1\leq \lambda_2 \leq \ldots \le \lambda_n$. At the Lie algebra level this inclusion is 
$$
T_1(\SU(n))=\su(n)\hookrightarrow C^{\infty}(\O_\lambda)/\R\bone\simeq T_{\id} \Ham(\O_\lambda,\omega)
$$ 
and $\mu^x:\O_\lambda\to\R$ is the linear Hamiltonian given by the opposite Killing form
$$
\mu^x(y)=-2n\tr(xy^*)=2n\tr(xy).
$$
The Hofer norm (\ref{normcadjoint}) is given by
$$
\|v\|_{\O_\lambda}=2n\max\{\tr(\lambda uvu^{-1}):u\in \SU(n)\}-2n\min\{\tr(\lambda wvw^{-1}):w\in \SU(n)\},
$$
which is exactly ($2n$-times) the $\lambda$-numerical diameter of $v$, see the next remark. 
\end{ex}

\begin{ex}
The group $\SU(2)/{\{\id,-\id\}}\simeq\SO(3)$ acts on a nontrivial (co)adjoint orbit $\O_\lambda\simeq S^2$ which is the two dimensional sphere endowed with the area form $\omega$. The action is given by rotations, see  Example 1.4.H in \cite{pol01}. 
\end{ex}

\begin{rem}[One-sided norms]\label{semiads} For fixed $0\ne \lambda\in \k$, and $v\in \k$, consider the function
$$
\|v\|_{\O_\lambda}^+=\max_{k\in K}\langle v , \Ad_k\lambda\rangle.
$$
Note that by the $\Ad$-invariance of the norm, this can be also computed as
$$
\|v\|_{\O_\lambda}^+=\max_{k\in K}\langle \Ad_k v ,\lambda\rangle=\max\{\langle y , \lambda\rangle:  y\in \conv(\O_v)\}.
$$
When $\lambda$ is non-zero on each simple-summand of $\k$ (in particular if $\lambda$ is regular), then $0$ is in the interior of its closed convex hull (see Remark \ref{onesided} below). Therefore this is a Finlser norm, in fact it is our one-sided Hofer norm of Remark \ref{otrasnorm}. In the context of the group of Hamiltonian symplectomorphisms, these norms appear in \cite{md} where the name \textit{one-sided} was used.  
\end{rem}

\section{Convex geometry of Hofer's norm on Cartan algebras}\label{convexgeocart}

In this section we study generalized Hofer norms on the Lie algebra of a compact semisimple Lie group. We use several convexity theorems and the convex analysis of $\Ad$-invariant convex functions. This convex analysis expresses properties of $\Ad$-invariant functions in terms of properties of its restrictions to Cartan algebras and positive Weyl chambers, which are fundamental domains for the adjoint action.

\smallskip

We first recall Kirwan's nonabelian convexity theorem (see \cite{guisja}), which will be useful for our purposes of characterizing the Hofer norm restricted to a Cartan subalgebra $\t\subseteq \k$. As before $\t_+\subseteq \t$ is a choice of closed positive Weyl chamber.

\begin{thm}[Kirwan]\label{kirwan}
If $K\curvearrowright M$ is a Hamiltonian action with $K$ a compact connected Lie group and $(M,\omega)$ compact and connected, then $\mu(M)\cap\t_+$ is a convex polytope, i.e. the convex hull of a finite set of points $x_1,\dots,x_n$ in $\t_+$, that is
$$\mu(M)\cap\t_+=\conv\{x_1,\dots, x_n\}.$$
\end{thm}

This theorem is a generalization of the Atiyah-Guillemin-Sternberg theorem, which states that for a Hamiltonian action of a torus on a compact connected manifold the image of the moment map is the  convex hull of the image of the fixed point set of the action. The Atiyah-Guillemin-Sternberg theorem is a generalization of Kostant's convexity theorem which will be useful  here:

\begin{thm}[Kostant]\label{kostant}
If $K$ is a compact Lie group and $T$ is a maximal torus in $K$ and $p:\k\to\t$ is the projection of the Lie algebra of $K$ onto the Lie algebra of $T$, then 
$$p(\O_\lambda )=\conv(\W . \lambda ),$$
where $\W .\lambda$ is the Weyl group orbit of $\lambda\in\k$.
\end{thm}

The projection $p$ is taken along the orthogonal direction given by (minus) the Killing form of $\k$. 

\subsection{The Hofer polytope} With these tools we characterize the intersection with Cartan algebras of unit balls (of Hofer norms given by compact, full and $\Ad$-invariant sets $E\subseteq \k$). 
 
\begin{lem}\label{projconv}
If $E$ is $\Ad$-invariant, and $p:\k\to \t$ is the orthogonal projection, then 
$$
p(\conv(E))=\conv\{w.x:x\in E\cap\t_+, \,w\in\W\},
$$
and if $E\cap\t_+=\conv\{x_1,\dots,x_n\}$, then 
$$
p(\conv(E))=\conv\{w.x_i:i=1,\dots,n,\,w\in\W\}.
$$
\end{lem}
\begin{proof}
This follows from the  identities
\begin{align*}
p(\conv(E))&=p(\conv(\cup_{\lambda\in E\cap\t_+}\O_\lambda))&\\
&=\conv(p(\cup_{\lambda\in E\cap\t_+}\O_\lambda))&\mbox{ since }p\mbox{ is affine}\\
&=\conv(\cup_{\lambda\in E\cap \t_+}p(\O_\lambda))\\
&=\conv(\cup_{\lambda\in E\cap \t_+}\{w.\lambda:w\in\W\})&\mbox{ by Theorem \ref{kostant}}\\
&=\conv\{w.\lambda:\lambda\in E\cap \t_+,\,w\in\W\}.
\end{align*}
Now observe that if $E\cap\t_+=\conv\{x_1,\dots,x_n\}$
\begin{align*}
p(\conv(E))&=\conv(\cup_{\lambda\in E\cap\t_+}\{w.\lambda:w\in\W\})& \\
&=\conv(\cup_{w\in \W}w.(E\cap \t_+))=\conv(\cup_{w\in \W}\conv(w.(E\cap\t_+)))\\
&=\conv(\cup_{w\in \W}\conv\{w.x_i:i=1,\dots,n\})\\
&=\conv\{w.x_i:i=1,\dots,n,\,w\in\W\}
\end{align*}
where the third and last equalities follow from basic properties of the convex hull operation. 
\end{proof}

\begin{rem}\label{proyecartan}
Let $B\subseteq \k$ be an $\Ad$-invariant convex body, and $p:\k\to \t$ the orthogonal projection to a Cartan sub-algebra, then Kostant's Theorem \ref{kostant} implies that $B\cap \t=p(B)$. This follows from $p(b)\in p(\O_b)=\conv(\O_b\cap \t)\subseteq\conv(B\cap \t)=B\cap\t$ for $b\in B$. 
\end{rem}

\begin{prop}\label{hoferpoly}
Set $A=E\cap\t_+$, then
\begin{align*}
(\conv(E)-\conv(E))\cap\t&=\conv\{w_1.x_1-w_2.x_2:x_1,x_2\in A,\,w_1,w_2\in\W\}.\\
(\conv(E\cup -E))\cap\t&=\conv\{w.x:x\in A\cup -A,w\in\W\}.\\
\conv(E)\cap\t&=\conv\{w.x:x\in A,w\in\W\}.
\end{align*}

If furthermore $A=\conv\{x_1,\dots,x_n\}$ is a polytope, then
\begin{align*}
(\conv(E)-\conv(E))\cap\t&=\conv\{w.x-w'.x':x,x'\in \{x_1,\dots,x_n\},\,w,w'\in\W\}.\\
(\conv(E)\cup\conv(E))\cap\t&=\conv\{w.x:x\in \{x_1,\dots,x_n,-x_1,\dots,-x_n\},w\in\W\}.\\
\conv(E)\cap\t&=\conv\{w.x:x\in \{x_1,\dots,x_n\},w\in\W\}.
\end{align*}
\end{prop}

\begin{proof}
For the first equality note that by the previous remark,
\begin{align*}
(\conv(E)-\conv(E))\cap\t & = p(\conv(E)-\conv(E)) \\
& = p(\conv(E))-p(\conv(E))\\
& = \conv\{w_1.x_1-w_2.x_2:x_1,x_2\in A,\,w_1,w_2\in\W\}&
\end{align*}
where the last equality is due to Lemma \ref{projconv}. For the fourth assertion note that
\begin{align*}
(\conv(E)-\conv(E))\cap\t & = p(\conv(E)-\conv(E)) = p(\conv(E))-p(\conv(E))\\
& = \conv\{w.x-w'.x':x,x'\in \{x_1,\dots,x_n\},\quad w,w'\in\W\}
\end{align*}
holds, where we used the second assertion of Lemma \ref{projconv} in the last equality.

For the second equality note that by Remark \ref{proyecartan},
\begin{align*}
(\conv(E\cup -E)\cap\t & = p(\conv(E\cup -E)) \\
& = p(\conv(\conv(E)\cup \conv(-E))) \\
& = \conv(p(\conv(E)\cup \conv(-E))) \\
& = \conv(p(\conv(E))\cup p(\conv(-E))) \\
& = \conv(p(\conv(E))\cup -p(\conv(E))) \\
& = \conv(\{w.x:x\in A,\,w\in\W\}\cup -\{w.x:x\in A,\,w\in\W\}) \\
& = \conv(\{w.x:x\in A\cup -A,\,w\in\W\}) 
\end{align*}
where the penultimate equality is due to Lemma \ref{projconv}. For the fifth assertion note that
\begin{align*}
(\conv(E\cup -E)\cap\t & = p(\conv(E\cup -E)) \\
& = \conv\{\pm w.x:x\in \conv\{x_1,\dots,x_n\},\,w\in\W\}\\
& = \conv\{w.x:x\in \{x_1,\dots,x_n,-x_1,\dots,-x_n\},\,w\in\W\}
\end{align*}
holds, where we used the previous identity and properties of the convex hull operation.

The proof of the third and fourth equalities is simpler and we omit it.
\end{proof}

\begin{figure}[h]
\def\svgwidth{12cm}
\begingroup%
  \makeatletter%
  \providecommand\color[2][]{%
    \errmessage{(Inkscape) Color is used for the text in Inkscape, but the package 'color.sty' is not loaded}%
    \renewcommand\color[2][]{}%
  }%
  \providecommand\transparent[1]{%
    \errmessage{(Inkscape) Transparency is used (non-zero) for the text in Inkscape, but the package 'transparent.sty' is not loaded}%
    \renewcommand\transparent[1]{}%
  }%
  \providecommand\rotatebox[2]{#2}%
  \newcommand*\fsize{\dimexpr\f@size pt\relax}%
  \newcommand*\lineheight[1]{\fontsize{\fsize}{#1\fsize}\selectfont}%
  \ifx\svgwidth\undefined%
    \setlength{\unitlength}{533.7414652bp}%
    \ifx\svgscale\undefined%
      \relax%
    \else%
      \setlength{\unitlength}{\unitlength * \real{\svgscale}}%
    \fi%
  \else%
    \setlength{\unitlength}{\svgwidth}%
  \fi%
  \global\let\svgwidth\undefined%
  \global\let\svgscale\undefined%
  \makeatother%
  \begin{picture}(1,0.28891258)%
    \lineheight{1}%
    \setlength\tabcolsep{0pt}%
    \put(0,0){\includegraphics[width=\unitlength,page=1]{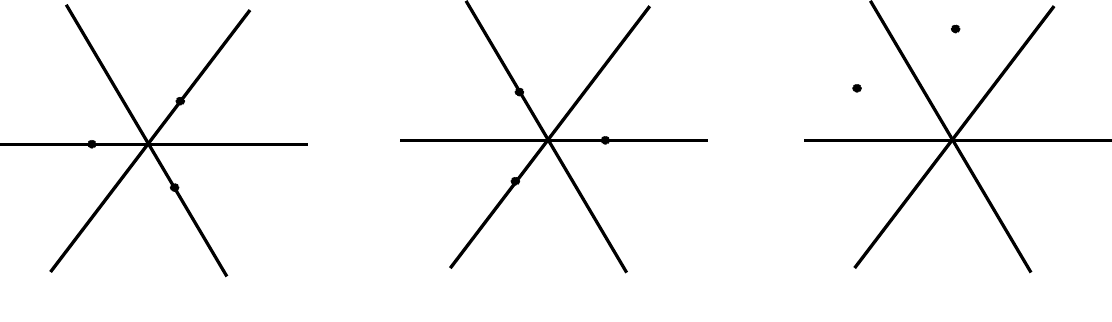}}%
    \put(0.07670314,0.01720077){\color[rgb]{0,0,0}\makebox(0,0)[lt]{\lineheight{1.25}\smash{\begin{tabular}[t]{l}$\O_{\lambda}\cap\t$\end{tabular}}}}%
    \put(0.45101048,0.01257712){\color[rgb]{0,0,0}\makebox(0,0)[lt]{\lineheight{1.25}\smash{\begin{tabular}[t]{l}$-\O_{\lambda}\cap \t$\end{tabular}}}}%
    \put(0,0){\includegraphics[width=\unitlength,page=2]{ext.pdf}}%
    \put(0.76260028,0.00311773){\color[rgb]{0,0,0}\makebox(0,0)[lt]{\lineheight{1.25}\smash{\begin{tabular}[t]{l}$ext(P)\cup\{0\}$\end{tabular}}}}%
  \end{picture}%
\endgroup%

\caption{\small{Extremal points of the Hofer norm polytope.}}
\label{fig: ext}
\end{figure}

In Figure \ref{fig: ext} we illustrate the first couple of equalities of Proposition \ref{hoferpoly} in the case $E=\O_\lambda$ where $\lambda\in\su(3)$ is singular and non zero.  

\begin{defn}\label{hoferpolytopes}
For $E\cap\t_+=\conv\{x_1,\dots,x_n\}$, we call the polytope 
$$
P=\conv\{w.x-w'.x':x,x'\in \{x_1,\dots,x_n\},\quad w,w'\in\W\}
$$
of Proposition \ref{hoferpoly} the \textit{Hofer norm polytope}. We call
$$
P'=\conv\{w.x:x\in \{x_1,\dots,x_n,-x_1,\dots,-x_n\},\,w\in\W\}
$$
the \textit{second Hofer norm polytope}. Finally, we call
$$P^+=\conv\{w.x:x\in \{x_1,\dots,x_n\},\,w\in\W\}$$
the \textit{one-sided Hofer norm polytope}.
\end{defn}

\begin{rem}[Polytopes, norms, unit balls]\label{polnorm}
If $\ext(P)=\{y_1,\dots,y_m\}$ are the extreme points of the Hofer norm polytope $P$ then this set is Weyl group invariant. If 
$$
B=(\conv(E)-\conv(E))^\circ
$$
is the unit ball of the Hofer norm (Remark \ref{ballh}), from Remarks \ref{projpolar} and  \ref{proyecartan}, we obtain
\begin{equation}\label{convep}
B\cap \t=(\conv(E)-\conv(E))^\circ\cap\t=P^\circ.
\end{equation} 
The same holds for the other Hofer norms with 
$B'=(\conv(E)\cap-\conv(E))^\circ $ and $B^+=\conv(E)^\circ$ and the polytopes $P'$ and $P^+$ respectively (for the second Hofer norm and the one-sided Hofer norm, see Remark \ref{otrasnorm}). The polytopes $P$ and $P'$ are centrally symmetric, and in general $P^+$ is not.
\end{rem}

From (\ref{convep}) and Proposition \ref{hofernorms} we obtain the following

\begin{cor}
Let $\|\cdot\|_{\mu(M)}$ be the Hofer norm derived from the $\Ad$-invariant set $E=\mu(M)$, and let $P$ be the corresponding Hofer norm polytope. Then Hofer norm restricted to the Cartan subalgebra $\t$ is given by the Minkowski gauge $g_{P^\circ}$, that is 
$$
\|x\|_{\mu(M)}=g_{P^\circ}(x)=\inf\{t>0: x\in tP^{\circ}\} \qquad \forall x\in \t.
$$
\end{cor}

\begin{rem}
From Remark \ref{polar} it follows that 
$$
P^\circ=\{x\in V:\langle y_i,x\rangle\leq 1\mbox{  for  }i=1,\dots,m\}.
$$
\end{rem}

\subsubsection{Direct sums of manifolds and polytopes}

If $M_1,\dots,M_n$ are symplectic manifolds equipped with Hamiltonian actions of $K$, with moment maps $\mu_1,\dots,\mu_n$, then the induced action on $M_1\times\dots\times M_n$ is also Hamiltonian with moment map $\mu:M_1\times\dots\times M_n\to\k$ given by 
$$\mu(m_1,\dots,m_n)=\mu_1(m_1)+\dots+\mu_n(m_n).$$
Therefore the image of $\mu$ is given by 
\begin{align}\label{sumaimagmoment}
\mu(M_1\times\dots\times M_n)=\mu_1(M_1)+\dots+\mu_n(M_n).
\end{align}
To find the Hofer norm polytope of this action we need the following 

\begin{lem}\label{sumadeadinv}
If $E_1,\dots,\E_n$ are $\Ad$-invariant and compact subsets of $\k$, and $p:\k\to \t$ is the orthogonal projection, then
\begin{align*}
(\conv(E_1&+\dots+\E_n)-\conv(E_1+\dots+\E_n))\cap\t\\
&=(\conv(E_1)-\conv(E_1))\cap\t+\dots +(\conv(E_n)-\conv(E_n))\cap\t.
\end{align*}
\end{lem}

\begin{proof}
Again, by Remark \ref{proyecartan}
\begin{align*}
(\conv(E_1&+\dots+\E_n)-\conv(E_1+\dots+\E_n))\cap\t  \\
& = p((\conv(E_1+\dots+\E_n)-\conv(E_1+\dots+\E_n))) \\
& = p((\conv(E_1)-\conv(E_1))+\dots +(\conv(E_n)-\conv(E_n)))\\
& = p(\conv(E_1)-\conv(E_1))+\dots + p(\conv(E_n)-\conv(E_n))\\
& =(\conv(E_1)-\conv(E_1))\cap\t+\dots +(\conv(E_n)-\conv(E_n))\cap\t.
\end{align*}
\end{proof}

The next proposition is now a straightforward consequence of Lemma \ref{sumadeadinv}.

\begin{prop}\label{sumahoferpoly}
If $M_1,\dots,M_n$ are symplectic manifolds equipped with Hamiltonian almost effective actions of a compact semi-simple group $K$ such that the Hofer norm polytopes are $P_1,\dots,P_n$. Then the induced action on $M_1\times\dots\times M_n$ has Hofer norm polytope $P_1+\dots+P_n$.
\end{prop}

\begin{rem}
Proposition \ref{sumahoferpoly} is also valid in the case of the one-sided Hofer Finsler norm (Remark \ref{semiads}), with polytopes $P^+_1,\dots,P^+_n$ and $P^+_1+\dots+P^+_n$.
\end{rem}

\smallskip

\subsubsection{Faces of balls of $\Ad$-invariant norms}

In this section we will use Lewis' results on convex analysis of $\Ad$-invariant functions and their restriction to Cartan subalgebras \cite{lewis} to describe the faces of the ball of $\Ad$-invariant norms. 

\smallskip

Since Lewis' results are stated in terms of subgradients of gauge function we provide next the correspondence between the faces of the balls and the gradients of the gauge functions: \textit{the subdifferential of a gauge $g_B$ at a unit norm $x$ corresponds via $y\mapsto \varphi_y$ to the supporting functionals of the polar of the ball at $x$}. 

Recall that the subdifferential of a convex function $f:V\to\R$ at $x_0\in V$ is defined as 
$$
\partial f(x_0)=\{y\in V : f(x)-f(x_0)\geq \langle x-x_0,y\rangle\qquad \forall\, x\in V\}.
$$
Each such $y$ which satisfies the condition stated above is called a \textit{subgradient}. It defines a supporting hyperplane to the graph of $f$ at $(x_0,f(x_0))$.  

\begin{prop}\label{gradface}
Let $B\subseteq V$ be a convex body containing $0\in V$ and let $g_B$ be the gauge function associated to $B$. For a non-zero $x\in V$
$$\partial g_B(x)=F_x(B^\circ)\quad \textrm{ and }\quad \partial g_{B^\circ}(x)=F_x(B)
$$
i.e. $\partial g_B(x)$ is the exposed face of $B^\circ$ defined by $x$. 
\end{prop}
\begin{proof}
If $x\neq 0$, we first claim that
$$
\partial g_B(x)=\{y\in V:\langle y,x\rangle=g_B(x),g_{B^\circ}(y)= 1\}.
$$
This is Corollary 23.5.3 in \cite{rock}, we give a direct proof for the convenience of the reader. Note that since $g_B=h_{B^{\circ}}$ (Proposition \ref{hofernorms}), $\langle y,x\rangle=g_B(x)$ and $g_{B^{\circ}}(y)= 1$ imply $\langle x,y\rangle=\sup_{w\in B^{\circ}}\langle x, w\rangle=g_B(x)$ and $\langle y,z\rangle\le g_B(z) \quad\forall z\in V$ respectively. Thus for such $y\in V$ and any $z\in V$, we have $\langle z,y\rangle- \langle x,y\rangle \le g_B(z)-g_B(x)$ which shows that $y\in \partial g_B(x)$. Reciprocally, if $y\in \partial g_B(x)$, then replacing with $z=0$ and $z=2x$ in
$$
\sup_{w\in B^{\circ}}\langle w,z\rangle - \sup_{w\in B^{\circ}}\langle w,x\rangle=  g_B(z)-g_B(x)\ge \langle z,y\rangle -\langle x,y\rangle,
$$
we obtain $g_B(x)=\langle x,y\rangle$; in particular $g_{B^{\circ}}(y)\ge 1$. But then from $g_B(z)\ge \langle z,y\rangle$ for all $z$ it also follow that $g_{B^{\circ}}(y)\le 1$, thus proving the claim. 

Therefore, noting that $x/{g_B(x)}\in\bd B$
\begin{align*}
\partial g_B(x)&=\{y\in V:\langle y,x\rangle=g_B(x),g_{B^\circ}(y)= 1\}\\
&=\{y\in V:\langle y,x\rangle=g_B(x),y\in\bd B^\circ\}\\
&=\{y\in V:\langle y,x/{g_B(x)}\rangle=1,y\in\bd B^\circ\}\\
&=\{y\in V:{H_{x/{g_B(x)}}}\mbox{  supports  } B^\circ\mbox{  at  }y\}&\mbox{ by Theorem \ref{supportdual}} \\
&=F_{x/{g_B(x)}}(B^\circ)=F_{x}(B^\circ).&\mbox{ by Definition \ref{defih}}
\end{align*}
The other identity is immediate from the polar duality.
\end{proof}

The next theorem relates the subdifferential of the gauges to the subdifferentials of the gauges restricted to Cartan subalgebras $\t$. It's proof can be found in \cite[Theorem 3.5]{lewis}. We define the map $\delta:\k\to\t^+$ by $\O_x\cap\t^+=\{\delta(x)\}$, that is, $\delta(x)$ is the unique element in the positive Weyl chamber $\t^+$ that is $\Ad$-conjugated to $x$.

\begin{thm}[Lewis]\label{lewis}
Let $K$ be a semi-simple compact group and let $g_\k:\k\to\R_+$ be an $\Ad$-invariant gauge function. Then $y\in\partial g_\k(x)$ if and only if $\delta(y)\in\partial g_\t(\delta(x))$ and there is  $u\in K$ such that $\Ad_u(x),\Ad_u(y)\in\t_+$. 
\end{thm}

If $\g=\k\oplus\p$ is the Cartan decomposition of a semi-simple Lie algebra $\g$, then Lewis' theorem was stated for an $\Ad$ invariant gauge on $\p$. We can take the complexified Lie algebra $\g=\k\oplus i\k$ and apply the theorem to $\p=i\k$. In the case $K=\SU(n)$, the matrices $x$ and $y$ such that there is  $u\in K$ with $\Ad_u(x),\Ad_u(y)\in\t_+$ are said to be \textit{simultaneously ordered diagonalizable}.

\medskip 

We now restate Lewis' theorem in a form convenient to its application in Section \ref{geodesicommute}; there is a related formulation due to Lewis in \cite{lewis2} for the orthogonal group of matrices.

\begin{thm}\label{lewisfaces}
Let $K$ be a semi-simple compact group, let $B$ be an $\Ad$-invariant convex body in $\k$ containing $0$, and let $\t$ be a Cartan algebra in $\k$ and $\t_+$ a positive Weyl chamber. Then for non-zero $x\in\t_+$  
$$F_x(B)=\Ad_{Z(x)}(F_x(B\cap\t)\cap\t_+),$$
and for general $x'\in\k$, if $v\in K$ is such that $\Ad_v(x')\in\t_+$, then
$$F_{x'}(B)=\Ad_{v^{-1}}\Ad_{Z(\Ad_v(x'))}(F_{\Ad_v(x')}(B\cap\t)\cap\t_+).$$
\end{thm}

\begin{proof}
We restate Theorem \ref{lewis} in the case that $x\in\t_+$. Let $g=g_B$ be the gauge function, see Definition \ref{minkowskigauge}. The theorem states that $y\in\partial g(x)$ if and only if $\gamma(y)\in\partial g_\t(x)$ and for some $u\in K$ we have $\Ad_u(x),\Ad_u(y)\in\t_+$. Observe that  $\Ad_u(x)\in\t_+$ for $u\in K$ if and only if $u\in Z(x)$, so the condition $\Ad_u(y)\in \t_+$ reduces to $y\in\Ad_{Z(x)}\t_+$. The conditions $\gamma(y)\in\partial g_\t(x)$ and $y\in\Ad_{Z(x)}\t_+$ are equivalent to $y\in\Ad_{Z(x)}(\partial g_\t(x)\cap\t_+)$. Hence
$$
\partial g(x)=\Ad_{Z(x)}(\partial g_\t(x)\cap\t_+).
$$
For a general $x'\in\k$ take $v\in K$ such that $\Ad_v(x')\in\t_+$. Then, since the adjoint action is isometric for the norm $g$ we get
$$\partial g(x')=\Ad_{v^{-1}}\Ad_{Z(\Ad_v(x'))}(\partial g_\t(\Ad_v(x'))\cap\t_+).$$
In the case that $g=g_B$ and $g_{\t}=g_{B\cap\t}$ we have
$$F_x(B^\circ)=\partial_B(x)=\Ad_{Z(x)}(\partial g_{B\cap\t}(x)\cap\t_+)=\Ad_{Z(x)}(F_x(B^\circ\cap\t)\cap\t_+),$$
where we used $F_x(B^\circ)=\partial_B(x)$ and $F_x(B^\circ\cap\t)=\partial g_{B\cap\t}$ (this follows from Proposition \ref{gradface}). If we take polars, by Remark \ref{projpolar} and the bipolar property, we obtain the required statement.
\end{proof}

If $B=\conv(E-E)^\circ$, then $B\cap\t=\conv(E-E)^\circ\cap\t=(\conv(E-E)\cap\t)^\circ=P^\circ$, so we get the following:

\begin{cor}
If $B=\conv(E-E)^\circ$ is the unit ball of Hofer's norm  and $P=\conv(E-E)\cap\t$ is Hofer's norm polytope then for non-zero $x\in\t_+$
$$F_x(B)=\Ad_{Z(x)}(F_x(P^\circ)\cap\t_+).$$
\end{cor}

The same characterization holds for the one-sided Hofer norm and its polytope $P^+$.

\section{The geometry of groups with bi-invariant Finsler metrics}\label{l}

In this section we focus our study of geodesics of Lie groups with Finsler metrics obtained by (left or right) translation of $\Ad$-invariant Finsler norms. We want to emphasize that in this section, we only require that norms are positively homogeneous,  i.e. $\|\lambda x\|=\lambda \|x\|$ for $\lambda \ge 0$, thus including our one-sided Hofer norms. Therefore the distance obtained is not necessarily symmetric.

\bigskip 

We need to begin this section with some remarks concerning compact semi-simple Lie algebras and their root decomposition.

\begin{rem}\label{milnor}A connected finite dimensional Lie group $K$ admits a bi-invariant \textit{Riemannian} metric if and only if it is isomorphic to the cartesian product of a compact group and an additive vector group (see \cite[Lemma 7.5]{milnor}). Therefore in that case $\k\simeq \k_0\oplus \mathfrak a$ where $a$ is an abelian sub-algebra and $\k_0$ is semi-simple. To simplify the discussion, we will only consider here compact groups $K$ with semisimple Lie algebra $\k$. Thus $K$ is always unimodular, in particular $\tr\ad v=0$ for any $v\in \k$. The $\Ad$-invariant inner product is given by (minus) the Killing form of $\k$, as before, and for this $\Ad$-invariant inner product, the operator $\ad v$ is skew-adjoint for any $v\in \k$ (\cite[Lemmas 6.3 \& 7.2]{milnor}). 
\end{rem}

What follows, in the form of a Remark, is in fact a recollection of useful facts that we will use about the real root decomposition of a real semi-simple Lie algebra. It will also help to fix and clarify the notation.

\begin{rem}[Real root decomposition]\label{rootdecom} Let $\Delta$ be the set of (real) roots of $\k$ with respect to a fixed Cartan subalgebra $\t $, and denote $\Delta_+$ the positive roots. There is an orthonormal set (the real root vectors)
$$
\{u_\alpha,v_\alpha:\alpha \in\Delta_+\}\subseteq \k
$$
with respect to this inner product, such that for each $h\in\t$ 
\begin{align}\label{weights}
[h,u_\alpha] & =-\alpha(h)v_\alpha \qquad [h,v_\alpha]=\alpha(h)u_\alpha \qquad [u_\alpha,v_\alpha]= h_\alpha,
\end{align}
where $h_\alpha\in \t$ is the unique element such that $\langle h_\alpha,\cdot\rangle=-\alpha( \cdot )$ (see for instance \cite[Chapter 6]{knapp}). Set 
$$Z_\alpha=\R u_\alpha\oplus\R v_\alpha.$$
Then $\k=\t\oplus\bigoplus_{\alpha\in\Delta_+}Z_\alpha$, and moreover for each $h\in \t$, we have
\begin{equation}\label{diago}
\ad h=i\sum_{\alpha\in\Delta_+} \alpha(h)(u_{\alpha}\otimes v_{\alpha}-v_{\alpha}\otimes u_{\alpha})=i \sum_{\alpha\in\Delta_+} \alpha(h) T_{\alpha}
\end{equation}
were we write $T_{\alpha}=(u_{\alpha}\otimes v_{\alpha}-v_{\alpha}\otimes u_{\alpha})$ for short. Note that $T_{\alpha}T_{\beta}=0$ when $\alpha\ne \beta$.
\end{rem}

\begin{rem}\label{perm}
The Weyl group of $K$ acts transitively on the roots (see the remark in \cite[p. 47]{serre}). Thus, if $\sigma$ is a permutation of the roots, there exists $k_{\sigma}\in K$ such that $\Ad_{k_{\sigma}}u_{\alpha}=u_{\sigma \alpha}$ and likewise with $v_{\alpha}$. Let $w\in \t$, and assume that $T_{\sigma}$ is the linear transform in $\k$ obtained by permuting the eigenvalues of $\ad w$, then
\begin{align*}
T_{\sigma} & = i\sum_{\alpha\in\Delta_+} \sigma\alpha(w)(u_{\alpha}\otimes v_{\alpha}-v_{\alpha}\otimes u_{\alpha})=  i\sum_{\alpha\in\Delta_+} \alpha(w)(u_{\sigma \alpha}\otimes v_{\sigma \alpha}-v_{\sigma\alpha}\otimes u_{\sigma\alpha})\\
& = \Ad_{k_{\sigma}} i\sum_{\alpha\in\Delta_+} \alpha(w)(u_{\alpha}\otimes v_{\alpha}-v_{\alpha}\otimes u_{\alpha})\Ad_{k\sigma}^{-1}\\
&= \Ad_{k_{\sigma}}\ad w \Ad_{k_{\sigma}}^{-1}=\ad(\Ad_{k_{{\sigma}}}w).
\end{align*}
Then $T_{\sigma}=\ad(\Ad_{k_{{\sigma}}}w)$. It is clear that $T_{\sigma}$ commutes with $\ad w$, and since $\k$ is semi-simple, $\Ad_{k_{\sigma}}w$ commutes with $w$ and in particular $w_{\sigma}=\Ad_{k_{{\sigma}}}w\in \t$. 
\end{rem}

\subsection{Geodesics of groups with bi-invariant Finsler metrics}\label{geodebi}

First we recall here some useful facts regarding norming functionals in $\k$, their proofs are quite elementary and can be found in \cite[Section 4]{lar19}. 

\begin{rem}[Gauss' Lemma]\label{gausslem}
Let $\k$ be the Lie algebra of $K$ equipped with an $\Ad$-invariant norm (or Finsler norm).  As mentioned in the introduction, it is relevant to remark here that the norm does not need to be homogeneous, it is only necessary that it is positively homogeneous. For $v\in\k$ let $\varphi$ be a norming functional of $v$ or $-v$. Then
\begin{enumerate}
\item We have $\varphi([w,v])=0$ for all $w\in\k$. Equivalently, $\varphi\circ \ad v\equiv 0$ on $\k$.
\item $\varphi(e^{\lambda \ad v}w)=\varphi(w)$ for all $w\in\k$ and $\lambda\in\R$.
\item $\varphi(e^{-v}D\exp_v w)=\varphi(w)$ for all $w\in\k$.
\end{enumerate}
\end{rem}

When the norm $\| \cdot \|$ is a norm derived from an inner product, the third item above is in fact Gauss' Lemma of Riemannian geometry (it is well-known that the Riemannian exponential map, for a bi-invariant Riemannian metric in a Lie group, coincides with the group exponential). We next show how to describe the boundary of a ball in $K$. We first recall a definition.

\begin{defn}
Let $B$ be a convex body in a vector space $V$ and let $v\in\bd B$. The \textit{solid tangent cone} at $v$ is 
$$TC_v:=\bigcap\{H_v^-:H_v \mbox{ is a supporting hyperplane of }B\mbox{ at }v\}.$$
\end{defn}
If $v\in (V,\|\cdot\|)$ and $\|v\|=r>0$, then taking $B=B_r(0)$ it is apparent that
$$
TC_v:=\bigcap\{\varphi^{-1}(-\infty,r]: \varphi \mbox{ is a norming functional of }v\}.
$$

Therefore, by Gauss' Lemma above, for each $0\ne v\in \k$, if $w\in TC_v$ then $ e^{-v}D\exp_v w\in TC_v$. Moreover, if $v$ is such that the differential of the exponential map is invertible (for instance, if $v$ is smaller that the injectivity radius of $K$, see Definition \ref{injrad} below), then it is clear that
$$
D\exp_v(TC_v)=e^vTC_v.
$$
Geometrically, the differential of the exponential map at $v$ acts on the tangent cone at $v$ as the left translation.

\begin{rem}\label{normicon}
 Any functional $\varphi$ can be described as $\varphi(v)=\langle v,a\rangle$ for some $a\in \k$, via the inner product in $\k$ (Remark \ref{milnor}). That is $\varphi=-\tr(\ad a\;\cdot\;)$.  If $\varphi(z)=\|z\|$ for some $z\in \k$, then $\varphi([v,z])=0$ for all $v\in\k$ (Remark \ref{gausslem}). Thus for any $v\in \k$,
$$
0=\varphi([v,z])=\langle [v,z],a\rangle =\langle v,[z,a]\rangle
$$
and then $[z,a]=0$. Therefore $\ad a$ and $\ad z$ commute, and they are simultaneously diagonalizable. Let $\{e_i\}_{i=1,\dots,N}$ be an orthonormal basis of $\k^{\mathbb C}$ that diagonalizes both simultaneously, let $q_{i,j}=e_i\otimes e_j$ and $q_i=q_{i,i}$ the corresponding rank-one orthogonal projections. If $\ad z=i\sum_{j\in J} z_j q_j$ is the spectral decomposition of the skew-adjoint operator $\ad z$, we can write
$$
A=\ad a=i\sum_{j\in J} a_j q_j +iB
$$
for certain $a_j\in \mathbb R$ and $B^*=B=\sum_{k,l\notin J}B_{kl}q_{kl}$, with $Bq_j=0$ for all $j\in J$. 
\end{rem}

\begin{lem}[Supporting norming functionals]\label{normicon2}
If $\varphi$ norms $z\in \k$ as in the previous remark, and $\|\varphi\|=1$, drop the term $B$ and all the $a_j=0$, and consider  $\psi=-\tr(\tilde{A}\quad \cdot)$, where $\tilde{A}=i\sum_j a_jq_j$. Then $\psi$ is still norming for $z$ and has unit norm. 
\end{lem}
\begin{proof}Note first that $\psi$ is still norming for $z$:
$$
\psi(z)=\sum_j a_j z_j\tr(q_j)=\tr(A\circ\ad z)=\varphi(z)=\|z\|.
$$
Now for any $x\in \k$, write $\ad x$ as a block matrix in terms of the $q_j$ and its orthogonal complement,
$$
A=\ad a=\left( 
\begin{array}{cc} ia_j & 0 \\ 
0 &  B 
\end{array} \right)\qquad 
\tilde{A}=\left( 
\begin{array}{cc} ia_j & 0 \\ 
0 &  0
\end{array} \right)\qquad 
\ad x=\left( 
\begin{array}{cc} x_{ij} & \ast \\ 
\ast&  x_0 
\end{array} \right)
$$
where $x_0q_j=0$ for all $j$. Then
\begin{align*}
1 & =\|\varphi\|=\sup\limits_{\|x\|=1} \varphi(x)=\sup\limits_{\|x\|=1} -\tr(\ad a\circ\ad x)= \sup\limits_{\|x\|=1} \sum_j a_j x_{jj}+ \tr(Bx_0)\\
& \ge \sup\limits_{\|x\|=1, \, x_0=0} \sum_j a_j x_{jj}+ \tr(Bx_0)= \sup\limits_{\|x\|=1, \, x_0=0} \sum_j a_j x_{jj}\\
& = \sup\limits_{\|x\|=1} \sum_j a_j x_{jj} = \sup\limits_{\|x\|=1} -\tr(\tilde{A}\circ\ad x) =   \sup\limits_{\|x\|=1}\psi(x)=\|\psi\|.
\end{align*}
This proves that $\|\psi\|\le 1$, but since $\psi(z)=\|z\|$, it must be $\|\psi\|=1$. 
\end{proof}

We now charachterize the linear order given by Finsler $\Ad$-invariant norms; results in the same vein can be found in \cite[Section 12]{atiyah}, \cite[Proposition 6]{bhatia} and \cite[Proposition 2.8]{tam}. For $z,w\in \k$ we denote $\overrightarrow{w}=(w_1,w_2,\dots,w_N)$ the string of real numbers such that $i w_j$ are the eigenvalues of $\ad w$ in the complexification $\k^{\mathbb C}$ of $\k$, and likewise with $\overrightarrow{z}$.

\begin{prop}[Majorization and norms]\label{majo}
Let $z,w\in\k$, let $N=\dim(\k)$. The following are equivalent:
\begin{enumerate}
\item $z\in \conv\O_w$, more precisely there exist (at most) $N+1$ points $k_i\in K$ and $N+1$ real numbers $\lambda_i \ge 0$ with $\sum_i \lambda_i= 1$ such that
$$
z=\sum_{i=1}^{N+1} \lambda_i\, \Ad_{k_i}w. 
$$
\item $\overrightarrow{z}\prec\overrightarrow{w}$ (strong majorization).
\item $\|z\|\le \|w\|$ for all $\Ad$-invariant Finsler norms in $\k$.
\item $\max_{k\in K} \langle z, \Ad_k x\rangle\le \max_{k\in K} \langle w, \Ad_k x\rangle$ for all $x\in \k$.  
\end{enumerate}
If moreover equality holds for some Finsler norm, then $z$ and all the $\Ad_{k_i}w$ lie in the same face of the ball for that norm (and in fact lie in the intersection of all the faces that $z$ lies in). If that norm is strictly  convex then $z=\Ad_k w$ for some $k\in K$.
\end{prop}
\begin{proof}
The fact that any element in the orbit can be written with a prescribed number of combinations ($N+1$) is a consequence of Caratheodory's theorem. Let us establish first the equivalence $(1)\Leftrightarrow (2)$. Assume $(1)$, passing to the adjoint representation we have 
$$
\ad z= \sum_{i=1}^N \lambda_i \Ad_{k_i} \ad w \Ad_{k_i}^{-1}.
$$
Each $\Ad_{k_i}$ is a unitary operator acting in $\k^{\mathbb C}$ therefore $\ad z$ belongs to the convex hull of the coadjoint orbit of $\ad w$, and this in turn (and by Schur-Horn's theorem) implies that strong majorization $\overrightarrow{z}\prec \overrightarrow{w}$ holds. Now assume $(2)$ holds, let $\t$ be a Cartan subalgebra containing $w$, let $\Delta^+$ be the positive simple roots and let $k\in K$ such that $\Ad_k z\in \t$. The spectrum of 
$$
\ad(\Ad_k z)=\Ad_k \ad z \Ad_k^{-1} 
$$
is also the string $\overrightarrow{z}$, and the assumption implies $\overrightarrow{z}$ is in the convex hull of the permutations of the eigenvalues of $\ad w$  (see \cite[Theorem  II.1.10]{bhatia2}). Equivalently (and again invoking Caratheodory's theorem), there are $N+1$ such elements with
$$
\ad (\Ad_k z)=\sum_{i=1}^{N+1} \lambda_i T_{\sigma_i}
$$
for a certain string of non-negative numbers $(\lambda_i)_{i=1,\dots , N+1}$ such that $\sum_i\lambda_i=1$. Now each $T_{\sigma_i}$ is obtained permuting the eigenvalues of $\ad w$ or equivalently, permuting the roots $\alpha\in\Delta_+$. By Remark \ref{perm}, any such permutation is obtained by an inner automorphism, therefore 
$$
\ad (\Ad_k z)=\sum_{i=1}^{N+1}\lambda_i \ad (\Ad_{h_i}w)=\ad(\sum_{i=1}^{N+1}\lambda_i  \Ad_{h_i}w),
$$
and by the semi-simplicity of $\k$, we obtain $z=\sum_{i=1}^{N+1}\lambda_i  \Ad_{k_i}w \in \conv \O_w$, where $k_i=k^{-1}h_i$. 

Clearly $(1)\Rightarrow (3)$, and $(3)\Rightarrow (4)$ when $x\in\k$ is a regular element (so we obtain a non-degenerate Finsler norm, see Remarks \ref{semiads} and \ref{onesided}). Since regular elements are dense in $\k$, it is straightforward to see that $(4)$ must then hold for any $x\in \k$, if $(3)$ holds. Now assume that $(1)$ does not hold, then by Hahn-Banach's theorem there exists a linear functional $\varphi$ separating $z$ and the convex capsule of the orbit, i.e.
$$
\varphi (y) \le r<\varphi(z) \qquad \forall\, y\in \conv(\O_w).
$$
Let $x\in \k$ such that $\varphi=\langle x,\,\cdot\rangle$, then by Remark \ref{semiads}
$$
\max_{k\in K}\langle w, \Ad_k x\rangle=\max_{k\in K}\langle\Ad_k w,x\rangle =\max\{\langle y , x\rangle:  y\in \conv(\O_w)\}\le r<\langle z,x\rangle.
$$
This shows that $(4)$ cannot hold, finishing the proof of the equivalences.  Now assume that any of the conditions hold, and we have equality of norms $\|z\|=\|w\|$ holds for some norm, then
$$
\|w\|=\|z\|=\|\sum \lambda_i \Ad_{k_i}w\|\le \sum \lambda_i \|\Ad_{k_i}w\|=\sum \lambda_i \|w\|\le \|w\|.
$$
We note that there is a common norming functional $\varphi$ for all the $\lambda_i \Ad_{k_i}w$ by Lemma \ref{mismacara}; since $\lambda_i\ge 0$ this amounts for the $\Ad_{k_i}w$ in the same face of a sphere of the norm. Then also
$$
\varphi(z)=\varphi (\sum_i \lambda_i \Ad_{k_i}w)=\sum \lambda_i \varphi(\Ad_{k_i}w)=\sum \lambda_i \|w\|=\|w\|=\|z\|.
$$
If $\varphi$ is any functional norming $z$, then 
$$
\|z\|=\varphi(z)=\sum \lambda_i \varphi(\Ad_{k_i}w)\le \sum \lambda_i \|\Ad_{k_i}w\|=\|w\|=\|z\|
$$
shows that it must be $\varphi(\Ad_{k_i}w)=\|\Ad_{k_i}w\|$ for all $i$, thus these vectors are in fact in the intersection of all the faces where $z$ lives. If the norm is strictly convex all the $\Ad_{k_i}w$ are aligned, but being normed by the same functional they must be equal therefore $z=\Ad_k w$ as claimed.
\end{proof}

\begin{rem}
It suffices to check condition $(3)$ above only for \textit{strictly convex norms} to obtain the equivalences. This is because any $\Ad$-invariant Finsler norm  $\|\cdot\|$ can be approximated explicitly with a strictly convex ($\Ad$-invariant, Finsler) norm by means of
$$
\|x\|_{\varepsilon}=\|x\|+\varepsilon \|x\|_2.
$$
Here $\|x\|_2^2=-\tr(\ad x\circ \ad x)$ is the norm derived from the Killing form of $\k$. Likewise, it suffices to check $(4)$ for regular $x\in \k$.
\end{rem}

\begin{rem}[One-sided norms]\label{onesided}
Since $K$ is compact it is unimodular and then $0=\sum_j w_j=i\tr(\ad w)$ (Remark \ref{milnor}). This easily implies that all the partial sums $\sum_{k=1}^m w_k$, with the $w_k$ rearranged in decreasing order, must be non-negative. Thus the vector $\overrightarrow{w}$ strongly majorizes the zero vector in $\mathbb R^N$, i.e. $\overrightarrow{0}\prec \overrightarrow{w}$, and by the previous proposition,  $0\in \conv(\O_w)$ for any $w\in \k$. In particular the one-sided Hofer norms (Remark \ref{semiads}) are in fact-nonnegative, regardless their degeneracy or non-degeneracy.  Assumming that the orbit $\O_w$ is full, then $0$ must be an interior point of $\conv(\O_w)$, and then we obtain a true Finlser norm. This can be seen using the argument in   \cite[Lemma 6]{bgh}: if $0$ is in the boundary of $\conv(\O_w)$, by Hahn-Banach's separation theorem, there exists $0\ne x\in\k$ such that $\varphi_x(0)=0$ and $\varphi_x(\conv(\O_w))\geq 0$. But then it must be $\varphi_x(\O_w)= 0$ because otherwise
$$
\int_{k\in K}\varphi_x(\Ad_k(w))dk=\varphi_x(\int_{k\in K}\Ad_k(w)dk)  >0
$$
contradicting that $\int_{k\in K}\Ad_k(w)dk$ is a fixed point of the adjoint action, therefore it is $0$ because $\k$ is semi-simple. Since $\varphi_x(\O_w)= 0$, the orbit is not full.

\end{rem}

\subsubsection{Domain of injectivity of the exponential}

Since $K$ is a finite-dimensional Lie group, the exponential map of $K$ is a local diffeomorphism for some open ball of the norm $\|\cdot \|$. More precisely, let $\mathbf D\subseteq \k$ be a maximal open convex $\Ad$-invariant set, such that $V=\exp(\mathbf D)$ is open in $\k$ and $\exp:{\mathbf D}\to V$ is a diffeomorphism. It will be convenient to denote $\|\cdot\|_{\infty}$ to the Minkowski gauge of the set $\mathbf D$; it is an $\Ad$-invariant Finsler norm in $\k$ and we will refer to it as the \textit{uniform norm}. Then we also define

\begin{defn}\label{injrad}
Let $K$ be a Lie group, $\| \cdot \|$ an $\Ad$-invariant Finsler norm on $\k$, and for $R>0$ let $B_R=\{v\in\k:\|v\|<R\}$, $V_R=\exp(B_R)$. If $\exp:B_R\to V_R$ is a diffeomorphism between open sets and $R$ is maximal, we call $R$ the \textit{radius of injectivity} for the given norm. 
\end{defn}

Note that if $B_R\subseteq \mathbf D$ then $\exp:B_R\to V_R$ is a diffeomorphism between open sets, thus one looks for balls of the given norm that fit inside the domain of injectivity of the exponential map.

\smallskip

\begin{rem}The condition $B_R\subseteq \mathbf D$ is equivalent to $B_R\cap\t\subseteq \mathbf D\cap\t$. In the case of Hofer norms $B_R\cap\t=R\{y_1,\dots,y_m\}^\circ$, where $\{y_1,\dots,y_m\}$ are the extreme points of the Hofer norm polytope. Note also that  $\mathbf D\cap\t$ can be taken as the interior of $\cup_{w\in\W}w.C$, where $C$ is a Weyl alcove.
\end{rem}

\begin{ex}\label{espectro}
For the group $\SU(n)$ we take $\mathbf D=\{z\in {\mathfrak su}_n: \|z\|_{\infty}<\pi\}$ where $\|\cdot\|_{\infty}$ now is exactly the usual spectral norm. Equivalently, $\mathbf D\cap\t$ is 
$$
\{\diag(x_1,\dots,x_n)\in\t:|x_i|<\pi,\mbox{ for }i=1,\dots,n\}.
$$
For the group $\SU(n)/\Z_n$ we take 
$$
\mathbf D\cap\t=\{\diag(x_1,\dots,x_n)\in\t:|x_i|<\pi/n,\mbox{ for }i=1,\dots,n\}.
$$
\end{ex}

\bigskip

We now show that the convex body $\mathbf D$ (which depends only on $K$) is optimal in terms of lengths of segments (one-parameter subgroups), for any bi-invariant distance:

\begin{thm}[Exponential map and Finsler norms]\label{expono} Let $z,w\in \k$ such that $e^z=e^w$. Assume that $z\in \mathbf D/2$, Then $w$ commutes with $z$ and $\|z\|\le \|w\|$ for any $\Ad$-invariant Finsler norm $\|\cdot\|$ in $\k$. If equality of norms holds for a strictly convex norm, then $w=z$.
\end{thm}
\begin{proof}
Assume first that $w$ is regular, let $\mathfrak z(w)=\t$ denote the Cartan subalgebra. Note that 
$$
\exp(e^{t \ad w}z)=\exp(\Ad_{e^{tw}}z)=e^{tw}e^ze^{-tw}=e^{tw}e^w e^{-tw}=e^w.
$$
Therefore, differentiating at $t=0$ we obtain $D\exp_z([w,z])=0$, and since $z\in \mathbf D$, we conclude that $[w,z]=0$, and $z\in \t$. Since $e^{\ad w}=e^{\ad z}$, then $\exp(\ad(w-z))=\exp(\ad w-\ad z)=1$, implying that $\sigma(\ad(w-z))\subseteq 2\pi i \mathbb Z$. This implies,  using equation (\ref{diago}), that we can write
$$
i\sum_{\alpha\in \Delta_+} w_{\alpha} T_{\alpha} =\ad w=i\sum_{\alpha\in \Delta_+} (z_{\alpha}+ 2\pi n_{\alpha})T_{\alpha}
$$
wiht $w_{\alpha},z_{\alpha}\in\mathbb R$ and $n_{\alpha}\in \mathbb Z$. It will be convenient to number the roots, so we let $J=card(\Delta_+)$ and we have $w_j=z_j+2\pi n_j$ for all $j\in J$, where some of the $z_j$ might be zero. Note that since $z\in \mathbf{D}/2$, then $2z\in \mathbf{D}$ therefore the exponential map is injective and a diffeomorphism in $t2z$ for $t\in [-1,1]$, and in particular it must be $\sigma(\ad 2z)\subseteq (-2\pi i,2\pi i)$ (by inspection of the formula of the differential of the exponential map, see Remark \ref{DAD} below). Thus  we have $|2z_j|<2\pi$ for all $j\in J$ or equivalently $-\pi<z_j<\pi$. We can asumme that the $z_j$ are given in order $z_1\ge z_2\ge \dots \ge z_J$ (recall also $\sum z_j=0)$. Let us reorder the $w_j$ in decreasing order also, so there is a permutation $\sigma$ of $\{1,\dots, J\}$ such that if $v_j=w_{\sigma j}$ then $v_1\ge v_2\ge \dots \ge v_J$ (and we also have $\sum v_j=0$). From here it is also clear that $\sum n_j=0$. Let's spare for a moment those $j$ such that $n_j=0$, and for the others, note that if $n_j>0$ and $n_r<0$ then
$$
z_j+2\pi n_j> -\pi+2\pi=\pi >-\pi =-2\pi +\pi > 2\pi n_r+z_r.
$$
This shows that the $v_j=z_{\sigma j}+2\pi n_{\sigma j}$, with positive $n_{\sigma j}$, are always bigger than those with  negative $n_{\sigma j}$.  We split the indices in two sets: let $j_0$ be such that if $j\in {1,\dots, j_0}$ then $n_{\sigma j}>0$ and otherwise $n_{\sigma j}<0$ when $j_0+1\le j\le J$. We compute the sum of the $z_j$ up to any such $1\le j\le j_0$, we have 
$$
\sum_{k=1}^j z_k \le \sum_{k=1}^j  \pi \le j\pi.
$$
On the other hand, it is clear that the sum of the first $j$ bigger $v_k$, for $k\le j$, must be of those $v_k$ with $n_{\sigma k}>0$, therefore
$$
\sum_{k=1}^j v_j=\sum_{k=1}^j z_{\sigma k}+ 2\pi n_{\sigma k}\ge -j\pi+2\pi j=j\pi.
$$
Thus if $j\le j_0$,
\begin{equation}\label{mayorizando}
z_1+z_2+\dots + z_j\le v_1+ v_2+\dots +v_j.
\end{equation}
Now assume that $j\ge j_0+1$, and note that
$$
\sum_{k=1}^j z_k=-\sum_{k=j+1}^N z_k< \pi( N-(j+1)).
$$
Likewise
$$
\sum_{k=1}^j v_k=-\sum_{k=j+1}^N v_k =-\sum_{k=j+1}^N (z_{\sigma k}+ 2\pi n_{\sigma k})
$$
and note that now all the $n_{\sigma k}\le -1$, since $k\ge j+1\ge j_0+1>j_0$. Therefore
$$
\sum_{k=1}^j v_k \ge -\pi(N-(j+1))+2\pi(N-(j+1))=\pi(N-(j+1)),
$$
and equation (\ref{mayorizando}) is also valid for $j\ge j_0$. Let $\overrightarrow{v}=(v_1,v_2,\dots,v_N)$ and likewise $\overrightarrow{z}=(z_1,z_2,\dots,z_N)$. Then equation (\ref{mayorizando}) together with $\sum z_j=\sum v_j=0$ tells us that $\overrightarrow{z}\prec \overrightarrow{v}$, that is $\overrightarrow{v}$ majorizes $\overrightarrow{z}$. Then by  \cite[Theorem  II.1.10]{bhatia2}, $\overrightarrow{z}$ is in the convex hull of all vectors obtained by permutating the coordinates of $\overrightarrow{v}$. Clearly, we can add those $w_j$ such that $w_j=z_j$ ($n_j=0$) and this still holds true. Since the $v_j$ are just a permutation of the $w_j$, then  $\overrightarrow{z}$ is in fact in the convex hull of all vectors obtained permutating the coordinates of $\overrightarrow{w}=(w_1,w_2,\dots,w_J)$. By Remark \ref{perm}, we have
$$
\ad z=\sum_{{\sigma}} \lambda_{{\sigma}} \ad(\Ad_{k_{\sigma}}w)=\ad (\sum_{\sigma} \lambda_{\sigma} \Ad_{k_{\sigma}}w),
$$
and since $\k$ is semi-simple, it must be $z=\sum \lambda_{\sigma} \Ad_{k_{\sigma}}w$. Then 
$$
\|z\|\le \sum \lambda_{\sigma} \|\Ad_{k_{\sigma}}w\|=\sum \lambda_{\sigma} \|w\|=\|w\|
$$
proving the claim for regular $w$. If $w$ is not regular, fix the bi-invariant distance in $K$ given by the uniform norm (the Minkowski Finsler norm of the convex set $\mathbf D$). For each $\varepsilon>0$ pick $w_{\varepsilon}\in\k$ such that $\|w_{\varepsilon}\|_{\infty}<\varepsilon$ and $w+w_{\varepsilon}$ is regular (regular elements are dense). Observe that 
$$
\dist(e^w,e^{w+w_{\varepsilon}})\le \|w_{\varepsilon}\|_{\infty}<\varepsilon
$$
by Theorem \ref{charactgeod}. Therefore there exists $y_{\varepsilon}\in \k$ such that $e^we^{y_{\varepsilon}}=e^{w+w_{\varepsilon}}$. Again, note that
$$
\|y_{\varepsilon}\|_{\infty}=\dist(1,e^{y_{\varepsilon}})=\dist(1,e^{-w}e^{w+w_{\varepsilon}})=\dist(e^w,e^{w+w_{\varepsilon}})\le \|w\|_{\infty}<\varepsilon.
$$
Now consider the map $f:v\mapsto e^{z+v}$. Since $f(0)=e^z$ and $Df_0=D\exp_z$, the hypotesis $z\in \mathbf{D}/2$ guarantees that $f$ is a local diffeomorphism from a $0$-neighbourhood to a neighbourhood of $e^z$. Therefore there exists a unique $z_{\varepsilon}\in \k$ in that neighbourhood, such that $e^{z+z_{\varepsilon}}=e^ze^{y_{\varepsilon}}$. Note also that when $y_{\varepsilon}\to 0$, then also $z_{\varepsilon}\to 0$. In particular, for small $\varepsilon>0$, $z+z_{\varepsilon}\in \mathbf{D}/2$ just like $z$. Then from
$$
e^{z+z_{\varepsilon}}=e^ze^{y_{\varepsilon}}=e^we^{y_{\varepsilon}}=e^{w+w_{\varepsilon}}.
$$
and the previous proof, we can conclude that $\|z+z_{\varepsilon}\|\le \|w+w_{\varepsilon}\|$ for any $\Ad$-invariant Finsler norm in $\k$. Letting $\varepsilon\to 0$ gives us the desired inequality $\|z\|\le \|w\|$.

Now assume that there is an equality of norms for a strictly convex norm, then  by Proposition \ref{majo}, $z=\Ad_k w$ for some $k\in K$, and in particular $w\in \mathbf D/2$ also. But $e^w=e^z$ and the injectivity of the exponential map implies $z=w$.
\end{proof}

\begin{rem}If $z,w$ are as in the previous theorem, then by Proposition \ref{majo}, we have
$$
 z=\sum_{i=1}^{N+1} \lambda_i\, \Ad_{k_i} w,
$$
for some $k_i\in K$, $\lambda_i\ge 0$ with $\sum_i\lambda_i=1$. We also mention here that the proof of the previous theorem shows that when $w$ is regular, the $k_i$ are in the Weyl group of $K$.
\end{rem}

For linear Lie groups such as $K=\SU(n)$, the passage to the adjoint representation is unnecessary. Thus from $e^z=e^w$ with $z\in \mathbf D$ we can conclude that $[w,z]=0$ and therefore
$$
w=z+ 2\pi i \sum_j n_j q_j.
$$
With the same proof as the previous theorem, we now obtain the same result for $z\in \mathbf D$, i.e. $\|z\|_{\infty}<\pi$ (and not just $z\in \mathbf D/2$):

\begin{cor}[Linear groups]\label{lingr}
Let  $K\subseteq M_n(\mathbb C)$ be compact semi-simple linear Lie group. Let $z,w\in \k$ such that $e^z=e^w$, and assume that $z\in \mathbf D$.  Then $\|z\|\le \|w\|$ for any $\Ad$-invariant Finsler norm $\|\cdot\|$ in $\k$, and if equality of norms holds for a strictly convex norm, then  $w=z$.
\end{cor}

\medskip

\subsubsection{Local Hopf-Rinow theorem and the characterization of geodesics}\label{1}

Let us take a look at  geodesics of a Lie group with a bi-invariant Finsler metric. We recall here the fundamental results about geodesics, for proofs see \cite[Section 4]{lar19}. For a given Finsler norm, let $R>0$ be an injectivity radius. The results are stated in terms of the left logarithmic derivatives $\gamma^{-1}_t\dot{\gamma}_t$ and hold also if stated in terms of the right logarithmic derivative since the norm is $\Ad$-invariant.

\begin{defn}[Short paths]
 We call a curve $\gamma:[0,1]\to K$ \textit{short} or we say that $\gamma$ is a  \textit{geodesic} if it minimizes the length functional
$$
L(\gamma)=\length(\gamma)=\int_a^b\|\dot{\gamma}_t\|_{\gamma_t}dt=\int_a^b\|\dot{\gamma}_t\gamma_t^{-1}\|dt=\int_a^b\|\gamma_t^{-1}\dot{\gamma}_t\|dt
$$
among all curves in $K$ with the same endpoints. 
\end{defn}

\begin{thm}[Geodesics]\label{charactgeod}
Let $u_0,u_1=u_0e^z\in K$ with $\|z\|< R$.
\begin{enumerate}
\item If $\delta(t)=u_0e^{tz}$, $t\in [0,1]$, then $\delta$ is shorter than any other piecewise $C^1$ path $\gamma$ in $K$ joining $u_0,u_1$ and $\dist(u_0,u_1)=\|z\|$.
\item If $v,w\in\k$ then 
$$\dist(e^v,e^w) \leq \|w-v\|$$
and if $w,v$ commute and $\|w-v\| \leq R$, then equality holds (this is known as the exponential metric decreasing property).
\item Let $\Gamma:[a,b]\to \k$ be a piecewise $C^1$ short path joining $0,z$, let $\gamma=e^{\Gamma}$. Then  $\|\gamma^{-1}_t\dot{\gamma}_t\|=\|\dot{\Gamma}_t\|$ for all $t$, and  $\gamma$ is short in $K$ with the same length than $\Gamma$. Moreover if $\varphi$ is norming functional for $z$, then
$$
\varphi(\gamma^{-1}_t\dot{\gamma}_t)=\|\gamma^{-1}_t\dot{\gamma}_t\|=\|\dot{\Gamma}_t\|=\varphi(\dot{\Gamma}_t)\quad \forall t\in [a,b],
$$
thus $\gamma^{-1}\dot{\gamma}$ (normalized) sits inside a face of the unit sphere of the norm.
\item $\gamma:[a,b]\to K$ is a piecewise $C^1$ short path joining $1,e^z$ in $K$ if and only if $\gamma=e^{\Gamma}$ for a piecewise $C^1$  path $\Gamma:[a,b]\to \k$ joining $0,z$ (with $\|\Gamma_t\|\le R$) and
$$
\varphi(\dot{\Gamma}_t)=\|\gamma_t^{-1}\dot{\gamma}_t\|=\varphi(\gamma_t^{-1}\dot{\gamma}_t) 
$$
for some unit norm functional $\varphi$ and all $t\in [a,b]$ (and then this holds for any norming functional of $z$).
\item If  $z/\|z\|$ is an extremal point of the unit sphere of $\k$, then the only short piecewise $C^1$ path joining $1,e^z$ in $K$ is (a reparametrization of) the segment $\delta(t)=e^{tz}$.
\end{enumerate}
\end{thm}

\bigskip

These results were established in \cite{lar19} with some generality; for finite dimensional groups we can improve the existence of short paths invoking the metric version of Hopf-Rinow's theorem. As usual, here $K$ denotes a connected compact Lie group with semi-simple Lie algebra $\k$. The last item of this theorem extends significanly (to this family of Lie groups) the results obtained in \cite{alv} for the group $U(n)$.

\begin{thm}\label{existen}
Let $\dist$ be a bi-invariant metric in $K$ (i.e. from an $\Ad$-invariant Finsler norm $\|\cdot\|$ in $\k$). Then
\begin{enumerate}
\item For each $u_1,u_2\in K$ there exists a short polygonal path $\delta$ joining them, i.e. a concatenation of segments $t\mapsto u_ie^{tz_i}$ such that $\|z_i\|<R$ and 
$$
L(\delta)=\sum \|z_i\|=\dist(u_1,u_2).
$$
\item If the norm is strictly convex, there exists $w\in \k$ such that $\|w\|=\dist(u_1,u_2)$ and the segment $t\mapsto u_1e^{tw}$ is a short path joining them. Any short path is a reparametrization of a segment, and if $\dist(u_1,u_2)<R$, there is exactly one short segment joining them.
\item If $z\in \mathbf D/2$ ($z\in \mathbf D$ for linear Lie groups), then $\dist(1,e^z)=\|z\|$. If the norm is strictly convex, $t\mapsto e^{tz}$ is the unique short path joining them. 
\end{enumerate}
\end{thm}
\begin{proof}
Let $\dist_g$ denote the bi-invariant distance induced by the $\Ad$-invariant metric given by the Killing form in $\k$. It is well-known that such $\dist_g$ metric has one-parameter groups as Riemannian geodesics, therefore it is geodesically complete. By Hopf-Rinow's theorem $(K,\dist_g)$ is metrically complete, and therefore $(K,\dist)$ is metrically complete since both metrics are uniformly equivalent (since both are bi-invariant and the tangent norms are uniformly equivalent, $\k$ being finite dimensional). Since any of these metrics induce the original topology of $K$, and since $K$ is  a finite dimensional manifold, it is locally compact and we can apply Cohn-Vossen's theorem \cite[Theorem 2.5.28]{burago} to the metric space $(K,\dist)$. This theorem tells us that any approximating sequence of paths in $K$ has a limit point $\gamma:[0,1]\to K$ such that $\gamma$ is rectifiable and $L(\gamma)=\dist(u_1,u_2)$. Partition $\gamma$ into finite pieces, in points denoted $\gamma_i=\gamma_{t_i}$, such that $\dist( \gamma_i,\gamma_{i+1})<R$. Write $\gamma_{i+1}=\gamma_ie^{z_i}$ using Theorem \ref{charactgeod}(1), hence if $\delta$ is the concatenation of these paths,
$$
L(\delta)=\sum \|z_i\|= \sum \dist(\gamma_i,\gamma_{i+1})=\dist(u_1,u_2)
$$
which shows that $\delta$ is minimizing. If the norm is strictly convex, it can be shown that the $z_i$ commute, hence we can replace the concatenation of segments with a segment; the proof of this and the local uniqueness can be found in \cite[Theorem 4.15]{lar19}. Now assume that $z$ is in (half of) the domain of injectivity of the exponential map. Assume first that the norm is strictly convex. By the previous item of this theorem, there exists $w\in \k$ such that $t\mapsto  e^{tw}$ joins $1,e^z$ and such that $\|w\|=\dist(1,e^z)\le \|z\|$. But Theorem \ref{expono} also tells us that $\|z\|\le \|w\|$, therefore $\|z\|=\|w\|=\dist(1,e^z)$. Now let $\|\cdot\|$ be any $\Ad$-invariant norm, let $\varepsilon >0$ and let $\|v\|_g=\sqrt{\langle v,v\rangle}$ be an $\Ad$-invariant Riemannian metric in $\k$. Consider
$$
|v|_{\varepsilon}=\|v\|+ \varepsilon \|v\|_g,
$$
and note that $|\cdot|_{\varepsilon}$ is $\Ad$-invariant and strictly convex. Therefore by what we just proved, if $z\in \mathbf D/2$ then 
$$
\|z\|\le |z|_{\varepsilon}=\dist_{\varepsilon}(1,e^z)\le L_{\varepsilon}(\gamma)=L(\gamma)+\varepsilon L_g(\gamma)
$$
for any piecewise smooth path $\gamma$ joining $1,e^z$ in $K$. Letting $\varepsilon\to 0^+$ first, and taking the infimum over the paths $\gamma$, shows that $\|z\|\le \dist(1,e^z)$ as claimed. If the norm is strictly convex, any other short path is also a segment $t\mapsto e^{tw}$ by the second item of this theorem. Thus $e^w=e^z$ and $\|w\|=\|z\|$, and by Theorem \ref{expono}, we conclude that $w=z$. For linear Lie groups  we can replace half of $\mathbf D$ with the full set $\mathbf D$ by Corollary \ref{lingr}.
\end{proof}

From the last assertion of the theorem we can give a nice characterization of the product of exponentials. This is connected with the noncommutative Horn inequalities as studied by Belkale et al, see \cite{belkale,entov} and the references therein.

\begin{cor}[Product of exponentials]
Let $x,y,z\in \k$ with $z\in\overline{\mathbf D}/2$ such that $e^xe^y=e^z$. Then
\begin{enumerate}
\item $\|z\|\le \|x+y\|$ for any $\Ad$-invariant Finsler norm in $\k$.
\item Let $N=\dim(\k)$, then there exist (at most) $N+1$ points $k_i\in K$ and $N+1$ real numbers $\lambda_i \ge 0$ with $\sum_i \lambda_i= 1$ such that
$$
z=\sum_{i=1}^{N+1} \lambda_i\, \Ad_{k_i}(x+y).
$$
\item If equality holds for some Finsler norm, then $z$ and all the $\Ad_{k_i}(x+y)$ lie in the same face of the ball for that norm (in fact, in the intersection of all the faces where $z$ sits).
\item If equality holds for a strictly convex norm then $x,y$ commute, $z=\Ad_k(x+y)$ for some $k\in K$ (thus $x+y\in \overline{\mathbf D}/2$) and $k$ commutes with $e^z$. If moreover $z\in \mathbf D/2$ then $x,y$ commute and $z=x+y$.
\end{enumerate}
\end{cor}
\begin{proof}
Let $\beta(t)=e^{tx}e^{ty}$, which joins $1,e^z$ in $K$. Note that $\beta_t^{-1}\dot{\beta}_t=e^{-t\ad y}(x+y)$ therefore $L(\beta)=\|x+y\|$ for any $\Ad$-invariant Finsler norm in $\k$. By the previous theorem, the third assertion also holds for the closure of $\mathbf D/2$ therefore we must have
$$
\|z\|\le L(\beta)=\|x+y\|
$$
for any such Finsler norm. Assertions $2$ and $3$ follow from Proposition \ref{majo}. If the norm is strictly convex then  $z=\Ad_k(x+y)$ by Proposition \ref{majo}, but also $x,y$ commute by \cite[Theorem 4.17]{lar19}. This implies $e^z=e^{\Ad_k(x+y)}=ke^{x+y}k^{-1}=ke^xe^yk^{-1}=ke^zk^{-1}$ thus $k$ commutes with $e^z$. When $z\in \mathbf D/2$, the condition $e^{\Ad_k z}=e^z$ is only possible if $\Ad_k z=z=\Ad_k(x+y)$ therefore $z=x+y$.
\end{proof}

\begin{rem}
For linear Lie groups such as $K=\SU(n)$, and by Corollary \ref{lingr}, the same results stated in the previous corollary hold for $z\in \overline{\mathbf D}$, i.e. $\|z\|_{\infty}\le \pi$ (and not just $z\in \mathbf \overline{\mathbf D}/2$).
\end{rem}

\subsection{Characterization of all short paths}

Before we proceed with the characterization of other short paths (for the case of non-strictly convex norms), we recall two results concerning the exponential map and the adjoint representation of the group $K$ and its differentials:

\begin{rem}\label{remark}
If $g_t$ is a smooth path in $K$, 
\begin{equation}\label{DAD}
\frac{d}{dt}\Ad_{g_t}v=\Ad_{g_t}[g_t^{-1}\dot{g}_t,v]\qquad \forall v\in \k.
\end{equation}
This follows by noting that to compute the derivative we can take $g_t=g_0e^{tg_0^{-1}\dot{g}_0}$, and differentiate at $t=0$, thus
$$
\Ad_{g_t}v=\Ad_{g_0} \Ad_{e^{tg^{-1}_0\dot{g}_0}}v=\Ad_{g_0}e^{t\ad g^{-1}_0\dot{g}_0}v=\Ad_{g_0}v+ t\Ad_{g_0}\ad(g^{-1}_0\dot{g}_0)v+o(t^2).
$$
On the other hand if $v,w\in \k$ then it is well-known that
\begin{equation}\label{dexp}
e^{-v}D\exp_v w=\int_0^1 e^{-\lambda \ad v}w\,d\lambda=F(\ad v)w
\end{equation}
where $F(\lambda)=\frac{1-e^{-\lambda}}{\lambda}$ is extended by $F(0)=1$ to be an holomorphic function in $\mathbb C$. In particular, note that if $v\in \mathbf D$ (the domain of injectivity of $\exp$),  then $F(\ad v)$ must be nonsingular and in particular $\pm 2\pi i\notin \sigma(\ad v)$; otherwise we would have $0\in \sigma(F(\ad v))$. Moreover, it must be $\sigma(\ad v)\subseteq i(-2\pi ,2\pi )$, otherwise replacing $v$ with $tv$ for some $t\in (0,1)$ we would obtain a contradiction (recall that $\mathbf D$ is convex).
\end{rem}

\medskip



With these tools at hand we now characterize short paths $\gamma$ without the restriction of having length less that the radius of inyectivity of the group. In the next theorem $K$ is a connected compact semi-simple Lie group with the metric induced by an $\Ad$-invariant Finsler norm in $\k$.

\begin{thm}[Characterization of geodesics]\label{cuasi}
Let $\gamma:[a,b]\to K$ be a piecewise $C^1$ path in $K$.  If $\gamma$ is short for the bi-invariant metric, then for (almost) all $t$
\begin{equation}\label{cuasiaut}
\varphi(\gamma_t^{-1}\dot{\gamma}_t)=\|\gamma_t^{-1}\dot{\gamma}_t\|
\end{equation}
for some unit norm functional $\varphi$. Reciprocally, if the equality holds for some $\varphi$ and (almost) all $t\in [t_0,t_1]$, and   $L(\gamma)_{t_0}^{t_1}\le R$, then $\gamma$ is short in $[t_0,t_1]\subseteq [a,b]$.
\end{thm}
\begin{proof}
By the invariance of the metric, it suffices in all cases to consider paths starting at $u=1$. If $\gamma$ is short, assume first that its length is smaller than $R$, then for any $t\in [a,b]$ we have $\dist(\gamma(a),\gamma(t))= L(\gamma)_a^t<R$. Therefore we can lift $\gamma_t= e^{\Gamma_t}$ for a rectifiable path $\Gamma:[a,b]\to\k$, which does not leave the ball $B_R$. Then by Theorem \ref{charactgeod}, for any norming functional of $z$ and any $t$ we have
$$
\varphi(\Gamma_t)=\int_{a}^t \varphi(\dot{\Gamma})=\int_{a}^t \|\gamma^{-1}\dot{\gamma}\|=L(\gamma)_{a}^t=\dist(\gamma_{a},\gamma_{a}e^{\Gamma_t})=\|\Gamma_t\|.
$$
Therefore again by the previous theorem and by Gauss' Lemma \ref{gausslem},
\begin{align*}
\|z\| & =\varphi(z)=\int_{a}^{b} \varphi(\dot{\Gamma}_s)ds=\int_{a}^b \varphi(e^{-\Gamma_s}D\exp_{\Gamma_s}\dot{\Gamma}_s) ds=\int_a^b \varphi(\gamma^{-1}_s\dot{\gamma}_s)ds \\
& \le \int_a^b\|\gamma^{-1}_s\dot{\gamma}_s\|ds =L(\gamma)_a^b=\dist(1,e^z)=\|z\|,
\end{align*}
and this is only possible if $\varphi(\gamma_t^{-1}\dot{\gamma}_t)=\|\gamma_t^{-1}\dot{\gamma}_t\|$ for (almost) all $t$. If $\gamma$ is short but its longer than $R$, we can still partition $\gamma$ in pieces of length smaller than $R$ and its still short on each piece, in any of these intervals $[a_i,a_{i+1}]$ $(i=1,\dots, k)$. There we have $\dist(\gamma(a_i),\gamma(t))\le L(\gamma)_{a_i}^t<R$. Write $\gamma_{a_{i+1}}=\gamma_{a_i}e^{z_i}$, thus we can lift $\gamma_t=\gamma_{a_i} e^{\Gamma^i_t}$ for a rectifiable $\Gamma^i:[a_i,a_{i+1}]\to\k$, which does not leave the ball $B_R$. Using the translation invariance of the metric (from $\gamma_{a_i}$ to $1$), and repeating the argument above, we have for each $i$
\begin{equation}\label{eachi}
\varphi_i(\gamma_t^{-1}\dot{\gamma}_t)=\|\gamma_t^{-1}\dot{\gamma}_t\|\qquad \textrm{ if }\,  \varphi_i \; \textrm{ norms }z_i \,\textrm{ and }\; \|\varphi_i\|=1,
\end{equation}
for (almost) all $t\in [a_i,a_{i+1}]$. Now let $\beta(s)=e^{sz_1}e^{sz_2}\dots e^{sz_k}$, then $\beta(0)=1=\gamma(0)$, $\beta(1)=\gamma(1)$ and
$$
\dist(1,\gamma_1)\le L(\beta)\le \sum_i \|z_i\|=\sum_i \dist(\gamma_i,\gamma_{i+1})=\sum_i L(\gamma|_{[a_i,b_i]}) = L(\gamma)_a^b=\dist(1,\gamma_1),
$$ 
therefore $\beta$ is short also. Since $\beta$ is smooth, there exists $c\in (0,1)$ such that
$$
\sum_i\|z_i\|=\dist(1,\gamma_1)=L(\beta)=\int_0^1 \|\beta_s^{-1}\dot{\beta}_s\|ds= \|\beta_c^{-1}\dot{\beta}_c\|.
$$
Now let $w_0=z_k$, $w_1=e^{-c\ad z_k}z_{k-1}$, $w_2=e^{-c\ad z_k}e^{-c\ad z_{k-1}}z_{k-2}$ and in general
$$
w_j=e^{-c\ad z_k}e^{-c\ad z_{k-1}}\dots e^{-c\ad z_{k-j+1}} z_{k-j}
$$
for $j=0,\dots, k-1$. Note that $\|w_j\|=\|z_{k-j}\|$ for each $j$. A straightforward computation shows that 
$$
\beta_c^{-1}\dot{\beta}_c=\sum_{j=0}^{k-1} w_j,
$$
therefore 
$$
\sum_j \|w_j\|=\sum_i \|z_i\|=L(\beta)=\|\beta_c^{-1}\dot{\beta}_c\|=\|\sum_j w_j\|,
$$
and by Lemmma \ref{mismacara}, there exists a unit norm functional such that $\varphi(w_j)=\|w_j\|$ for all $j=0,\dots, k-1$. In particular $\varphi(z_k)=\varphi(w_0)=\|w_0\|=\|z_k\|$, thus also by Gauss' Lemma (Remark \ref{gausslem}) we have
$$
\varphi(z_{k-1})=\varphi(e^{-c \ad z_k}z_{k-1})=\varphi(w_1)=\|w_1\|=\|z_{k-1}\|.
$$
Proceedings backwards in this fashion, we can conclude that $\varphi(z_i)=\|z_i\|$ for all $i=1,\dots, k$. Thus by (\ref{eachi}) we have that (\ref{cuasiaut}) holds for this $\varphi$, for (almost) all $t\in [a,b]$.

Now assume that (\ref{cuasiaut}) holds for some $\varphi$,  described as $\varphi(v)=\langle v,a\rangle=-\tr(\ad v\circ \ad a)$ for some $a\in \k$ (Remark \ref{normicon}),  and let $t_0\le t\le t_1$ then $\dist(\gamma_{t_0},\gamma_t)\le L(\gamma)_{t_0}^t\le L(\gamma)_{t_0}^{t_1}=R$. Using the invariance of the metric, we can assume that $\gamma_{t_0}=1$. Then there exists a rectifiable lift $\Gamma$ of $\gamma$ such that $\Gamma\subseteq B_R$, $\Gamma_{t_0}=0$, $e^{\Gamma_{t_1}}=\gamma_{t_1}$. Note that $\|\Gamma_{t_1}\|=R=\dist(1,\gamma_{t_1})$. By Remark \ref{normicon},  since $\varphi$ norms $\gamma_t^{-1}\dot{\gamma}_t$, we have that $\gamma_t^{-1}\dot{\gamma}_t$ and $a$ commute for all $t$. By formula (\ref{DAD}) of Remark \ref{remark}
$$
\frac{d}{dt} \Ad_{\gamma_t}a=\Ad_{\gamma_t}[\gamma_t^{-1}\dot{\gamma}_t,a]=0,
$$ 
which shows that $\Ad_{\gamma_t}a=a$ for all $t$ (since $\gamma_{t_0}=1$). Then by the formula  (\ref{dexp}) for the differential of the exponential map of $K$,
\begin{align*}
e^{-\Gamma_t}D\exp_{\Gamma_t}[\Gamma_t,a] & = (\int_0^1 e^{-\lambda \ad\Gamma_t}\ad \Gamma_t \,d\lambda\, )(a)=-e^{-\lambda \ad \Gamma_t}\Bigr|_0^1(a)\\
&=-e^{-\ad\Gamma_t}a+a=-\Ad_{\gamma_t^{-1}}a+a=0.
\end{align*}
Since $\Gamma_t$ is inside the injectivity radius of the exponential map, it follows that $[\Gamma_t,a]=0$ for all $t\in [t_0,t_1]$. Then for fixed $t$, the operators $\ad \Gamma_t$ and $\ad a$ commute, and since they are both skew-adjoint operators acting on $\k$, they can be simultaneously diagonalized in an orthonormal basis, say $\{e_1,\dots,e_k\}$ of $\k$. Let $p_i=e_i\otimes e_i$ denote the rank-one orthogonal projection with range $e_i$, then $\ad a=\sum_{i=1}^k a_i p_i$ and $\ad \Gamma_t=\sum_i \lambda_i p_i$. Note that 
$\Ad_{\gamma_t}=e^{\ad\Gamma_t}=\sum_i e^{\lambda_i}p_i$. Now we apply the formula (\ref{dexp}) for the differential of the exponential map to the path $t\mapsto e^{\ad\Gamma_t}$, and noting that $\ad\dot{\Gamma}_t=(\ad\Gamma_t)^{\, \cdot}$ we obtain
$$
e^{-\ad\Gamma_t} D\exp_{\ad\Gamma_t}\ad\dot{\Gamma}_t =\int_0^1 e^{-s\ad(\ad\Gamma_t)}\ad\dot{\Gamma}_t ds = \int_0^1 e^{-s\ad\Gamma_t}\ad\dot{\Gamma}_t e^{s\ad\Gamma_t}ds.
$$
Due to (\ref{DAD}) we also have
\begin{equation}\label{aii}
\ad \gamma_t^{-1}\dot{\gamma}_t=\Ad_{\gamma_t^{-1}}\frac{d}{dt}\Ad_{\gamma_t}=e^{-\ad\Gamma_t}D\exp_{\ad\Gamma_t}(\ad\dot{\Gamma_t}).
\end{equation}
Then since $\ad a$ and $e^{\ad\Gamma_t}$ are diagonal in the orthonormal basis $\{e_i\}_{i=1,\dots,k}$ (for this particular $t$), we have
\begin{align*}
p_i ( e^{-s\ad\Gamma_t}\ad\dot{\Gamma}_t e^{s\ad\Gamma_t}\circ \ad a) p_i & = e^{-s\lambda_i} p_i \ad\dot{\Gamma}_t \,e^{s\lambda_i}\,a_i\, p_i= p_i \ad\dot{\Gamma}_t\, a_i\, p_i \\
& =  a_i(\ad\dot\Gamma_t)_{ii} \,p_i=p_i(\ad a\circ \ad\dot{\Gamma}_t)p_i.
\end{align*}
Integrating $s$ in $[0,1]$ it follows that
$$
(\ad a\circ e^{-\ad\Gamma_t} D\exp_{\ad\Gamma_t}\ad\dot{\Gamma}_t)_{ii}=(\ad a\circ \ad\dot{\Gamma}_t)_{ii}
$$
and using (\ref{aii}) shows that
$$
(\ad a\circ (\ad \gamma_t^{-1}\dot{\gamma}_t))_{ii}= (\ad a\circ \ad\dot{\Gamma}_t)_{ii}
$$
for all $i$. Thus
\begin{align*}
\varphi(\gamma_t^{-1}\dot{\gamma}_t)&=\langle a,\gamma_t^{-1}\dot{\gamma_t}\rangle=-\tr(\ad \gamma_t^{-1}\dot{\gamma}_t\circ\ad a)=-\sum_i (\ad \gamma_t^{-1}\dot{\gamma}_t\circ\ad a)_{ii}\\
&= -\sum_i (\ad a\circ\ad\dot{\Gamma}_t)_{ii}=-\tr(\ad a\circ\ad\dot{\Gamma}_t) = \langle\dot{\Gamma_t},a\rangle=\varphi(\dot{\Gamma}_t).
\end{align*}
Since this holds for any $t\in [t_0,t_1]$, we have $\varphi(\dot{\Gamma}_t)=\|\gamma_t^{-1}\dot{\gamma}_t \|$ there, and then
$$
\dist(\gamma_{t_0},\gamma_{t_1})=\|\Gamma_{t_1}\|\ge \varphi(\Gamma_{t_1})=\varphi(\int_{t_0}^{t_1} \dot{\Gamma}_t \, dt)=\int_{t_0}^{t_1} \varphi(\dot{\Gamma}_t)\, dt= \int_{t_0}^{t_1} \|\gamma_t^{-1}\dot{\gamma}_t \|dt=L_{t_0}^{t_1}(\gamma)
$$
which proves that $\gamma$ is short in that interval.
\end{proof}

\begin{rem}It is clear from the proof of the previous theorem, that
\begin{itemize}
\item[i)] If $\gamma=e^{\Gamma}$ is short, then there exists $a\in \k$ such that $\ad\Gamma_t$ (and then also $\ad\dot{\Gamma}_t$, $\Ad_{\gamma_t}$,  $\ad\gamma_t^{-1}\dot{\gamma}_t$) commute with $\ad a$ for all $t$. From the faithfulness of the adjoint representation, this is equivalent to $\Gamma_t,\dot{\Gamma}_t,\gamma_t,\gamma_t^{-1}\dot{\gamma}_t$ commuting with $a$ for each $t$. For linear Lie groups it then follows that $a$ and $\Gamma_t$ can be simultaneously diagonalized for each $t$, and this basis also diagonalizes $\gamma_t=e^{\Gamma_t}$. However, a word of caution: the orthonormal basis depends on $t$. 
\item[ii)] Once $\gamma$ is short in a certain interval $[a,b]$, and assuming that we first write $\gamma_b=\gamma_a e^{z_1}e^{z_2}\dots e^{z_k}$, with $\gamma(t_{i+1})=\gamma(t_i)e^{z_i}$, and the $z_i$ giving the distance among the endpoints and $\|z_i\|\le R$, then the equality (\ref{cuasiaut}) if fulfilled \textit{for any} norming functional of $z_1+\dots+z_n$. If $\gamma_b=\gamma_ae^z$ with $z\in \mathbf D$, when can we ensure that $z$ and the $z_i$ are in the same face of the sphere? Equivalently, $z$ and $\gamma^{-1}\dot{\gamma}$ are in the same face of the sphere? If $\|z\|<R$ this follows from Theorem \ref{charactgeod}, see also the next remark.
\item[iii)] If $\gamma_0=1,\gamma_1=e^z$ with $z\in \mathbf D$, and $\gamma$ is short for a given norm, does it follow that  $\gamma=e^{\Gamma}$ with $\Gamma\subset \mathbf D$? The argument used in repeated occasions is that if $\|z\|<R$ and $\gamma$ is short for that norm, then there is a lift $\Gamma\subset B_R$; but this radius depends on the norm.
\end{itemize}
\end{rem}

\medskip

\begin{rem}[Non-optimal lifts]
By Theorem \ref{charactgeod}, we know that if $\Gamma$ is short in $\k$ for a given norm, then its exponential $\gamma=e^{\Gamma}$ is short in $K$ for the bi-invariant metric of that norm. Is there any concrete example of Lie group $K$ with $\Ad$-invariant Finsler norm such that: there exists a path $\gamma=e^{\Gamma}:[a,b]\to K$, with $\gamma$  short in $K$ joining $1,e^z$ but $\Gamma$ (which joins $0,z$ in $\k$) not short in $\k$? For strictly convex norms this is not possible since the only short paths are segments; as we will see in the next section it is also impossible if all the faces of the unit sphere are abelian (Corollary \ref{corabelian}). Another relevant example, studied by Antezana, Ghighlioni and Stojanoff in \cite{ags} is the full unitary group (or in our setting, $K=\SU_n$) with the spectral norm: examining \cite[Theorem 2.1]{ags} it is apparent that lifted short paths  are also short there.
\end{rem}

\smallskip

\subsection{Norms with abelian faces} In Theorem \ref{existen} we showed that when the faces of the unit ball are singletons, the short paths are one-parameter groups. We will show here that the situation is somewhat similar if we allow the faces of the  unit ball to be inside \textit{abelian} subalgebras of $\k$. In Section \ref{geodesicommute} we will give a full characterization of those norms with this important structural property.

\begin{prop}\label{abelianfaces1}
Let $B$ the unit ball of the bi-invariant norm in $\k$. Then all the faces of $B$ are abelian if and only if for all piecewise $C^1$ short curves $\gamma\subseteq K$, the logarithmic derivatives $x_t=\gamma_t^{-1}\dot{\gamma}_t$ of $\gamma$ commute for all $t$.
\end{prop}

\begin{proof}
Assume first that each face of the unit ball is abelian. By Theorem \ref{cuasi}, if $\gamma$ is short, it has logarithmic derivative (after normalizing)  inside a face of the ball, and these commute for all $t$. 

Now assume that for some $w\in\k$ the face $F_w(B)$ is non-abelian. Let $x,y\in F_w(B)$ with $[x,y]\ne 0$. Let $R>0$ be the radius of injectivity and for $0<b<R$ let $\gamma:[0,b]\to K$ be the path which solves
$$\gamma_t^{-1}\dot{\gamma}_t=tx+(b-t)y$$
for $t\in[0,b]$. Then by Theorem \ref{cuasi} $\gamma$ is a geodesics, furthermore the logarithmic derivatives $\gamma_0^{-1}\dot{\gamma}_0=x$ and $\gamma_b^{-1}\dot{\gamma}_b=y$ do not commute.
\end{proof}

\begin{cor}\label{corabelian}
Let  $B$ be the unit ball of the bi-invariant norm in $\k$ and assume that the faces of $B$ are all abelian. Let $\gamma:[a,b]\to K$ be a short piecewise $C^1$ path. Then
\begin{itemize}
\item[i)] There exists $z\in \k$ such that $\delta_t=\gamma_a e^{tz}$ is also short with the same endpoints.
\item[ii)] If $L(\gamma)\le R$ (thus $\gamma_b=\gamma_ae^z$ for some $\|z\|\le R$), then $\gamma=\gamma_a e^{\Gamma}$ where $\Gamma:[a,b]\to \k$ and $[\Gamma_t,\Gamma_s]=0$ for all $s,t\in [a,b]$, thus $\gamma_t^{-1}\dot{\gamma}_t=\dot{\Gamma}_t$. In particular the logarithmic derivatives of $\gamma$ commute and $\Gamma$ is short in $\k$, with the same length than $\gamma$.
\end{itemize}
\end{cor}
\begin{proof}
As always we can assume that $\gamma_a=1$, partition $\gamma$ in small pieces such that $\gamma(t_{i+1})=\gamma(t_i)e^{z_i}$ with $\|z_i\|\le R$ as in the proof of Theorem \ref{cuasi}. Let $\beta(s)=e^{sz_1}e^{sz_2}$, then $\beta$ joins $1,e^{z_1}e^{z_2}=\gamma(t_2)$ and $L(\beta)=\|z_1+z_2\|$, thus
$$
\dist(1,\gamma_{t_2})\le L(\beta)\le\|z_1\|+\|z_2\|=\dist(1,\gamma_{t_1})+\dist(\gamma_{t_1},\gamma_{t_2})=\dist(1,\gamma_{t_2}),
$$
which shows that $\beta$ is short among its endpoints. For small $s$, let $\beta_s=e^{B_s}$ with $B_s=s(z_1+z_2)+\frac{s^2}{2}[z_1,z_2]+o(s^3)$, then by Theorem \ref{charactgeod}(1), $\|B_s\|=\dist(1,\beta_s)=L_0^s(\beta)=s\|z_1+z_2\|$. On the other hand, by Theorem \ref{charactgeod}(4), we also have $\varphi(\dot{B}_s)=\|\beta_s^{-1}\dot{\beta}_s\|=\|z_1+z_2\|$, then integrating we have $\varphi(B_s)=s\|z_1+z_2\|=\|B_s\|$, and this shows that $B_s$ (normalized) is inside a face of the ball, which implies that the $B_s$ commute for all (small) $s$ (therefore they also commute with $\dot{B}_s$). Then using the formula (\ref{dexp}) we have
$$
e^{-s\ad z_2}(z_1+z_2)=\beta_s^{-1}\dot{\beta}_s=e^{-B_s}D\exp_{B_s}\dot{B}_s=\dot{B_s}=z_1+z_2+s[z_1,z_2]+o(s^2).
$$
Differentiating at $s=0$ shows that $[z_1,z_2]=0$. Thus $\gamma(t_2)=e^{z_1}e^{z_2}=e^{z_1+z_2}$ and $\gamma(t_3)=e^{z_1}e^{z_2}e^{z_3}=e^{z_1+z_2}e^{z_3}$. We repeat the argument now using $\beta_s=e^{s(z_1+z_2)}e^{sz_3}$, this shows that $[z_3,z_1+z_2]=0$. Thus $\gamma_b=e^{z_1}e^{z_2}e^{z_3}\dots e^{z_k}=e^{\sum z_i}$, and if we let $z=\sum z_i$, then $\|z\|=\sum \|z_i\|=\dist(\gamma_a,\gamma_b)$ and the first claim follows.

For the second assertion let $\Gamma$ be the smooth lift of $\gamma$, $\Gamma_a=0$, then by Theorem \ref{charactgeod}(4) we have $\varphi(\dot{\Gamma}_t)=\|\gamma_t^{-1}\dot{\gamma}_t\|$ for some unit norm $\varphi$, and again integrating $\varphi(\Gamma_t)=L_0^t(\gamma)=\dist(1,e^{\Gamma_t})=\|\Gamma_t\|$ shows that $\Gamma_t$ (normalized) is inside a face of the sphere, which by hypothesis is abelian. But then $\Gamma_t$ and $\dot{\Gamma}_t$ also commute for all $t$ and it is apparent that
$$
L(\Gamma)=\int_a^b \|\dot{\Gamma}\|dt=\int_a^b\|e^{-\Gamma}D\exp_{\Gamma}\dot{\Gamma}\|dt=\int_a^b\|\gamma^{-1}\dot{\gamma}\|dt=L(\gamma)
$$
thus if $\Gamma_b=z$ then $\|z\|\le L(\Gamma)=L(\gamma)=\dist(1,e^z)=\|z\|$ and $\Gamma$ is short.
\end{proof}

\subsection{Geodesic are quasi-autonomous}\label{sectionquasiaut}

Throughout, the action of the semi-simple compact group $K$ on the symplectic manifold $M$ complies the hypothesis of Section \ref{hamactions}. Recall (Definition \ref{quasiauto}) that a Hamiltonian $H_t$ is called \textit{quasi-autonomous} if there exists $x^-,x^+ \in M$ such that $H_t(x^-) = \min_M H_t$, $H_t(x^+) = \max_M H_t$ for all $t\in[a,b]$. As before, we use $R$ to indicate the injectivity radius of the exponential map of $K$, for the given norm. In the next theorem, the geometry of the group $K$ is the one given by the generalized Hofer norm (\ref{hofer1}) obtained by the almost effective action.

\begin{thm}\label{teoautono}
Let $K\curvearrowright (M,\omega)$ be a Hamiltonian almost effective action. Let $\gamma:[a,b]\to K$ be piecewise $C^1$, and denote its right logarithmic derivative by $x_t=\dot{\gamma}_t\gamma^{-1}_t$. Then if $\gamma$ is short, $(\mu^{x_t})_{t\in [0,1]}$ is a quasi-autonomous Hamiltonian, and if $\mu$ is quasi-autonomous, $\gamma$ is locally short (in each interval of length $\le R$).
\end{thm}

\begin{proof}
By Theorem \ref{cuasi}, $\gamma$ is locally short if and only if there is a common norming functional for $x_t=\dot{\gamma}_t\gamma^{-1}_t$, for all $t$. By Corollary \ref{arghofer} the set of logarithmic derivatives $\{x_t\}_{t\in [0,1]}\subseteq \k$ is contained in a cone generated by a face if and only if there exist $x^-$ and $x^+$ such that $x^-\in\cap_{t\in [0,1]}\argmin(\varphi_{x_t})$ and $x^+\in\cap_{t\in [0,1]}\argmax(\varphi_{x_t})$, where $\varphi_{x_t}:\mu(M)\to\R$. We can choose $m^-,m^+\in M$ such that $x^-=\mu(m^-)$ and $x^+=\mu(m^+)$. The result follows since $\mu^{x_t}=\varphi_{x_t}\circ\mu$ for all $t\in[0,1]$.
\end{proof}

\begin{rem}
A similar characterization of geodesics (in fact, of their logarithmic derivative $\gamma_t^{-1}\dot{\gamma}_t=x_t$) can be obtained for the one-sided norm induced by the action of $K$ in $M$, if one replaces the condition of quasi-autonomous for the Hamiltonian $H_t=\mu^{x_t}$, with the condition that there exists a point $x^+\in M$ such that $H_t(x^+)=\max_M H_t$ for all $t\in [a,b]$.
\end{rem}

\section{Spheres with abelian faces}\label{geodesicommute}

In this section we characterize those $\Ad$-invariant norms which have unit balls with abelian faces in terms of conditions on its intersection with a Cartan subalgebra. When this intersection is a polytope the condition reduces to the regularity of the extreme points of its polar dual. Based on this result and Kirwan's Theorem \ref{kirwan} we characterize the compact semi-simple groups of Hamiltonian diffeomorphims such that geodesics have commuting Hamiltonians. We start with the important special case of Hofer norms derived from coadjoint actions on regular coadjoint orbits. Groups with length structures derived from these norms have the property that the logarithmic speed of its geodesics lie in a Weyl chamber.

\subsection{Regular coadjoint actions and non-crossing of eigenvalues}\label{subsectionspeedweyl}

In this section we follow the setting of Example \ref{coadj} on (co)adjoint orbits in $\O\subseteq \k\simeq \k^*$. 

\begin{defn}
The (co)adjoint orbit $\O=\O_{\lambda}\subseteq \k$ of the action is \textit{regular} if $\lambda\in\k$ is a regular element (Remark \ref{regular}).
\end{defn}

\smallskip

In order to characterize the faces of the unit ball in $\k$ via Corollary \ref{arghofer} we need to  study, for a nonzero $x$ in $\k$, the sets $\argmax_{\O}(\varphi_x)$ and $\argmin_{\O}(\varphi_x)$ of $\varphi_x=\mu^x:\O\to\R$. Let $m_t=\Ad_{e^{tv}}m$ be a path through $m\in\O$ with $\dot{m}_0=[v,m]\in T_m\O$, differentiating $\varphi_x(m_t)=\langle x,m_t\rangle$ at $t=0$ we obtain
$$
D(\varphi_x)_m(\dot{m}_0)=\langle x, [v,m]\rangle =-\langle x,[m,v]\rangle =\langle [m,x],v\rangle,
$$
since $\ad m$ is skew-adjoint for the Killing form. Since $v\in \k$ is arbitrary, this  implies  that $m\in\O$ is a critical point of $\varphi_x$ if and only if $m\in \z(x)$, where $\z(x)$ is the centralizer of $x$, i.e. 
\begin{equation}\label{critpoint}
Crit(\varphi_x)=\O\cap\z(x).
\end{equation}

The proof of following proposition can be found in \cite[Lemma 23]{bgh}.

\begin{prop}\label{maximumorbit}
Fix a maximal torus $T\subseteq K$, a nonzero vector $x\in \t=Lie(T)$ and a  point $m\in \O\cap\t$; thus $m\in Crit(\varphi_x)$. Then $m$ is a maximum point of $\varphi_x$ if and only if there is a Weyl chamber in $\t$ whose closure contains both $x$ and $m$.
\end{prop}

With these tools at hand, we next characterize the faces of the unit ball of $\|\cdot\|_{\O}$. 
 
\begin{prop}\label{speedweylchamber}
Let $\O$ be a regular (co)adjoint orbit and let $\|\cdot\|_{\O}$ be the associated $\Ad$-invariant Hofer norm. A set of vectors has a common norming functional if and only if it is contained in a Weyl chamber (given by a choice of torus and positive simple roots). Hence, the maximal cones generated by faces of the unit ball are Weyl chambers.  
\end{prop}

\begin{proof}
By Corollary \ref{arghofer} a set of elements $S\subseteq\k$ have the same norming functional if and only if the set $\{\varphi_u:u\in S\}$ has a common maximizer $x^+\in \O$ and a common minimizer $x^-\in\O$. 

We denote by $\t_+$ the Weyl chamber that contains $x^+$, it is unique since $x^+$ is regular. If $u\in S$, then the functional $\varphi_u$ has a maximum at $x^+$, hence by equation (\ref{critpoint}) $u$ commutes with $x$ and therefore $u\in\t$.  By Proposition \ref{maximumorbit} we conclude that $u\in\t_+$ if and only if the functional $\varphi_u$ has a maximum at $x^+$. Let us denote by $x'\in\k$ the element defined by $\{x'\}=\O\cap -\t_+$ , that is, the element of $\O$ in the opposite Weyl chamber. A functional $\varphi_u$ has a minimum at $x'$ if and only if $\varphi_{-u}$ has a maximum at $x'$. This holds if and only if $-u$ and $x'$ belong to the same Weyl chamber, which is equivalent to $u\in\t_+$. 
Hence, $\varphi_u$ has a maximum at $x^+$ and $\varphi_u$ has a minimum at $x'$, are both equivalent to $u\in\t_+$. If we take $x^-=x'$ as the common minimizer we get a maximal face which is $\t_+$.    

If $\t_+$ is a Weyl chamber we denote by $x^+$ and $x^-$ the elements defined by $\{x^+\}=\O\cap \t_+$ and $\{x^-\}=\O\cap -\t_+$. Then we can reverse the argument in the previous paragraph to conclude that the functionals $\varphi_u$ with $u\in \t_+$ have $x^+\in \O$ as maximizer and $x^-\in\O$ as minimizer.  
\end{proof}

From this characterization of the faces of the unit ball in $\k$ we can determine the geodesics in $K$. Let $\gamma:[a,b]\to K$ be a curve in a group $K$ endowed with the length structure obtained from Hofer's norm $\|\cdot\|_{\O}$ for a regular (co)adjoint orbit $\O$. As a combination of the previous proposition and Theorem \ref{cuasi}, we obtain

\begin{thm}\label{speedchamber}
If $\gamma$ is short then all its logarithmic derivatives are contained in the same Weyl chamber $\t_+$ (given by a choice of torus and positive simple roots). If its derivatives are contained in a Weyl chamber, then $\gamma$ is locally short (in each interval of $\length \le R$, where $R$ is the injectivity radius of the norm).
\end{thm}

\begin{ex}[$K=\SU(n)$ and the non-crossing of eigenvalues]
Consider the case of $\SU(n)$ acting on a regular (co)adjoint orbit $\O_{\lambda}$ containing  $\lambda=i\diag(\lambda_1,\dots,\lambda_n)$, with $\lambda_1<\dots<\lambda_n$. Recall that here we obtain in $\k$ the $\lambda$-numerical radius as Hofer norm (Example \ref{sun}). The local condition for the logarithmic derivatives to be the speeds of geodesics is that there exist a fixed orthonormal basis such that the speeds are simultaneously diagonal in this basis for all $t$ (Corollary \ref{corabelian}) and if $x_i(t)$ are the eigenvalues of $x_t=\dot{\gamma}_t\gamma^{-1}_t$, then 
$$
x_1(t)\leq\dots\leq x_n(t) \qquad \forall t\in [a,b].
$$
This is because in $\SU(n)$ the condition of being in the same Weyl chamber is given by the non-crossing of the eigenvalues.
\end{ex}

\begin{rem}
The Hofer norm associated to \textit{singular} (co)adjoint orbits can be studied using the following result on maximizers and minimizers of linear functionals restricted to the orbits, see \cite[Lemma 22]{bgh}. Let $Z(x)$ be the centraliser of $x$ in $K$ and let $F_x(\O)$ be  the face of $\O$ defined by $x$. Then
\begin{itemize}
    \item $\argmax(\varphi_x)$ is a $Z(x)$-orbit,
    \item $\ext(F_x(\conv(\O))=\argmax(\varphi_x)$, so $\ext(F_x(\conv(\O))$ is a $Z(x)$-orbit.
    \item $F_x(\conv(\O))\subseteq \z(x)$.
\end{itemize}
\end{rem}

\subsection{Invariant norms with abelian faces}\label{ssabelianface}

Let $K$ be a compact connected semi-simple Lie group and let $\langle\cdot,\cdot\rangle$ be as usual the opposite Killing form of $\k^\C$ restricted to $\k$. Let $\Delta$ be the set of (real) roots of $\k$ with respect to a fixed Cartan subalgebra $\t $, as in Remark \ref{rootdecom}. For $x\in\t$ we have 
\begin{align}\label{centralizer}
\z(x) & =\t\oplus\bigoplus_{\alpha\in\Delta_+,\alpha(x)=0}Z_\alpha,
\end{align}
here  $\z(x)$ is the Lie algebra of $Z(x)$ as usual (the centralizer of $x\in \k$). We define the \textit{smallest Weyl chamber wall containing} $x$ as the linear space 
$$
W_x=\bigcap_{\{\alpha\in\Delta_+ : \alpha(x)=0\}}\ker(\alpha).
$$

\begin{thm}\label{abelianfaces}
Let $K$ be a semi-simple compact connected group and let $B$ be an $\Ad$-invariant convex body in $\k$ containing $0$. Let $\t$ be a Cartan algebra in $\k$ and $\t_+$ a positive Weyl chamber. Then all faces of $B$ are abelian if and only if for all $x\in\t^+$ we have
$$
F_x(B\cap\t)\cap\t_+\subseteq W_x,
$$ 
and in this case 
$F_x(B)=F_x(B\cap\t)\cap\t_+.$
\end{thm}

\begin{proof}
We assume for simplicity that $x\in\t^+$. Theorem \ref{lewisfaces} states that 
$$F_x(B)=\Ad_{Z(x)}(F_x(B\cap\t)\cap\t_+).$$ 

Next we show that this set is abelian when the inclusion stated in the theorem holds. If $x$ is in the interior of the Weyl chamber then $Z(x)$ is trivial and $F_x(B)=F_x(B\cap\t)\cap\t_+$, hence the face is abelian. If $x$ is not regular, and hence $Z(x)$ is not trivial, then by (\ref{weights}) and (\ref{centralizer}) the centralizer $Z(x)$ acts trivially on $W_x$. Therefore, the inclusion $F_x(B\cap\t)\cap\t_+\subseteq W_x$ implies that $F_x(B)=\Ad_{Z(x)}(F_x(B\cap\t)\cap\t_+)=F_x(B\cap\t)\cap\t_+$, so the face is abelian. Note also that the last assertion of the theorem follows from this argument.

If the inclusion in the statement of the theorem is not satisfied there exists $y\in F_x(B\cap\t)\cap\t_+$ such that $y\notin W_x$; it follows that there exists $\alpha$ such that $\alpha(y)\neq 0$ and $\alpha(x)=0$. Consider the orbit $\Ad_{Z(x)}(y)\subseteq \k$ as a submanifold with its differentiable structure, and observe that 
$$
[\z(x),y]\subseteq T_y(\Ad_{Z(x)}(y)).
$$
By (\ref{centralizer}), we have $v_\alpha,u_\alpha\in\z(x)$;  the equations $[y,u_\alpha]=-\alpha(y)v_\alpha$ and $[y,v_\alpha]=\alpha(y)u_\alpha$ hold by (\ref{weights}), therefore we conclude that 
$$\{u_\alpha,v_\alpha\}\subseteq T_y(\Ad_{Z(x)}(y)).$$ 
Since $[u_\alpha,v_\alpha]=h_\alpha\neq 0$ the tangent $T_y(\Ad_{Z(x)}(y))$ is not an abelian space, so that $\Ad_{Z(x)}(y)$ cannot be an abelian set. The inclusion $\Ad_{Z(x)}(y)\subseteq \Ad_{Z(x)}(F_x(B\cap\t)\cap\t_+)=F_x(B)$ implies that $F_x(B)$ is not an abelian set.   
\end{proof}

\begin{rem}\label{maximalabelianfaces}
From the previous theorem we know that if the faces of $B$ are abelian then 
$$F_x(B)=F_x(B\cap\t)\cap\t_+$$
for a Weyl chamber $\t_+$ given by a choice of torus and positive simple roots such that $x\in\t_+$. Therefore the balls of the norms studied in Section \ref{subsectionspeedweyl} have maximal commuting cones generated by faces. 
\end{rem}

\subsection{Polytopes with regular extreme points}\label{ssabelianfacepolytope}

The aim of this section is to prove the following theorem. Note that an element $x\in \k$ is regular if and only if $\alpha(x)\neq 0$ for all $\alpha\in\Delta_+$.

\begin{thm}\label{abelian}
Let $K$ be a semi-simple compact connected group, let $B$ be an $\Ad$-invariant convex body in $\k$ containing $0$, such that $B\cap\t=P^\circ$ is a polytope. Then all faces of $B$ are abelian if and only if all extreme points of $P$ are regular.
\end{thm}

For the proof of this theorem we need a couple of preliminary lemmas on polytopes invariant under finite reflection groups.  These lemmas will be later applied to the Hofer norm polytopes defined in Definition \ref{hoferpolytopes}, which are polytopes invariant under the Weyl group. 

\begin{figure}[h]
\def\svgwidth{6.5cm}
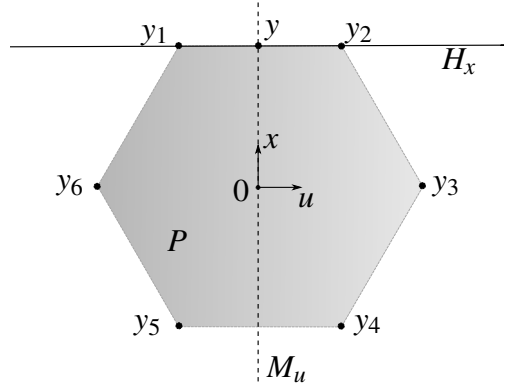
\caption{\small{Mirror hyperplane and extreme points of a polytope $P$.}}
\label{fig: refl}
\end{figure}

\begin{defn}[Finite reflection groups]\label{frl}
Let $\W$ be a finite reflection group acting isometrically on an inner product vector space $(V,\langle \cdot,\cdot\rangle)$ and generated by reflections $\{r_u\}_{u\in\Phi}$. Here $r_u$ is the refection which fixes the mirror hyperplane $M_u=\{u\}^\perp$ and $\Phi$ is a finite subset of the unit sphere that is called the positive root system. For a finite reflection group the space can be subdivided into chambers bounded by mirror hyperplanes. We can choose any (closed) chamber $C$ and call it the fundamental chamber. Given a positive root system $\Phi$, a fundamental chamber is given by
$$
C=\{x\in V:\langle x,u\rangle\geq 0 \mbox{ for all }u\in\Phi\}.
$$
Let $P$ be a polytope with extreme points $\ext(P)=\{y_1,\dots,y_m\}$ and assume that $P$ is invariant under $\W$.
\end{defn}

\begin{lem}\label{lemregext}
Let $\Pi\subseteq\Phi$ and assume that $\cup_{u\in\Pi}M_u\cap\{y_1,\dots,y_m\}=\emptyset$. If $H_x$ supports $P$ at $y$ and $y\in\cap_{u\in\Pi}M_u$ it follows that $x\in\cap_{u\in\Pi}M_u$.
\end{lem}
\begin{proof}
Assume $P$ is invariant under a reflection $r_u$ and the mirror hyperplane does not contain any extreme point of $P$, i.e. $M_u\cap\{y_1,\dots,y_m\}=\emptyset$, see Figure \ref{fig: refl}. Then, if $H_x$ is a hyperplane that supports $P$ at an $y\in M_u$  we claim that $x\in M_u$. To prove this we write $y$ as a convex combination of the $y_1,\dots,y_m$, i.e. $y=\lambda_1y_1+\dots\lambda_ny_n$, with $n\le m$, $\lambda_1,\dots,\lambda_n>0$, and $\sum_{i=1}^n\lambda_i=1$. If we write 
$$
y=\frac{1}{2}y+ \frac{1}{2}r_u(y)=\frac{1}{2}(\lambda_1y_1+\dots\lambda_ny_n)+\frac{1}{2}r_u(\lambda_1y_1+\dots\lambda_ny_n),
$$
we see that we can assume that $y=\lambda_1y_1+\dots\lambda_ny_n$, with $\lambda_1,\dots,\lambda_n>0$, $\sum_{i=1}^n\lambda_i=1$ and $r_u(y_j)\in\{y_1,\dots,y_n\}$ for $i=1,\dots,n$. Since $H_x$ supports $P$ at $y\in M_u$ we get $\varphi_x(y)=1$ and $\varphi_x(y_1),\dots\varphi_x(y_m)\leq 1$. From $y=\lambda_1y_1+\dots+\lambda_ny_n$, with $\lambda_1,\dots,\lambda_n>0$, $\sum_{i=1}^n\lambda_i=1$, it follows that $\varphi_x(y_1),\dots, \varphi_x(y_n)= 1$. Since $y_1\notin M_u$ we know that $r_u(y_1)\in\{y_2,\dots,y_n\}$ so that $y_1-r_u(y_1)$ is non zero. Since $y_1-r_u(y_1)$ is orthogonal to  the mirror $M_u$ and $\varphi_x(y_1-r_u(y_1))=1-1=0$ we conclude that $x$ is orthogonal to $y_1-r_u(y_1)$ and is therefore contained in $M_u$. The Lemma follows by applying this result to each of the mirrors $\{M_u\}_{u\in\Pi}$.
\end{proof}

\begin{lem}\label{lemnonregext}
Assume that there is an extreme point $y_1$ of $P$ such that $y_1\in C$ and $y_1$ is contained in a mirror. Then there is a supporting hyperplane $H_{x'}$ of $P$ at $y_1$ such that $x'$ is in the interior of the fundamental chamber.
\end{lem}

\begin{proof}
By assumption $\Pi:=\{u\in\Phi:y_1\in M_u\}\neq\emptyset$. The set of reflections $\{r_u:u\in\Pi\}$ generates the stabilizer $\Stab(y_1)$ of $y_1$ (see e.g.  \cite[Theorem 12.6]{bor}). Since $P$ is a polytope there exists $x$ which determines a supporting hyperplane of $P$ at $y_1$ such that $\varphi_x(y_1)=1$ and $\varphi_x(y_j)<1$ for $j=2,\dots,m$. If we denote by $x'$ the average
$$
x'=\frac{1}{\vert \Stab(y_1) \vert}\sum_{w\in \Stab(y_1)} w.x,
$$
then $\varphi_{x'}(y_1)=1$, $\varphi_{x'}(y_j)<1$ for $j=2,\dots,m$, and $x'\in\cap_{u\in\Pi}M_u$. To show that $x'\in C$ we need to verify that $\langle x',u\rangle\geq 0 $ for all $u\in\Phi$. If $u$ is not in $\Pi$ then $y_1$ is not in $M_u$, therefore $\varphi_{x'}(y_1)=1$ and $\varphi_{x'}(r_u(y_1))<1$ which implies that $\varphi_{x'}(y_1-r_u(y_1))>0$. On the other hand, since $y_1$ is in the positive chamber $C$ it follows that $y_1-r_u(y_1)=d.u$ for $d>0$. Hence $\varphi_{x'}(d.u)>0$ which is equivalent to $\langle x',u\rangle>0$.

The $x'$ can be perturbed so that $H_{x'}$ is a supporting hyperplane of $P$ at $y_1$ and $x'$ is in the interior of the fundamental chamber.
\end{proof}

\begin{proof}\textit{(of Theorem  \ref{abelian})}. We apply the previous lemmas to the Weyl group action on $\t$. It is easy to check that the condition of  Theorem \ref{abelianfaces} can be stated using the more general notation of the previous lemmas as follows:
$$F_y(P^\circ)\cap C\subseteq W_y:=\bigcap_{u\in\Phi:y\in M_u}M_u$$
for all nonzero $y$ in $C=\t_+$. Here, for example, $M_{h_{\alpha}}=\ker(\alpha)$, with $h_\alpha$ as in (\ref{weights}). We can multiply $y$ by a  positive scalar and assume that it is in $\bd P$. Then by Theorem \ref{supportdual} 
\begin{align*}
F_y(P^\circ)&=\{x\in V:H_y\mbox{  supports  } P^\circ\mbox{  at  }x\}\\
&=\{x\in V:H_x\mbox{  supports  } P\mbox{  at  }y\}.
\end{align*}
By Lemma \ref{lemregext}, if the extreme points of $P$ are all regular and $y\in W_y=\bigcap_{u\in\Phi:y\in M_u}M_u$, then $x\in W_y$ holds when $H_x$ supports $P$ at $y$, i.e. $x$ is in $F_y(P^\circ)$.

If there is a non-regular extreme point of $P$ then by invariance of $P$ under the reflection group $\W$ we can choose a non-regular extreme point $y_1$ of $P$ in $C$. Lemma \ref{lemnonregext} implies that there is a supporting hyperplane $H_{x'}$ of $P$ at $y_1$ such that $x'$ is in the interior of $C$. Then $x'\in F_y(P^\circ)\cap C$ and $x'$ does not belong to $W_y$, so that the condition of Theorem \ref{abelianfaces} is not satisfied and there is a face of $B$ which is not abelian.
\end{proof}

Let $K$ be a semi-simple compact connected group, let $B$ be an $\Ad$-invariant convex body in $\k$ containing $0$, such that $B\cap\t=P^\circ$ is a polytope. Let $g_B$ be the associated $\Ad$-invariant norm and endow $K$ with the corresponding Finsler length structure. 

\begin{thm}\label{geodpolytope}
The extreme points of $P$ are regular if and only if all short curves $\gamma$ in $K$ have commuting logarithmic derivatives.
\end{thm}

\begin{proof}
By Theorem \ref{abelian} the extreme points of $P$ are regular if and only if $B$ has abelian faces, and by Proposition \ref{abelianfaces1} this holds if and only if all short curves have commuting logarithmic derivatives.
\end{proof}

We now specialize Theorem \ref{geodpolytope} to compact semi-simple groups of Hamiltonian diffeomorphisms.

\begin{thm}\label{polytopecommuting}
Let $K\curvearrowright M$ be an almost effective Hamiltonian action with moment map $\mu:M\to\k^*\simeq\k$ and endow $K$ with the pullback metric of Section \ref{subsectionhoferham}. Let $\mu(M)\cap\t^+=\conv\{x_1,\dots,x_n\}$ be Kirwan's polytope given by Theorem \ref{kirwan}, and let 
$$P=\conv\{w.x-w'.x':x,x'\in \{x_1,\dots,x_n\},\quad w,w'\in\W\}$$
be the Hofer norm polytope derived from it. Then all short curves in $K$ have commuting Hamiltonians if and only if all the extreme points of $P$ are regular.
\end{thm}

\begin{proof}
By Proposition \ref{actiondiffeo} a curve $\gamma$ has Hamiltonians $\mu^{x_t}$, where $x_t=\dot{\gamma}_t\gamma_t^{-1}$ is the right logarithmic derivative. Recalling that $\{\mu^{x_s},\mu^{x_t}\}=\mu^{[x_s,x_t]}$ the theorem follows.
\end{proof}

Note that the same result Holds if we endow $K\to\Ham(M,\omega)$ with the second Hofer norm by taking the second Hofer norm polytope 
$$
P'=\conv(\{w.x:x\in \{x_1,\dots,x_n,-x_1,\dots,-x_n\},\quad w\in\W\})
$$
of Definition \ref{hoferpolytopes}, and also if we consider the one-sided Finlser Hofer norm with its polytope 
$$
P^+=\conv\{w.x:x\in \{x_1,\dots,x_n\},\,w\in\W\}.
$$

\begin{ex}
An example of a group $K$ with non-commuting Hamiltonians is $\SU(4)$ acting on the singular (co)adjoint orbit containing $x=i\diag(3,-1,-1,-1)$. The Hofer norm polytope is the convex hull of the permutations of the matrix $i\diag(4,0,0,-4)$. Its extreme points are the permutations of the matrix given by $i\diag(4,0,0,-4)$, which are all singular. We can give a short informal explanation for this which is similar to the proof of Proposition \ref{speedweylchamber}. If the maximum of $\varphi_x$ is at $x^+=i\diag(3,-1,-1,-1)$ and the minimum is at $x^-=i\diag(-1,-1,-1,3)$ then $x^+$ block diagonalizes $x$ and $x^-$ block diagonalizes $x$, so that 
$$
x=i(\lambda_1P_{\C e_1}\oplus A\oplus  \lambda_4P_{\C e_4}),
$$
with $A$ a selfadjoint  operator on $\C e_2\oplus\C e_3$ such that its spectrum $spec(A)$ satisfies $\lambda_1\geq spec(A)\geq \lambda_4$.   
\end{ex}

\subsection{Properties determined by Kirwan's polytope and product actions}

In Section \ref{ssabelianfacepolytope} we characterized Hofer norms with abelian faces using the regularity of the norm polytope (Theorem \ref{abelian}), and before that in Section \ref{subsectionspeedweyl} we studied the case of faces generating maximal abelian cones. In this section we show how conditions on Kirwan's polytope can characterize these properties, and then to finish the paper we study how these properties behave if we take products of Hamiltonian actions. 

\smallskip

We are here in the context of finite reflection groups of Section \ref{ssabelianfacepolytope}, in particular see Definition \ref{frl}. We recall from  \cite[Section 4]{eaton} or \cite[Lemma 2.9]{mozz} the following generalization of the rearrangement inequality 
$$
\sum_{i=1}^nx_{\pi(i)}y_i\leq\sum_{i=1}^nx_iy_i
$$
for $x_1\leq\dots\leq x_n$, $y_1\leq,\dots\leq y_n$ and certain permutations $\pi$.

\begin{lem}\label{rearrengementineq}
Suppose $x,y\in V$, then
$$\sup_{w\in\W}\langle x,w.y\rangle=\langle x,y\rangle$$
if and only if $x$ and $y$ belong to the same Weyl chamber. In this case
$$\langle x,w.y\rangle=\langle x,y\rangle$$
if and only if there exists $w'\in\Stab(x)$ such that $w.y=w'.y$. That is, the set of maximizers of $\varphi_x$ in $\W.y$ is exactly $\Stab(x).y$.
\end{lem}

\begin{defn}
Let $E\subseteq V$ be a convex set with $x\in E$. The normal cone to $E$ at $x$ is defined as
$$N(x,E):=\{y\in V:\langle y,z-x\rangle\leq 0 \mbox{  for all  }z\in E \}.$$
It is a closed convex cone.
\end{defn}

We will also need the following simple fact about a polytope and its normal cones.

\begin{lem}\label{normalcones}
Let $P\subseteq V$ be a convex polytope. If $A\subseteq\ext(P)$ and 
$$\bigcup_{y\in A}N(y,P)=V$$
then $A=\ext(P)$.
\end{lem}

Note that if $P$ is a convex polytope that contains $0$ in its interior, then for $y\in\ext(P)$ we have $\R_+F_y(P^\circ)=N(y,P)$.

\subsubsection{Conditions on Kirwan's polytope}

If a Hamiltonian action has the same Hofer norm polytope as the action on a \textit{regular} (co)adjoint orbit, then it has the same geodesics; therefore they are characterized by Theorem \ref{speedchamber}. We next characterize the Hofer norm polytopes which are derived from regular coadjoint orbits.

\begin{defn}
Given a Weyl chamber $\t_+$ let $-\t_+=w^*.\t_+$ be the opposite Weyl chamber (there is a  unique such $w^*\in\W$). We say that  $y\in\t_+$ is  \textit{symmetric} if $w^*.y=-y$.
\end{defn}

\begin{prop}\label{extremoshofercadjunta}
A $\W$-invariant symmetric polytope $P$ is the Hofer norm polytope of $\|\cdot\|_{\O}$ for a coadjoint orbit $\O$ if and only if $\ext(P)=\W.y$ for a symmetric $y\in\t_+$. 

A $\W$-invariant symmetric polytope $P$ is the Hofer norm polytope of $\|\cdot\|_{\O}$ with $\O$ a regular coadjoint orbit if and only if $\ext(P)=\W.y$ for a symmetric regular $y\in\t_+$.
\end{prop}

\begin{proof}
For $x\in\t_+$ let $\O_x$ be a coadjoint orbit. The Hofer norm polytope of $\|\cdot\|_{\O_x}$ is given by
$$P=\conv\{w.x-w'.x:w,w'\in\W\}.$$
If we take $y=x-w^*x\in\t_+$ we claim that $P$ has an extreme point at $y\in\t^+$ with normal cone at this point which equals $\cup_{w\in\Stab(y)}w.\t_+$. To see this note that by Lemma \ref{rearrengementineq} for $z\in\t_+$ and $w\in\W$
$$\varphi_z(x)\geq\varphi_z(w.x),$$
and
$$\varphi_z(-w^*.x)=\varphi_{-z}(w^*.x)\geq\varphi_{-z}(ww^*.x)=\varphi_{z}(-ww^*.x).$$
Therefore 
$$\varphi_z(x-w^*.x)\geq\varphi_z(w.x-w'.x),$$
for $w,w'\in\W$ and we conclude that the normal cone to $P$ at $y$ includes $\t_+$. 

Since this argument is Weyl group invariant, for $w\in\W$ the normal cone to $P$ at $w.y$ includes $w.\t_+$. In the case of singular $y$ there can be repetitions: if $\Stab(y)$ is not trivial then the normal cone to $P$ at $y$ includes $\cup_{w\in\Stab(y)}w.\t_+$.  We have $\cup_{w\in\W}w.\t_+=\t$, hence by Lemma \ref{normalcones} $\W.y$ are all the extreme points of $P$ and the normal cone to $P$ at $w'.y$ is $w'.\cup_{w\in\Stab(y)}w.\t_+$.

If $y$ is symmetric then we can take $\O_{{(1/2)}y}$ and the Hofer norm polytope of $\|\cdot\|_{\O_{{(1/2)}y}}$ satisfies $\ext(P)=\W.y$.

The proof of the second assertion is similar and we omit it.
\end{proof}

\begin{thm}\label{kirwant+}
Let $E\subseteq \k$ be an $\Ad$-invariant set such that $E\cap\t_+=\conv\{x_1,\dots,x_n\}$. The Hofer norm polytope derived from $E$ has extreme points $\W.y$ for regular $y$ if there is an $x_i$, say $x_1$, such that
\begin{itemize}
\item The point $x_1-w^*.x_1\in\t_+$ is regular.
\item For $x\in\t_+$ we have $\varphi_x(x_1)\geq\varphi_x(x_j)$ for $j\in\{2,\dots,n\}$, that is, $\t_+$ is contained in the normal cone of $\conv\{x_1,\dots,x_n\}\subseteq\t$ at $x_1$.
\end{itemize}
\end{thm}

\begin{proof}
The Hofer norm polytope is given by 
$$P=\conv\{w.x-w'.x':x,x'\in \{x_1,\dots,x_n\},\quad w,w'\in\W\}.$$
We are going to prove that $\ext(P)=\W.(x_1-w^*.x_1)$. Since $x_1-w^*.x_1$ is regular $\W.(x_1-w^*.x_1)$ has $|\W|$ points.  We claim that $P$ has an extreme point at $x_1-w^*.x_1\in\t^+$ with a normal cone at this point which equals $\t_+$. To see this note that by the second assumption in the statement of the theorem and by Lemma \ref{rearrengementineq} for $x\in\t_+$, $j\in\{1,\dots,n\}$ and $w\in\W$
$$\varphi_x(x_1)\geq\varphi_x(x_j)\geq\varphi_x(w.x_j),$$
and
$$\varphi_x(-w^*.x_1)=\varphi_{-x}(w^*.x_1)\geq\varphi_{-x}(w^*.x_j)\geq\varphi_{-x}(ww^*.x_j)=\varphi_{x}(-ww^*.x_j).$$
Therefore 
$$\varphi_x(x_1-w^*.x_1)\geq\varphi_x(w.x-w'.x'),$$
for $x,x'\in \{x_1,\dots,x_n\}$ and $w,w'\in\W$ and we conclude that the normal cone to $P$ at $x_1-w^*.x_1$ includes $\t_+$. Since this argument is Weyl group invariant, for $w\in\W$ the normal cone to $P$ at $w.(x_1-w^*.x_1)$ includes $w.\t_+$. We have $\cup_{w\in\W}w.\t_+=\t$, hence by Lemma \ref{normalcones} the orbit $\W.(x_1-w^*.x_1)$ are all the extreme points of $P$ and the normal cone to $P$ at $w.(x_1-w^*.x_1)$ is $w.\t_+$.
\end{proof}

\begin{ex}
An example of this is the action of $\SU(3)$ on the singular (co)adjoint orbit containing $\lambda=i\diag(2,-1,-1)$. The Hofer norm polytope is the convex hull of $0$ and the permutations of $i\diag(3,0,-3)$. Its extreme points are the permutations of $i\diag(3,0,-3)$. This is the same Hofer norm polytope as the one derived from the (co)adjoint action on the regular (co)adjoint orbit containing $\frac{i}{2}\diag(-3,0,3)$. Several other examples can be computed from the cases of $\SU(3)$ acting on products of $\mathbb{P}^2$ studied in \cite{ms}. In fact, it is easy to see from all the figures in \cite{ms} that all the Kirwan polytopes listed there lead by Theorem \ref{kirwant+} to Hofer norm polytopes with extreme points $\W.y$ for regular $y$.  
\end{ex}

\begin{cor}\label{cororeg}
Let $\O_x$ for $x\in\t_+$ be a coadjoint orbit. The Hofer norm polytope derived from $\O$ has extreme points $\W.y$ for regular $y$ if and only if $x-w^*.x\in\t_+$ is regular.
\end{cor}

\begin{ex}
An example of the previous corollary is the case of (co)adjoint orbits of $\SU(n)$. Let $y$ be an extreme point of $P$. The set of $x$ such that $\varphi_x$ has a unique maximum at $y$ is an open cone so we can assume that all eigenvalues of $x$ are different, and without loss of generality we assume that they are ordered increasingly: $x_1<\dots <x_n$. Hence if the eigenvalues of the (co)adjoint orbit are $\lambda_1\le \dots\le \lambda_n$ the extreme point of the Hofer norm polytope which maximizes $\varphi_x$ is 
$$
i\diag(\lambda_1,\dots,\lambda_n)-i\diag(\lambda_n,\dots,\lambda_1)=i\diag(\lambda_1-\lambda_n,\lambda_2-\lambda_{n-1},\dots,\lambda_n-\lambda_1).
$$
This is the same maximum as the one that would be obtained from looking at the Hofer norm polytope derived from the (co)adjoint orbit with eigenvalues $x_1=\frac{i}{2}(\lambda_1-\lambda_n)<\dots<x_n=\frac{i}{2}(\lambda_n-\lambda_1)$. If these eigenvalues are all distinct then the Hofer norm polytope is equal to the Hofer norm polytope of a regular (co)adjoint orbit.
\end{ex}

\begin{rem}
For the one-sided Hofer norm polytope the condition for being derived from a regular coadjoint orbit, is that there exists a regular $x_i\in P^+$, say $x_1$ such that for $x\in\t_+$ we have $\varphi_x(x_1)\geq\varphi_x(x_j)$ for $j\in\{2,\dots,n\}$, where
$$
P^+=\conv\{w.x:x\in \{x_1,\dots,x_n\},\,w\in\W\}.
$$
\end{rem}

The extreme points of the Hofer norm polytope are Weyl group invariant, so we can partition them into orbits
$$\ext(P)=\W.y_1\sqcup\dots\sqcup\W.y_m$$
for $y_1,\dots,y_m\in\t_+$. If $y_1,\dots,y_m$ are regular then by Theorem \ref{abelian} the unit ball $B$ has abelian faces. We next give a sharper characterization of these faces.

\begin{prop}
If $\ext(P)=\W.y_1\sqcup\dots\sqcup\W.y_m$ for regular $y_1,\dots,y_m\in\t_+$, then for $i\in\{1,\dots,m\}$
\begin{align*}
F_{y_i}(B)&=\{x\in\t_+:\varphi_{x}(y_i)=1\mbox{  and  }\varphi_{x}(y_i)\leq 1\mbox{  for  }i\neq j\}.
\end{align*}
\end{prop}

\begin{proof}
From Theorem \ref{abelianfaces} we know that if the faces of $B$ are abelian then 
$$F_x(B)=F_x(B\cap\t)\cap\t_+$$
for a Weyl chamber $\t_+$ given by a choice of torus and positive simple roots such that $x\in\t_+$. We take for simplicity $y_i=y_1$, note that
\begin{align*}
F_{y_1}(B)&=F_{y_1}(B\cap\t)\cap\t_+=F_{y_1}(P^\circ)\cap\t_+\\
&=\{x\in\bd P^\circ:H_{y_1}\mbox{  supports  }P^\circ\mbox{  at  }x\}\\
&=\{x\in\bd P^\circ:H_{x}\mbox{  supports  }P\mbox{  at  }y_1\}\\
&=\{x\in\t_+:\varphi_{x}(y_1)=1\mbox{  and  }\varphi_{x}(y_j)\leq 1\mbox{  for  }j=2,\dots,n\}.
\end{align*}
Note that $H_{x}$ supports $P$ at $y_1$ is equivalent to $\varphi_{x}(y_1)=1$ and  $\varphi_{x}(w.y_j)\leq 1$ for $w\in\W$ and $j=1,\dots,n$. This in turn is equivalent by Lemma \ref{rearrengementineq} to $\varphi_{x}(y_1)=1$ and  $\varphi_{x}(y_j)\leq 1$ for $w\in\W$ and $j=2,\dots,n$ since $x\in\t_+$. This establishes the last equality. The third and fourth equalities follow from previous results on polar duality.
\end{proof}

\begin{rem}
We now have an alternative proof of Proposition \ref{speedweylchamber}: the extreme points of the Hofer norm polytope are $\W.y$ for a regular $y\in\t_+$, hence $F_{y}(B)=\{x\in\t_+:\varphi_{x}(y)=1\}$. The cone generated by this face is $\t_+$.
\end{rem}

Next we obtain conditions in terms of Kirwan's polytope which imply that $y_1,\dots,y_m\in\t_+$ are regular. This follows from the previous discussion therefore we omit the proof.

\begin{thm}
Let $E\subseteq \k$ be an $\Ad$-invariant set such that $E\cap\t_+=\conv\{x_1,\dots,x_n\}$. The Hofer norm polytope derived from $E$ has regular extreme points if  the extreme points $y$ of 
$$A=\conv\{x_i-w^*.x_j:i,j=1,\dots,n\}$$
such that the normal cone $N(y,A)$ intersects the interior of $\t_+$ are all regular.
\end{thm}

\subsubsection{Products of actions}

We now study products of Hamiltonian actions where the image of the moment map is given by (\ref{sumaimagmoment}). A property of regularity on one of the factors implies the same property in the product, in the following two cases:

\medskip

Consider the case of the Hofer norm $\|\cdot\|_{\O_1+\dots+\O_n}$ for coadjoint orbits $\O_1,\dots,\O_n$. This norm arises by (\ref{sumaimagmoment}) from the canonical symplectic action of $K$ on $\O_1\times\dots\times\O_n$.

\begin{prop}\label{sumaextremounico}
Let $P_1,\dots,P_n$ be $\W$-invariant convex polytopes in $\t$ such that $\ext(P_i)=\W.y_i$ for $y_i\in\t_+$ and $i=1,\dots,n$. Then $\ext(P_1+\dots+P_n)=\W.(y_1+\dots+y_n)$. Hence, if  there is a $y_i$ that is regular, then the extreme points of $P_1+\dots+P_n$ are the Weyl group orbit of a regular element.
\end{prop}

\begin{proof}
Let $x$ be a point in the interior of $\t_+$. Then, by Lemma \ref{rearrengementineq} $\varphi_x(y_i)>\varphi_x(w.y_i)$ for $i=1,\dots,n$ and $w\in\W\setminus\Stab(y_i)$. Hence $\varphi_x$ has a unique maximum in $P_i$ at $y_i$. This implies that $\varphi_x$ has a unique maximum in $P_1+\dots+P_n$ at $y_1+\dots+y_n$, i.e. $y_1+\dots+y_n$ is an extreme point of $P_1+\dots+P_n$ and $x\in N(P_1+\dots+P_n,y_1+\dots+y_n)$. Since the normal cones are closed $\t_+\subseteq N(P_1+\dots+P_n,y_1+\dots+y_n)$ and since this argument is Weyl group invariant 
$$w.\t_+\subseteq N(P_1+\dots+P_n,w.(y_1+\dots+y_n))$$
for $w\in\W$. By Lemma \ref{normalcones} all the extreme points are $\W.(y_1+\dots+y_n)$. If there is a $j\in\{1,\dots,n\}$ such that $y_j$ is regular, then this point is in the interior of the cone $\t_+$ so that $y_1+\dots+y_n$ is also in the interior, i.e. is regular.
\end{proof}

From Proposition \ref{extremoshofercadjunta} we get the following

\begin{cor}\label{speedweylchamber2}
Let $\O_1,\dots,\O_n$ be (co)adjoint orbits and let $\|\cdot\|_{\O_1+\dots+\O_n}$ be the $\Ad$-invariant Hofer norm defined by $\O_1+\dots+\O_n$. If at least one coadjoint orbit is regular, then the Hofer norm polytope derived from $\|\cdot\|_{\O_1+\dots+\O_n}$ has extreme points equal to $\W.y$ for a symmetric regular $y\in\t_+$, i.e. it is the Hofer norm polytope derived from a regular coadjoint orbit.
\end{cor}

\begin{rem}
By Corollary \ref{arghofer} a cone $C\subseteq\k$ generated by a face has the same norming functional, so the set $\{\varphi_x:x\in C\}$ has a common maximizer $x^+\in \O_1+\dots+\O_n$ and a common minimizer $x^-\in\O_1+\dots+\O_n$. Let us write $x^-=x^-_1+\dots +x^-_n$ and $x^+=x^+_1+\dots +x^+_n$ with $x^-_i,x^+_i\in \O_i$ for $i=1,\dots,n$. Then by Proposition \ref{interconosmaxmin} 
$$C_{x^-,x^+}(\O_1+\dots+\O_n)=\bigcap_{i=1,\dots,n}C_{x^-_i,x^+_i}(\O_i).$$
If $\O_1$ is regular then by Theorem \ref{speedweylchamber} $C_{x^-_1,x^+_1}(\O_1)$ is contained in a Weyl chamber (given by a choice of torus and positive simple roots). 

Conversely, if $\t_+$ is a Weyl chamber (given by a choice of torus and positive simple roots) we choose $x^+_i$ to be the intersection of $\O_i$ with $\t_+$ and $x^-_i$ to be the intersection of $\O_i$ with $-\t_+$. By Proposition \ref{maximumorbit} the Weyl chamber $\t_+$ is contained in $C_{x^-_i,x^+_i}(\O_i)$ for $i=2,\dots,n$ and by Proposition \ref{speedweylchamber} $C_{x^-_1,x^+_1}(\O_1)=\t_+$. Hence if we write $x^-=x^-_1+\dots +x^-_n$ and $x^+=x^+_1+\dots +x^+_n$ we get by Proposition \ref{interconosmaxmin} 
$$C_{x^-,x^+}(\O_1+\dots+\O_n)=\bigcap_{i=1,\dots,n}C_{x^-_i,x^+_i}(\O_i)=\t_+,$$
which is a set with the same norming functionals. We conclude again that the Weyl chambers (given by a choice of torus and positive simple roots) are the cones generated by maximal faces.
\end{rem}

We now turn to $\W$-invariant polytopes such that its extreme points are more than one $\W$-orbit.

\begin{prop}\label{sumaextremoregular}
Let $P_1,\dots,P_n$ be $\W$-invariant convex polytopes in $\t$. If the extreme points of one of the polytopes are all regular then the extreme points of $P_1+\dots+P_n$ are all regular.
\end{prop}

\begin{proof}
Let $y$ be an extreme point of $P_1+\dots+P_n$. There exists an $x\in \t$ such that $\varphi_x$ attains its unique maximum in $P_1+\dots+P_n$ at $y$. Since the normal cone to the polytope $P_1+\dots+P_n$ at $y$ is open we can chose a regular $x$, therefore $\Stab(x)$ is trivial. This regular $x$ is in a unique Weyl chamber which we denote by $\t_+$. We have $y=y_1+\dots +y_n$, where $y_1,\dots,y_n$ are points where $\varphi_x$ attains its unique maximums in $P_1,\dots,P_n$ respectively. Since for $i=1,\dots,n$ we have $\W.y_i\subseteq P_i$ and $\varphi_x$ attains a maximum at $y_i$ Lemma \ref{rearrengementineq} implies that $y_1,\dots,y_n$ are in $\t_+$. Since one of the $y_i$ is in the interior of $\t_+$ so is their sum $y=y_1+\dots +y_n$, hence $y$ is regular. 
\end{proof}

To finish this paper, from Proposition \ref{sumaextremoregular} and Theorem \ref{polytopecommuting} we obtain the following structural property of geodesics (with different metrics) in the group $K$:

\begin{thm}\label{sumahoferpolyregular}
Let Let $K\curvearrowright M_1,\dots,K\curvearrowright M_n$ be Hamiltonian almost effective actions of a compact semi-simple group $K$ and let $K\curvearrowright M_1\times\dots\times M_n$ be the product action. Endow $K$ with the pullback metrics of Section \ref{subsectionhoferham}. Assume that each short curve in $K$ with the metric derived from the action on one factor has commuting Hamiltonians: then all short curves in $K$ with the metric derived from the product action have commuting Hamiltonians.
\end{thm}

\begin{rem}
Another proof of the previous theorem can be given as follows. If the extreme point of one of the Hofer norm polytopes, say $P_1$, are regular, then all the maximal faces of the sphere of the norm $\|\cdot\|_{\mu_1(M)}$ are abelian. The cones of the form $C_{x^-_1,x^+_1}(\mu_1(M))$ for $x^-_1,x^+_1\in \mu_1(M)$ are contained in cones generated by faces and are therefore abelian. The cone generated by faces of the sphere of the norm $\|\cdot\|_{\mu(M)}$ are contained in $C_{x^-,x^+}(\mu(M))$ for $x^-,x^+\in \mu(M)$. Also 
$$C_{x^-,x^+}(\mu(M))=\bigcap_{i=1,\dots,n}C_{x^-_i,x^+_i}(\mu_i(M))$$
for $x^-_i,x^+_i\in \mu_i(M)$ and $i=1,\dots,n$ by Proposition \ref{interconosmaxmin}. Since the first set of the intersection is abelian the result follows. 
\end{rem}

\section*{Acknowledgments}

This research was supported by Instituto Argentino de Matem\'atica (CONICET) and Universidad de Buenos Aires.

\noindent
\end{document}